%% file: sliced_plans_main.tex
\documentclass[dvipsnames, a4paper, 11pt]{article}

\input{fancy_theorem.tex}
\input{theme.tex}

\input{shortcuts.tex}

\input{impbot/impbot_requirements.tex}

\title{Sliced Transport Plans}
\author[1]{Eloi Tanguy}
\author[2]{Laetitia Chapel}
\author[1]{Julie Delon}
\affil[1]{Universit\'e Paris Cit\'e, CNRS, MAP5, F-75006 Paris, France}
\affil[2]{Institut Agro Rennes-Angers, IRISA}
\date{2nd August 2025}
  
\begin{document}
\maketitle

\begin{abstract}
	\input{sections/abs}

\end{abstract}

\tableofcontents
\newpage

\input{sections/intro.tex}
\input{sections/nu_based_wass.tex}
\input{sections/pivot_sliced.tex}
\input{sections/equality_CWtheta.tex}
\input{sections/pivot_sliced_discrete.tex}
\input{sections/min_pivot_sliced.tex}
\input{sections/expected_sliced.tex}
\input{sections/numerics.tex}
\input{sections/ackn.tex}

\clearpage
\printbibliography

\appendix
\include{sections/appendix}
\subsection{Result Dependency Graph}
\input{impbot/implication_graph.tex}
\end{document}

%% file: fancy_theorem.tex
\usepackage{amsthm}
\usepackage[most]{tcolorbox}

\newtheoremstyle{upright}   %
  {10pt}                    %
  {10pt}                    %
  {}                        %
  {}                        %
  {\bfseries}               %
  {.}                       %
  { }                       %
  {}                        %

\theoremstyle{upright}

\tcbuselibrary{theorems}
\definecolor{theoremblue}{HTML}{2196F3} %
\definecolor{proofpurple}{HTML}{9013FE} %
\definecolor{defyellow}{HTML}{FFEB3B} %
\definecolor{assumptionred}{HTML}{ff3b3b} %
\definecolor{conventionred}{HTML}{ff3b3b} %
\definecolor{remarkorange}{HTML}{dc5800} %
\definecolor{examplegreen}{HTML}{00695c} %

\tcolorboxenvironment{theorem}{
  enhanced jigsaw,
  sharp corners,
  colframe=theoremblue,
  colback=theoremblue!5,
  coltitle=black,
  fonttitle=\bfseries,
  borderline west={2pt}{0pt}{theoremblue},
  before skip=10pt,
  after skip=10pt,
  boxrule=0pt,
  left=10pt,
  right=10pt,
  breakable
}

\tcolorboxenvironment{definition}{
  enhanced jigsaw,
  sharp corners,
  colframe=defyellow,
  colback=defyellow!5,
  coltitle=black,
  fonttitle=\bfseries,
  borderline west={2pt}{0pt}{defyellow},
  before skip=10pt,
  after skip=10pt,
  boxrule=0pt,
  left=10pt,
  right=10pt,
  breakable
}

\tcolorboxenvironment{corollary}{
  enhanced jigsaw,
  sharp corners,
  colframe=theoremblue,
  colback=theoremblue!5,
  coltitle=black,
  fonttitle=\bfseries,
  borderline west={2pt}{0pt}{theoremblue},
  before skip=10pt,
  after skip=10pt,
  boxrule=0pt,
  left=10pt,
  right=10pt,
  breakable
}

\tcolorboxenvironment{prop}{
  enhanced jigsaw,
  sharp corners,
  colframe=theoremblue,
  colback=theoremblue!5,
  coltitle=black,
  fonttitle=\bfseries,
  borderline west={2pt}{0pt}{theoremblue},
  before skip=10pt,
  after skip=10pt,
  boxrule=0pt,
  left=10pt,
  right=10pt,
  breakable
}

\tcolorboxenvironment{condition}{
  enhanced jigsaw,
  sharp corners,
  colframe=theoremblue,
  colback=theoremblue!5,
  coltitle=black,
  fonttitle=\bfseries,
  borderline west={2pt}{0pt}{theoremblue},
  before skip=10pt,
  after skip=10pt,
  boxrule=0pt,
  left=10pt,
  right=10pt,
  breakable
}

\tcolorboxenvironment{example}{
  enhanced jigsaw,
  sharp corners,
  colframe=examplegreen,
  colback=examplegreen!5,
  coltitle=black,
  fonttitle=\bfseries,
  borderline west={2pt}{0pt}{examplegreen},
  before skip=10pt,
  after skip=10pt,
  boxrule=0pt,
  left=10pt,
  right=10pt,
  breakable
}

\tcolorboxenvironment{lemma}{
  enhanced jigsaw,
  sharp corners,
  colframe=theoremblue,
  colback=theoremblue!5,
  coltitle=black,
  fonttitle=\bfseries,
  borderline west={2pt}{0pt}{theoremblue},
  before skip=10pt,
  after skip=10pt,
  boxrule=0pt,
  left=10pt,
  right=10pt,
  breakable
}

\tcolorboxenvironment{assumption}{
  enhanced jigsaw,
  sharp corners,
  colframe=assumptionred,
  colback=assumptionred!5,
  coltitle=black,
  fonttitle=\bfseries,
  borderline west={2pt}{0pt}{assumptionred},
  before skip=10pt,
  after skip=10pt,
  boxrule=0pt,
  left=10pt,
  right=10pt,
  breakable
}

\tcolorboxenvironment{remark}{
  enhanced jigsaw,
  sharp corners,
  colframe=remarkorange,
  colback=remarkorange!5,
  coltitle=black,
  fonttitle=\bfseries,
  borderline west={2pt}{0pt}{remarkorange},
  before skip=10pt,
  after skip=10pt,
  boxrule=0pt,
  left=10pt,
  right=10pt,
  breakable
}

\tcolorboxenvironment{convention}{
  enhanced jigsaw,
  sharp corners,
  colframe=conventionred,
  colback=conventionred!5,
  coltitle=black,
  fonttitle=\bfseries,
  borderline west={2pt}{0pt}{conventionred},
  before skip=10pt,
  after skip=10pt,
  boxrule=0pt,
  left=10pt,
  right=10pt,
  breakable
}

\newtcolorbox{myproof}{
  enhanced jigsaw,
  sharp corners,
  colframe=proofpurple,
  colback=proofpurple!0,
  borderline west={1pt}{0pt}{proofpurple},
  borderline south={1pt}{0pt}{proofpurple},
  before skip=10pt,
  after skip=10pt,
  boxrule=0pt,
  left=10pt,
  right=10pt,
  breakable,
}

\renewenvironment{proof}[1][\proofname]{
  \begin{myproof}
    \textbf{#1.}
}{\hfill$\qed$\end{myproof}
}

%% file: theme.tex
\usepackage{hyperref}
\usepackage[utf8]{inputenc}
\usepackage[T1]{fontenc}
\usepackage{amsmath}
\usepackage{amsthm}
\usepackage{amsfonts}
\usepackage{amssymb}
\usepackage{graphicx,caption,subcaption}
\usepackage{fancybox}
\usepackage{enumitem}
\usepackage{soul}
\usepackage{mathtools}
\usepackage{lmodern}
\usepackage{stmaryrd}
\usepackage{multirow}
\usepackage{multicol}
\usepackage{xcolor}
\usepackage{tabularx}
\usepackage{array}
\usepackage{colortbl}
\usepackage{varwidth}
\PassOptionsToPackage{dvipsnames,svgnames}{xcolor}     
\usepackage{wrapfig}
\usepackage[explicit,calcwidth]{titlesec}
\usepackage{fancyhdr}
\usepackage[left=2.1cm, right=2.1cm, top=2cm, bottom=2cm]{geometry}
\usepackage{esint}
\usepackage{tocloft}
\usepackage[skip=8pt plus1pt, indent=0pt]{parskip}
\usepackage{twoopt}
\usepackage{authblk}
\usepackage{accents}
\hypersetup{ colorlinks=true, %
    linktoc=all,     %
    linkcolor=blue,  %
}
\usepackage{mathrsfs}
\PassOptionsToPackage{unicode}{hyperref}
\PassOptionsToPackage{naturalnames}{hyperref}
\usepackage{float}
\usepackage{placeins}
\usepackage[
    backend=biber,
    uniquename=false,
    bibencoding=utf8,
    sorting=nyt,
    doi=false, isbn=false, url=false,
    style=alphabetic,
    maxcitenames=2,
    uniquelist=false,
    maxbibnames=123,
    backref=false,
    hyperref=true
]{biblatex}
\usepackage[ruled,vlined,linesnumbered]{algorithm2e}
\labelformat{algocf}{Algorithm #1}

\usepackage{todonotes}
\setuptodonotes{fancyline, color=red!20, shadow, inline}

\usepackage{silence}

\usepackage[bbgreekl]{mathbbol}
\DeclareSymbolFont{bbold}{U}{bbold}{m}{n}
\DeclareSymbolFontAlphabet{\mathbbold}{bbold}
\DeclareSymbolFontAlphabet{\mathbb}{AMSb}%

\usepackage[nameinlink, capitalise]{cleveref}

\newtheorem{theorem}{Theorem}%
\newtheorem*{theorem*}{Theorem}
\newtheorem{corollary}{Corollary}%
\newtheorem{lemma}{Lemma}%

\newtheorem{remark}{Remark}%
\newtheorem{prop}{Proposition}%
\newtheorem{assumption}{Assumption}
\newtheorem{definition}{Definition}
\newtheorem{example}{Example}

%% file: shortcuts.tex
\newcommand{\N}[0]{\mathbb{N}}

\newcommand{\R}[0]{\mathbb{R}}
\newcommand{\C}[0]{\mathbb{C}}
\newcommand{\K}[0]{\mathcal{K}}

\renewcommand{\L}[0]{\mathbb{L}}
\newcommand{\U}[0]{\mathbb{U}}
\renewcommand{\P}[0]{\mathbb{P}}
\newcommand{\M}{\mathrm{M}}

\newcommand{\Int}[2]{\displaystyle\int_{#1}^{#2}}

\newcommand{\Sum}[2]{\displaystyle\sum\limits_{#1}^{#2}}

\newcommand{\V}[1]{\mathbb{V}\left[#1\right]}

\newcommand{\dd}[0]{\mathrm{d}}

\renewcommand{\epsilon}{\varepsilon}

\newcommand{\nt}[1]{{\left\vert\kern-0.25ex\left\vert\kern-0.25ex\left\vert #1
    \right\vert\kern-0.25ex\right\vert\kern-0.25ex\right\vert}}

\DeclareMathOperator{\argmin}{argmin}

\DeclareMathOperator{\supp}{supp}
\newcommand{\W}[0]{\mathrm{W}}
\renewcommand{\SS}{\mathbb{S}}

\newcommand{\app}[4]{\left\lbrace\begin{array}{ccc}
   #1 & \longrightarrow & #2 \\
   #3 & \longmapsto & #4 \\
\end{array} \right.}

\newcommand{\Perm}[0]{\mathfrak{S}}
\newcommand{\oll}[1]{\overline{#1}}

\renewcommand{\O}{\mathcal{O}}
\newcommand\eqquestion{\stackrel{\mathclap{\normalfont\mbox{?}}}{=}}

\newcommand{\step}[2]{\textrm{---} \textit{Step #1}: #2}

\newcommand{\SW}{\mathrm{SW}}
\newcommand{\SWGG}{\mathrm{SWGG}}
\newcommand{\SWGGfix}{\widetilde{\mathrm{SWGG}}}

\newcommand{\Leb}{\mathscr{L}}
\newcommand{\X}{\mathcal{X}}
\newcommand{\Y}{\mathcal{Y}}

\newcommand{\LStheta}{\mathrm{LS}_\theta}
\newcommand{\LSthetaL}{\mathrm{LS}_{\theta_\ell}}
\newcommand{\ES}{\mathrm{ES}}
\newcommand{\Law}{\mathrm{Law}}
\newcommand{\half}{{\text{\tiny$\tfrac{1}{2}$}}}
\newcommand{\minS}{\min \mathrm{PS}}
\newcommand{\PStheta}[1][\theta]{\mathrm{PS}_{#1}}

\newcommand{\CWtheta}[1][\theta]{\mathrm{CW}_{#1}}
\newcommand{\PthetaXY}{\mathcal{P}_\theta(X, Y)}
\newcommand{\CPerm}[1]{\mathcal{P}^{#1}}
\DeclareMathOperator{\Extr}{Extr}

%% file: sections/abs.tex
Since the introduction of the Sliced Wasserstein distance in the literature, its simplicity and efficiency have made it one of the most interesting surrogates for the Wasserstein distance in image processing and machine learning. 
However, its inability to produce transport plans limits its practical use to applications where only a distance is necessary. Several heuristics have been proposed in recent years to address this limitation when the probability measures are discrete. In this paper, we propose to study these different propositions by redefining and analysing them rigorously for generic probability
measures. 
Leveraging the $\nu$-based Wasserstein distance and generalised geodesics, we introduce and study the Pivot Sliced Discrepancy, inspired by a recent work by Mahey et al.. We demonstrate its semi-metric properties and its relation to a constrained Kantorovich formulation.
In the same way, we generalise and study the recent Expected Sliced plans introduced by Liu et al. for completely generic measures. Our theoretical contributions are supported by numerical experiments on synthetic and real datasets, including colour transfer and shape registration, evaluating the practical relevance of these different solutions.

%% file: sections/intro.tex
\section{Introduction}

Known for its ability to capture geometric structure in probability
distributions, optimal transport has attracted considerable attention in both
theoretical and applied fields. Several studies have developed its mathematical
foundations in great detail~\cite{santambrogio2015optimal,villani}, and its
practical impact has been demonstrated across a broad spectrum of applications.
Originally developed for applications in logistics,
economics~\cite{galichon2017survey} and fluid mechanics, computational optimal
transport has also emerged in the last fifteen years as a central tool in data
science. It is used nowadays for a large variety of applications, ranging from
image processing, computer vision and computer
graphics~\cite{rabin2009statistical,hertrich2022wasserstein,feydy2017optimal,bonneel2023survey,pont2021wasserstein},
to domain
adaptation~\cite{courty2016optimal,montesuma2021wasserstein,Fatras2021}, natural
language processing~\cite{chenplot}, generative
modelling~\cite{arjovsky2017wasserstein,gulrajani2017improved,salimans2018improving,tong2024improving,hertrich2025relation},
quantum chemistry~\cite{buttazzo2012optimal} or
biology~\cite{bunne2024optimal,naderializadeh2025aggregating}, to cite just a
few. 

In these applications, optimal transport is  used to define meaningful
discrepancies between probability distributions, taking into account the
underlying geometry of the data, but also as a way to define optimal plans or
maps between such data, in order to transform a given distribution into another
in an optimal way. In the continuous setting, we recall that the 2-Wasserstein
distance $\W_2$ between two probability measures $\mu$ and $\nu$ on
$\mathbb{R}^d$ is defined as:
$$
\W_2^2(\mu, \nu) = \inf_{\pi \in \Pi(\mu, \nu)} 
\int_{\mathbb{R}^d \times \mathbb{R}^d} \|x - y\|_2^2 \, \dd\pi(x, y),
$$
where $\Pi(\mu, \nu)$ is the set of couplings with marginals $\mu$ and $\nu$. In
the discrete case, with empirical measures supported on finite point clouds,
this problem becomes a linear program over a polytope. Computing Wasserstein
distances between discrete datasets comes with significant computational
expense. Classical linear programming solvers used to evaluate the transport
cost between two discrete measures of size $n$ typically have a complexity of
$\O(n^3 \log n)$~\cite{computational_ot}. This limitation has motivated the
development of computationally lighter surrogates or approximations that
preserve key characteristics of optimal transport metrics.

One of these popular and efficient surrogates is the Sliced Wasserstein distance
($\mathrm{SW}$)~\cite{rabin2012wasserstein,bonneel2015sliced}. This approach
leverages the fact that in one dimension, the Wasserstein distance has a
closed-form solution. The Sliced Wasserstein distance is derived by averaging 1D
Wasserstein distances over all directions on the unit sphere, offering a simple
alternative to $\W_2$: 
$$
\mathrm{SW}_2^2(\mu, \nu) = \int_{\mathbb{S}^{d-1}} \W_2^2(P_{\theta}\#\mu, P_{\theta}\#\nu) \, \dd\theta,
$$
where $P_\theta$ denotes the projection onto direction $\theta$. Since
evaluating the full integral is intractable in practice, it is approximated by
Monte Carlo sampling. One draws $L$ random directions, computes the 1D
Wasserstein distance for each, and averages the results. The 1D Wasserstein
distance between empirical distributions of $n$ points can be obtained in $\O(n
\log n)$, so the approximate $\mathrm{SW}_2$ distance can be computed in $\O(L n
\log n)$. This efficiency makes it especially appealing for large values of $n$.

The SW distance remains a true distance in the space of probability measures and
retains several fundamental features of Wasserstein distances. For probability
measures with compact support, it has been shown to be equivalent to the
Wasserstein distance~\cite{bonnotte}. It also has desirable statistical
properties, such as sample complexity bounds and
robustness~\cite{nadjahi_statistical_properties_sliced}. Its efficiency has been
confirmed in numerous use cases, including domain
adaptation~\cite{lee2019sliced}, texture generation, colour and style
transfer~\cite{heitz2021sliced,bonneel2015sliced,elnekave2022generating},
statistical inference~\cite{kolouri2018CVPR}, generative
modelling~\cite{deshpande2018generative,wu2019sliced,chapel2025differentiable},
auto-encoder regularisation~\cite{kolouri2018sliced}, topological data
analysis~\cite{sisouk2025user} or shape analysis~\cite{le2024integrating,
nguyen2023self}. Extensions to Riemannian settings have also been
investigated~\cite{bonet2024sliced}. Nevertheless, a key limitation of SW is
that it does not provide a transport plan or a map between distributions, which
limits its use in applications that require  correspondences between datasets.

To circumvent this issue, several heuristics have been proposed to extract
approximate transport plans from SW. A notable example is the use of stochastic
gradient descent (SGD) to minimise the objective $X \longmapsto
\mathrm{SW}(\delta_X, \delta_Y)$, as a way to gradually move points from a
source point cloud $X$ to a target point cloud\footnote{For a point cloud $X =
(x_i)_{i=1}^N$, we write $\delta_X = \frac 1 N \sum_{i=1}^N\delta_{x_i}$} $Y$.
This strategy has been first explored for colour transfer and image matching
tasks in~\cite{Rabin_texture_mixing_sw,bonneel2015sliced}, and can provide
plausible pointwise correspondences in practice, although theoretical guarantees
remain partial~\cite{tanguy2023discrete_sw_losses,cozzi2025long,li2025measure}.

More recently, two alternative strategies have been introduced to build
transport plans grounded in Sliced Wasserstein distances. The first one, called
Sliced Wasserstein Generalised Geodesics
(SWGG)~\cite{mahey23fast,chapel2025differentiable}, defines a map between two
discrete distributions $\delta_X = \tfrac{1}{n}\sum_i \delta_{x_i}$ and
$\delta_Y = \tfrac{1}{n}\sum_i \delta_{y_i}$ as $\tau_\theta \circ
\sigma_\theta^{-1}$, where $\sigma_\theta$ is a permutation which sorts
$(\theta^\top x_i)_{i=1}^n$ and $\tau_\theta$ a permutation sorting
$(\theta^\top y_i)_{i=1}^n$. The Sliced Wasserstein Generalised Geodesic
distance (\cite[Equation 8]{mahey23fast}) is then defined as: (see also
\cref{fig:ex_swgg})
\begin{equation}\label{eqn:swgg_discrete}
    \SWGG_2^2(\mu_1, \mu_2, \theta) := \cfrac{1}{n}\Sum{i=1}{n}\|x_{\sigma_\theta(i)} - y_{\tau_\theta(i)}\|_2^2.
\end{equation}
The second one, called Expected Sliced Transport Plans, was introduced
in~\cite{liu2024expected} (inspired by~\cite{rowland2019orthogonal}), also for
discrete measures. It aims to construct couplings by averaging the 1D optimal
transport plans obtained from projections. Given $\bbsigma$ a probability
measure on the hypersphere, with the same notations as above, the Expected
Sliced Transport distance is defined as:
\begin{equation}\label{eqn:expected_discrete}
    \mathbb{E}_{\theta \sim \bbsigma } \left[\cfrac{1}{n}\Sum{i=1}{n}\|x_{\sigma_\theta(i)} - y_{\tau_\theta(i)}\|_2^2\right]
\end{equation}
and the average transport plan as $\mathbb{E}_{\theta \sim \bbsigma
}[\tau_\theta \circ \sigma_\theta^{-1}]$. This yields a plan between the two
$d$-dimensional measures that reflects the averaged behaviour along slices.

These approaches provide practical and interpretable ways to define approximate
transport maps. However, they are currently defined only for discrete measures
and lack a rigorous theoretical grounding in more general measure spaces.
Moreover, even in the discrete setting, it can easily be shown that the RHS
quantity in \cref{eqn:swgg_discrete} depends on the choice of the permutations,
rendering the quantity ill-defined, as showcased in
\cref{sec:pathological_cases_swgg}. 

The goal of this paper is to rigorously define and analyse these different
Sliced Optimal Transport Plans for completely generic probability measures. We
introduce the Pivot Sliced Discrepancy $\PStheta$, a  discrepancy measure based
on the $\nu$-based generalised geodesics~\cite{nenna2023transport}, and
generalising the Sliced Wasserstein Generalised Geodesic
distance~\cite{mahey23fast}. In doing so, we also provide new theoretical
insights on the $\nu$-based Wasserstein distance~\cite{nenna2023transport}. We
prove that $\PStheta$ is well-defined, symmetric and separates points. We then
establish an equivalence between $\PStheta$ and a constrained version of the
Wasserstein distance, showing that $\PStheta$ coincides with the minimal
transport cost among plans that preserve the projected coupling. For empirical
measures, we provide Monge and Kantorovich formulations of $\PStheta$, proving a
constrained version of the Birkhoff-von Neumann theorem
\cite{birkhoff1946three}. Additionally, we study the Min-Pivot Sliced
Discrepancy, a variant that matches the true Wasserstein distance for discrete
measures when the space dimension is large enough with respect to the number of
points. We then study the Expected Sliced Wasserstein
Plan~\cite{liu2024expected}, which averages 1D sliced transport plans to obtain
high-dimensional (non sparse) couplings. This theoretical study is followed by
numerical experiments, illustrating the behaviour of the proposed transport
plans on synthetic datasets and shape registration tasks.

The paper is organised as follows. In \cref{sec:preliminary_nu_based_wass}, we
recall the necessary background on $\nu$-based Wasserstein geodesics, along with
some new theoretical results that will serve as building blocks for the rest of
the work. \cref{sec:pivot_sliced} presents and analyses the Pivot Sliced
Discrepancy. In \cref{sec:equality_CWtheta}, we establish a precise connection
between $\PStheta$ and a constrained Wasserstein discrepancy, showing that both
quantities coincide. This correspondence is further developed in
\cref{sec:pivot_sliced_discrete}, where we explore the Monge and Kantorovich
formulations of $\PStheta$ for discrete measures. We then study in
\cref{sec:min_PS} the Min-Pivot Sliced Discrepancy, and show that it recovers
the exact Wasserstein distance in certain discrete settings.
\cref{sec:expected_sliced} introduces and analyses the concept of Expected
Sliced Wasserstein Plans. Finally, \cref{sec:numerics} is dedicated to numerical
experiments.

%% file: sections/nu_based_wass.tex
\section{Reminders and New Results on the \texorpdfstring{$\nu$}{nu}-based
Wasserstein Distance}\label{sec:preliminary_nu_based_wass}

In this section, we lay some pre-requisites for the objects at play in the
paper. We begin by recalling the concept of generalised geodesics in
\cref{sec:gen_geod}, which allows us to introduce the $\nu$-based Wasserstein
distance in \cref{sec:nu_based_wass}. This (semi-)metric was first defined in
\cite{ambrosio2005gradient,nenna2023transport}, and we will sometimes also refer
to it as ``Pivot Wasserstein'', and prove new technical properties that will be
useful later. Later in this work, we will consider the Pivot Wasserstein
distance using a ``Wasserstein Mean'' pivot, and to this end we propose some
reminders on Wasserstein means in \cref{sec:wass_means}. Finally, in
\cref{sec:nu_based_wass_disintegration}, we revisit a disintegration formulation
of the $\nu$-based Wasserstein distance (first proved in
\cite{nenna2023transport}), which will sometimes be convenient for computations.

\subsection{Wasserstein Geodesics and Generalised Geodesics}\label{sec:gen_geod}

Given two measures $\mu_1, \mu_2 \in \mathcal{P}_2(\R^d)$, we denote by
$\Pi^*(\mu_1, \mu_2)$ the set of Optimal Transport plans between $\mu_1$ and
$\mu_2$ for the cost $\|x-y\|_2^2$. Using such plans, we can define a notion of
shortest path (i.e. geodesic) between $\mu_1$ and $\mu_2$ in the space
$(\mathcal{P}_2(\R^d), \W_2)$.

\begin{definition}
    A constant-speed geodesic between $\mu_1, \mu_2 \in \mathcal{P}_2(\R^d)$ is
    a curve $[0, 1] \longrightarrow \mathcal{P}_2(\R^d)$ constructed using an
    optimal transport plan $\gamma\in \Pi^*(\mu_1, \mu_2)$ as follows:
    \begin{equation}\label{eqn:W2_geodesic}
        \mu^{1\rightarrow 2}_\gamma(t) := \left((1-t)P_1+tP_2\right)\#\gamma,
    \end{equation}
\end{definition}
where $P_1: (x,y) \longmapsto x$ and $P_2: (x,y) \longmapsto y$ are the marginal
projection operators. Not only is $\mu_\gamma^{1\rightarrow 2}$ a geodesic for
the $\W_2$ metric, but all (constant-speed) geodesics between $\mu_1$ and
$\mu_2$ are of the form $\mu_\gamma^{1\rightarrow 2}$ for a suitable $\gamma\in
\Pi^*(\mu_1, \mu_2)$ (this is \cite[Theorem 7.2.2]{ambrosio2005gradient}).

If the chosen optimal transport plan $\gamma$ is induced by a transport map $T$
(which is to say that $\gamma = (I, T)\#\mu_1$), then the geodesic takes the
intuitive ``displacement'' formulation:
\begin{equation}\label{eqn:W2_geodesic_map}
    \mu^{1\rightarrow 2}_\gamma(t) := \left((1-t)I+tT\right)\#\mu_1,
\end{equation}
with $I$ denoting the identity map of $\R^d$.

A remarkable property of the 2-Wasserstein space is that it is a Positively
Curved (according to Alexandrov's metric definition of curvature) space, as
proved in \cite[Theorem 7.3.2, Equation 7.3.12]{ambrosio2005gradient}: for
$\mu_1, \mu_2, \nu \in \mathcal{P}_2(\R^d)$, $\gamma \in \Pi^*(\mu_1, \mu_2)$
and $t\in [0,1]$, we have
\begin{equation}\label{eqn:W2_PC}
    \W_2^2(\mu_\gamma^{1\rightarrow 2}(t), \nu) \geq (1-t)\W_2^2(\mu_1, \nu) + t\W_2^2(\mu_2, \nu) - (1-t)t\W_2^2(\mu_1, \mu_2).
\end{equation}
For $t := \tfrac{1}{2}$, this can be re-written as
\begin{equation}
    \W_2^2(\mu_1, \mu_2) \geq 2\W_2^2(\mu_1, \nu) + 2\W_2^2(\mu_2, \nu) - 4\W_2^2(\mu_\gamma^{1\rightarrow 2}(t), \nu).
\end{equation}
Unfortunately, the squared distance $\W_2^2$ is not $\lambda$-convex along these
Wasserstein geodesics \cite[Example 9.1.5]{ambrosio2005gradient}, which
motivated \cite{ambrosio2005gradient} to introduce other curves, coined
``generalised geodesics'', that satisfy this desirable property. First, we
consider two optimal plans $\gamma_{1}\in \Pi^*(\nu, \mu_1)$ and $\gamma_2\in
\Pi^*(\nu, \mu_2)$. To introduce the notion of generalised geodesics, we will
require a 3-plan $\rho \in \Pi(\nu, \mu_1, \mu_2) \in \mathcal{P}_2(\R^{3d})$
(i.e. with marginals $\rho_0 = \nu, \rho_1 = \mu_1, \rho_2 = \mu_2$), such that
its bi-marginals coincide with the plans $\gamma_1$ and $\gamma_2$: we require
$\rho_{0,1}:= P_{0,1}\#\rho=\gamma_1$ and $\rho_{0,2}:=P_{0,2}\#\rho=\gamma_2$,
where $P_{0,i}:= (y, x_1, x_2) \longmapsto (y, x_i)$. We introduce the following
notation for such 3-plans:
\begin{equation}\label{eqn:def_Gamma3} 
    \Gamma(\nu, \mu_1, \mu_2) := \left\{\rho
    \in \mathcal{P}_2(\R^{3d}) : \rho_{0,1}\in \Pi^*(\nu, \mu_1)\
    \text{and}\ \rho_{0,2}\in \Pi^*(\nu, \mu_2)\right\}.
\end{equation}
\begin{definition}
    A generalised geodesic based on $\nu$ between $\mu_1$ and $\mu_2$ is then
    defined as (\cite[Definition 9.2.2]{ambrosio2005gradient}), given a $\rho
    \in \Gamma(\nu, \mu_1, \mu_2)$:
    \begin{equation}\label{eqn:generalised_geodesic}
        \mu_\rho^{1\rightarrow 2}(t) := \left((1-t)P_1 + tP_2\right)\#\rho.
    \end{equation}
\end{definition}
Note that this curve depends on the choice of the 3-plan $\rho$, which itself
depends on the optimal plans $\gamma_1$ and $\gamma_2$. The existence of such a
$\rho$ can be shown using the gluing lemma (as presented in \cite[Lemma
5.5]{santambrogio2015optimal}, for example). As desired, the curvature induced
by these curves makes $\W_2^2$ convex along these geodesics (in a certain sense,
see \cite[Definition 9.2.4]{ambrosio2005gradient}), namely we have the following
inequality (\cite[Equation 9.2.7c]{ambrosio2005gradient}), which is reversed
compared to \cref{eqn:W2_PC}:
\begin{equation}\label{eqn:W2_convex_along_generalised_geodesics}
    \W_2^2(\mu_\rho^{1\rightarrow 2}(t), \nu) \leq (1-t)\W_2^2(\mu_1, \nu) 
    + t\W_2^2(\mu_2, \nu) - (1-t)t\W_2^2(\mu_1, \mu_2).
\end{equation}
Like before, setting $t:=\tfrac{1}{2}$ yields the following inequality:
\begin{equation}
    \W_2^2(\mu_1, \mu_2) \leq 2\W_2^2(\mu_1, \nu) + 2\W_2^2(\mu_2, \nu) 
    - 4\W_2^2(\mu_\rho^{1\rightarrow 2}(t), \nu).
\end{equation}
If the optimal transport plans $\gamma_1$ and $\gamma_2$ are induced
respectively by transport maps $T_1$ and $T_2$, then the choice of $\rho$ is
unique, with $\rho = (I, T_1, T_2)\#\nu$ (\cite[Remark
9.2.3]{ambrosio2005gradient}, see also \cite[Lemma 5.3.2]{ambrosio2005gradient}
for a formal proof). This yields the following expression of the generalised
geodesic, which is substantially more intuitive:
\begin{equation}\label{eqn:generalised_geodesic_map}
    \mu_\rho^{1\rightarrow 2}(t) = ((1-t)T_1 + tT_2)\#\nu.
\end{equation}

\subsection{The \texorpdfstring{$\nu$}{nu}-based Wasserstein
Distance}\label{sec:nu_based_wass}

A closely related concept is the $\nu$-based Wasserstein (semi)-distance,
introduced by Nenna and Pass in \cite{nenna2023transport}. This time we use a
pivot measure $\nu$ to introduce a variant of the Wasserstein distance, yielding
the following definition by \cite[Definition
3]{nenna2023transport}\footnote{Their definition seems to have a typo, with
$\Pi^*(\mu_i, \nu)$ instead of $\Pi^*(\nu, \mu_i)$. Furthermore, they work with
measures supported on a bounded and convex domain of $\R^d$, but as they remark
(\cite[footnote 4]{nenna2023transport}), and given \cite[Chapter
9]{ambrosio2005gradient}, generalisation to measures on $\R^d$ with a finite
moment of order 2 is perfectly natural.}:
\begin{definition} For $\nu \in \mathcal{P}_2(\R^d)$, the $\nu$-based
    Wasserstein (semi)-metric between $\mu_1, \mu_2 \in \mathcal{P}_2(\R^d)$ is
    defined as:
    \begin{equation}\label{eqn:nu_based_Wass} \W_\nu^2(\mu_1, \mu_2) :=
        \underset{\rho \in \Gamma(\nu, \mu_1, \mu_2)}{\min}\ 
        \int_{\R^{3d}}\|x_1 - x_2\|_2^2\dd\rho(y,x_1,x_2).
    \end{equation}
\end{definition}

We illustrate the $\nu$-based Wasserstein distance on a simple example in
\cref{fig:ex_W_nu}.

\begin{figure}[ht]
    \begin{center}
        \includegraphics[width=0.6\linewidth]{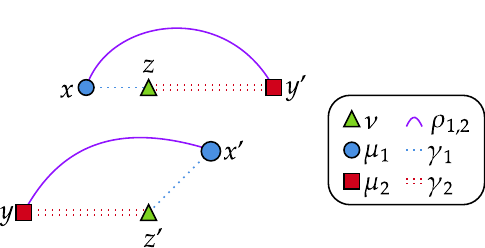}
    \end{center}
    \caption{Example of the couplings behind $\W_\nu(\mu_1, \mu_2)$ for discrete
    measures on $\R^2$. The measure $\nu$ is drawn with green triangles, $\mu_1$
    with blue circles and $\mu_2$ with red squares. The (unique) OT plan
    $\gamma_1$ between $\nu$ and $\mu_1$ is drawn with dotted blue lines, the
    (also unique) OT plan $\gamma_2$ between $\nu$ and $\mu_2$ with red double
    dotted lines. The plans induce a unique valid 3-plan $\rho \in \Gamma(\nu,
    \mu_1, \mu_2)$, we represent the coupling $\rho_{1, 2}$ between $\mu_1$ and
    $\mu_2$ with curved purple lines. Notice that the coupling $\rho_{1, 2}$
    differs from the (unique) OT coupling between $\mu_1$ and $\mu_2$.}
    \label{fig:ex_W_nu}
\end{figure}

The question of whether the infimum defining $\W_\nu$ is attained was not
addressed by \cite{nenna2023transport}, we show that it is indeed the case in
\cref{prop:W_nu_inf_attained}, using a technical property of the 3-plan set
$\Gamma$ defined in \cref{eqn:def_Gamma3}. We remind that by Prokhorov's
theorem, a subset of $\mathcal{P}_2(\R^d)$ is a tight set if and only if it is
pre-compact, which means that any sequence of measures in the set has a weakly
converging subsequence.
\begin{lemma}\label{lemma:tightness_Gamma3}    
    \begin{enumerate}
        \item For tight sets $P, Q_1, Q_2 \subset \mathcal{P}_2(\R^d)$, the set
        $$\Gamma(P, Q_1, Q_2) := \{\rho \in \Gamma(\nu, \mu_1, \mu_2): (\nu,
        \mu_1, \mu_2) \in P\times Q_1 \times Q_2\}$$ is tight in
        $\mathcal{P}_2(\R^{3d})$.
        \item Consider sequences $\nu^{(n)}, \mu_1^{(n)}, \mu_2^{(n)} \in
        \mathcal{P}_2(\R^d)^\N$ respectively converging to $\nu, \mu_1, \mu_2
        \in \mathcal{P}_2(\R^d)$ for the weak convergence of measures, and a
        sequence $(\rho_n) \in \mathcal{P}_2(\R^{3d})^\N$ such that $\forall n
        \in \N, \;\rho_n \in \Gamma(\nu^{(n)}, \mu_1^{(n)}, \mu_2^{(n)})$ with
        $\rho_n \xrightarrow[n\longrightarrow +\infty]{w} \rho \in
        \mathcal{P}_2(\R^{3d})$. Then $\rho \in \Gamma(\nu, \mu_1, \mu_2)$.
        \item For $\nu, \mu_1, \mu_2 \in \mathcal{P}_2(\R^d)$, the set
        $\Gamma(\nu, \mu_1, \mu_2)$ is compact in $\mathcal{P}_2(\R^{3d})$.
    \end{enumerate}
\end{lemma}
\begin{proof}
    For 1. we set $\varepsilon > 0$. By tightness of $P, Q_1, Q_2$ and
    Prokhorov's theorem, there exists a compact set $\K \subset \R^d$ such that
    for any $\mu \in P \cup Q_1 \cup Q_2,\; \mu(\R^d \setminus \K) < \varepsilon
    / 3$. It follows that for any $\rho \in \Gamma(P, Q_1, Q_2)$,
    \begin{align*}
        \rho(\R^{3d}\setminus \K^3) &\leq 
        \rho\left((\R^d\setminus \K) \times \R^d \times \R^d\right) \\
        &\quad + \rho\left(\R^d \times (\R^d\setminus \K) \times \R^d\right) \\ 
        &\quad + \rho \left(\R^d \times \R^d \times (\R^d\setminus \K)\right) \\
        &= \nu(\R^d\setminus \K) + \mu_1(\R^d\setminus \K) 
        + \mu_2(\R^d\setminus \K) \\
        &< \varepsilon,
    \end{align*}
    and thus $\Gamma(P, Q_1, Q_2)$ is tight.

    For 2. we observe that for $i \in \{1, 2\},\; [\rho_n]_{0,i} \in
    \Pi^*(\nu^{(n)}, \mu_i^{(n)})$. Given that $\nu^{(n)}
    \xrightarrow[n\longrightarrow+\infty]{w} \nu$ and $\mu_i^{(n)}
    \xrightarrow[n\longrightarrow+\infty]{w} \mu_i$, and that $[\rho_n]_{0,i}
    \xrightarrow[n\longrightarrow+\infty]{w} \rho_{0, i}$, \cite[Theorem
    5.20]{villani} shows that $\rho_{0, i} \in \Pi^*(\nu, \mu_i)$ (the result
    provides the existence of a subsequence converging to an element of
    $\Pi^*(\nu, \mu_i)$, then uniqueness of the limit shows $\rho_{0, i} \in
    \Pi^*(\nu, \mu_i)$), and we conclude that $\rho \in \Gamma(\nu, \mu_1,
    \mu_2)$ by definition.

    For 3., take $(\rho_n) \in \Gamma(\nu, \mu_1, \mu_2)^\N$. By 1) and
    tightness of $\{\nu\}, \{\mu_1\}, \{\mu_2\}$, there exists an extraction
    $\alpha$ such that $\rho_{\alpha(n)}
    \xrightarrow[n\longrightarrow+\infty]{w} \rho \in \mathcal{P}_2(\R^{3d})$,
    then we show that $\rho \in \Gamma(\nu, \mu_1, \mu_2)$ using 2) with
    $\forall n \in \N,\; \nu^{(n)} := \nu,\; \mu_i^{(n)} := \mu_i$ for $i\in
    \{1, 2\}$.
\end{proof}

\begin{prop}\label{prop:W_nu_inf_attained} For $\nu, \mu_1, \mu_2 \in
    \mathcal{P}_2(\R^d)$, it holds
    $$\underset{\rho \in \Gamma(\nu, \mu_1, \mu_2)}{\inf}\ \Int{\R^{3d}}{}\|x_1
    - x_2\|_2^2\dd\rho(y,x_1,x_2) = \underset{\rho \in \Gamma(\nu, \mu_1,
    \mu_2)}{\min}\ \Int{\R^{3d}}{}\|x_1 - x_2\|_2^2\dd\rho(y,x_1,x_2). $$
\end{prop}
\begin{proof}    
    By \cref{lemma:tightness_Gamma3} item 3), $\Gamma(\nu, \mu_1, \mu_2)$ is a
    compact subset of $\mathcal{P}_2(\R^{3d})$. Then the map $J :\rho\in
    \mathcal{P}_2(\R^{3d}) \longmapsto \int_{\R^{3d}}\|x_1 - x_2\|_2^2\dd\rho(y,
    x_1, x_2)$ is lower semi-continuous with respect to the weak convergence of
    measures (\cite[Lemma 1.6]{santambrogio2015optimal}), hence the infimum is
    attained.
\end{proof}

Another consequence of \cref{lemma:tightness_Gamma3} is that the $\nu$-based
Wasserstein distance is lower semi-continuous with respect to the weak
convergence of measures, which is a property that was not studied in
\cite{nenna2023transport}.
\begin{prop}\label{prop:W_nu_lsc} The map $(\nu, \mu_1, \mu_2) \in
    \mathcal{P}_2(\R^d)^3 \longmapsto \W_\nu(\mu_1, \mu_2)$ is lower
    semi-continuous with respect to the weak convergence of measures: for any
    $\nu^{(n)} \xrightarrow[n\longrightarrow + \infty]{w} \nu\in
    \mathcal{P}_2(\R^d),\; \mu_i^{(n)}\xrightarrow[n\longrightarrow + \infty]{w}
    \mu_i\in \mathcal{P}_2(\R^d),\; i\in \{1, 2\}$, we have:
    \begin{equation}\label{eqn:W_nu_lsc}
        \W_\nu(\mu_1, \mu_2) \leq \underset{n\longrightarrow+\infty}{\liminf}
        \W_{\nu^{(n)}}(\mu_1^{(n)}, \mu_2^{(n)}).
    \end{equation}
\end{prop}
\begin{proof}
    Without loss of generality, we can assume that \[\W_{\nu^{(n)}}(\mu_1^{(n)},
    \mu_2^{(n)}) \xrightarrow[n \longrightarrow +
    \infty]{}\underset{n\longrightarrow+\infty}{\liminf}
    \W_{\nu^{(n)}}(\mu_1^{(n)}, \mu_2^{(n)})\] (up to considering an extraction
    of all sequences). For $n\in \N$, we can choose $\rho_n \in
    \Gamma(\nu^{(n)}, \mu_1^{(n)}, \mu_2^{(n)})$ optimal by
    \cref{prop:W_nu_inf_attained}. By \cref{lemma:tightness_Gamma3} item 1) and
    tightness of the sets $\{\nu^{(n)}\}, \{\mu_1^{(n)}\}, \{\mu_2^{(n)}\}$,
    there exists an extraction $\alpha$ such that $\rho_{\alpha(n)}
    \xrightarrow[n\longrightarrow +\infty]{w} \rho \in \mathcal{P}_2(\R^{3d})$
    and by \cref{lemma:tightness_Gamma3} item 2) we have $\rho \in \Gamma(\nu,
    \mu_1, \mu_2)$. By lower semi-continuity of the map $$J :\rho\in
    \mathcal{P}_2(\R^{3d}) \longmapsto \int_{\R^{3d}}\|x_1 - x_2\|_2^2\dd\rho(y,
    x_1, x_2),$$ (see \cite[Lemma 1.6]{santambrogio2015optimal}) we have:
    \begin{align*}
        \W_\nu^2(\mu_1, \mu_2) 
        &\leq J(\rho) \\
        &\leq \underset{n\longrightarrow +\infty}{\liminf} 
        J(\rho_{\alpha(n)}) \\ 
        &= \underset{n\longrightarrow +\infty}{\liminf}
        \W_{\nu^{(\alpha(n))}}^2(\mu_1^{(\alpha(n))}, \mu_2^{(\alpha(n))}) \\
        &= \underset{n\longrightarrow+\infty}{\liminf}
        \W_{\nu^{(n)}}^2(\mu_1^{(n)}, \mu_2^{(n)}),
    \end{align*}
    where the first inequality follows from the definition of $\W_\nu$, since
    $\rho \in \Gamma(\nu, \mu_1, \mu_2)$ is admissible, and the second
    inequality follows from the lower semi-continuity of $J$. The first equality
    is due to the optimality of $\rho_{\alpha(n)}$, and the second equality
    follows from our reduction to the case where $$\W_{\nu^{(n)}}(\mu_1^{(n)},
    \mu_2^{(n)}) \xrightarrow[n \longrightarrow +
    \infty]{}\underset{n\longrightarrow+\infty}{\liminf}
    \W_{\nu^{(n)}}(\mu_1^{(n)}, \mu_2^{(n)}).$$
    \vskip -20pt
\end{proof}

Full continuity with respect to the weak convergence of measures is not
guaranteed, as shown in \cref{ex:ce_W_nu_not_continuous}.

\begin{example}[$\W_\nu(\cdot, \mu_2)$ is not
    continuous]\label{ex:ce_W_nu_not_continuous} Consider the following
    empirical measures in $\R^2$:
    \begin{align*}
        \nu &:= \tfrac{1}{2}(\delta_{z} + \delta_{z'}),\; 
        z := (0, 1),\; z':= (0, -1);\\
        \mu_1^{(n)} &:= \tfrac{1}{2}(\delta_{x_n} + \delta_{x'}),\;
        x_n := (-1, 2^{-n}),\; x' := (1, 0); \\
        \mu_2 &= \tfrac{1}{2}(\delta_{y} + \delta_{y'}),\; 
        y := (-2, -1),\; y := (2, 1).
    \end{align*}
    For each $n\in \N$, we have $\Pi^*(\nu, \mu_1^{(n)}) = \{\gamma_1^{(n)}\}$
    with $\gamma_1^{(n)} := \frac{1}{2}(\delta_{(z, x_n)} + \delta_{(z', x')})$.
    We also have $\Pi^*(\nu, \mu_2) = \{\gamma_2\}$ with $\gamma_2 :=
    \frac{1}{2}(\delta_{(z, y')} + \delta_{(z', y)})$. This shows that
    $\Gamma(\nu, \mu_1^{(n)}, \mu_2) = \{\rho_n\}$ where $\rho_n :=
    \frac{1}{2}(\delta_{(z, x_n, y')} + \delta_{(z', x', y)})$, yielding the
    cost
    \begin{align*}
        \W_\nu^2(\mu_1^{(n)}, \mu_2) &= \tfrac{1}{2}\|x_n - y'\|_2^2 +
        \tfrac{1}{2}\|x' - y\|_2^2 \\
        &= \tfrac{1}{2}\left(3^2 + (1-2^{-n})^2\right) + \tfrac{1}{2}\left(3^2 + 1^2\right) \xrightarrow[n\longrightarrow +\infty]{} 10.
    \end{align*}
    However, we have $\mu_1^{(n)}\xrightarrow[n\longrightarrow +\infty]{w}
    \mu_1 =\frac{1}{2}(\delta_{x} + \delta_{x'})$ with $x := (-1, 0)$. We see
    that $\Pi^*(\nu, \mu_1) = \Pi(\nu, \mu_2)$, and clearly the choice $\gamma_1
    := \frac{1}{2}(\delta_{(z, x')} + \delta_{(z', x)})$ will be optimal, such
    that $\rho := \frac{1}{2}(\delta_{(z, x', y')} + \delta_{(z', x, y)})$ is
    optimal for $\W_\nu^2(\mu_1, \mu_2) = 2 < \underset{n\longrightarrow
    +\infty}{\lim}\W_\nu^2(\mu_1^{(n)}, \mu_2) = 10.$ We illustrate the setting
    of this example in \cref{fig:ce_W_nu_not_continuous}.
\end{example}

\begin{figure}[ht]
    \begin{center}
        \includegraphics[width=0.6\linewidth]{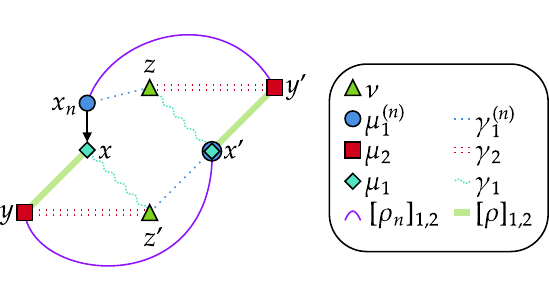}
    \end{center}
    \caption{Representation of \cref{ex:ce_W_nu_not_continuous}. The measure
    $\nu$ is drawn with green triangles, $\mu_1^{(n)}$ with blue circles,
    $\mu_2$ with red squares, and the limit $\mu$ with light blue diamonds. The
    OT plan $\gamma_1^{(n)}$ between $\nu$ and $\mu_1^{(n)}$ is drawn with
    dotted blue lines, the OT plan $\gamma_2$ between $\nu$ and $\mu_2$ with red
    double dotted lines, and the induced plan $[\rho_n]_{1, 2}$ between
    $\mu_1^{(n)}$ and $\mu_2$ with curved purple lines. As for the limit, an OT
    plan $\gamma_1$ between $\nu$ and $\mu_1$ is drawn with curved dashed light
    blue lines, and the induced plan $[\rho]_{1, 2}$ between $\mu_1$ and $\mu_2$
    using $\gamma_1$ and $\gamma_2$ is drawn with thick green lines.}
    \label{fig:ce_W_nu_not_continuous}
\end{figure}

As remarked earlier (again, \cite[Remark 9.2.3]{ambrosio2005gradient}), if each
$\Pi^*(\nu, \mu_i)$ is reduced to a single plan $\gamma_i$ induced by $T_i$ (for
$i\in \{1, 2\}$), then the only element $\rho \in \Gamma(\nu, \mu_1, \mu_2)$ is
$\rho=(I, T_1, T_2)\#\nu$, yielding the following formulation for the
$\nu$-based Wasserstein distance (see also \cite[Example 9]{nenna2023transport}
and the Linear OT framework \cite{wang2013linear}):
\begin{equation}\label{eqn:nu_based_Wass_maps}
    \W_\nu^2(\mu_1, \mu_2) = \Int{\R^d}{} \|T_1(y)-T_2(y)\|_2^2\dd\nu(y).
\end{equation}
A result of interest is \cite[Lemma 9.2.1, Equation
9.2.7b]{ambrosio2005gradient}, which states that for any $\rho \in \Gamma(\nu,
\mu_1, \mu_2)$ (see \cref{eqn:def_Gamma3})
\begin{equation}\label{eqn:ags_eq_9.2.7b}
    \W_2^2(\mu_1, \mu_\rho^{1\rightarrow 2}(t)) = (1-t)\W_2^2(\mu_1, \nu) + 
    t\W_2^2(\mu_2, \nu) - (1-t)t\int_{\R^{3d}}\|x_1 - x_2\|_2^2
    \dd\rho(y,x_1,x_2).
\end{equation}
Taking in particular a 3-plan $\rho^*\in \Gamma(\nu, \mu_1, \mu_2)$ that is
optimal for the $\nu$-based Wasserstein distance (\cref{eqn:nu_based_Wass}), we
obtain
\begin{equation}\label{eqn:nu_W_equality}
    \W_2^2(\mu_1, \mu_{\rho^*}^{1\rightarrow 2}(t)) = (1-t)\W_2^2(\mu_1, \nu) + 
    t\W_2^2(\mu_2, \nu) - (1-t)t\W_\nu^2(\mu_1, \mu_2).
\end{equation}

\subsection{Reminders on Wasserstein Means}\label{sec:wass_means}

A natural application of Wasserstein geodesics is the concept of Wasserstein
means, which we will require in \cref{sec:pivot_sliced}. The following result
states that Wasserstein means are exactly the middles of Wasserstein geodesics.
For the sake of completeness, we provide some reminders on geodesic middles in
\cref{sec:midpoints_are_geodesic_middles}, wherein we recall and prove an
analogous result for geodesic spaces.
\begin{prop}\label{prop:wass_mean_geodesic} For $\mu_1, \mu_2 \in
    \mathcal{P}_2(\R^d)$, the set of Wasserstein Means between $\mu_1$ and
    $\mu_2$
    \begin{equation}\label{eqn:def_W_mean}
        \M(\mu_1, \mu_2) := \underset{\mu \in \mathcal{P}_2(\R^d)}{\argmin}\ 
        \W_2^2(\mu_1, \mu) + \W_2^2(\mu, \mu_2)
    \end{equation}
    can be expressed using Wasserstein geodesics \cref{eqn:W2_geodesic}:
    \begin{equation}
        \M(\mu_1, \mu_2) = 
        \left\{\mu_{\gamma}^{1\rightarrow 2}(\tfrac{1}{2}) : 
        \gamma \in \Pi^*(\mu_1, \mu_2)\right\} = 
        \left\{\left(\tfrac{1}{2}P_1 + \tfrac{1}{2}P_2\right)
        \#\gamma : \gamma \in \Pi^*(\mu_1, \mu_2)\right\}.
    \end{equation}
\end{prop}
\begin{proof}
    The result is an application of \cref{lemma:midpoints} in the geodesic space
    $(\mathcal{P}_2(\R^d), \W_2)$.
\end{proof}

\subsection{Another Formulation of \texorpdfstring{$\W_\nu$}{Wnu} with Measure
Disintegration}\label{sec:nu_based_wass_disintegration}

In this work, we will need a convenient formulation of the $\nu$-based
Wasserstein distance which uses the notion of disintegration of measures. We
recall this notion in \cref{sec:disintegration}, and provide a proof in
\cref{sec:proof_W_nu_disintegration} of \cite[Theorem 12 item
1)]{nenna2023transport}, adapted to measures in $\mathcal{P}_2(\R^d)$. In
\cref{ex:W_nu_disintegration}, we illustrate the result on a simple example with
discrete measures.
\begin{example}\label{ex:W_nu_disintegration} We consider measures $\nu, \mu_1,
    \mu_2 \in \mathcal{P}_2(\R^d)$ as in \cref{fig:W_nu_disintegration}. We
    consider two optimal plans $\gamma_1 \in \Pi^*(\nu, \mu_1)$ and $\gamma_2
    \in \Pi^*(\nu, \mu_2)$, represented in \cref{fig:W_nu_disintegration}.
    Writing the disintegrations as $\gamma_i(\dd y, \dd x_i) = \nu(\dd y)
    \gamma_i^{y}(\dd x_i)$, we can apply \cref{thm:W_nu_disintegration} to
    compute $\W_\nu^2(\mu_1, \mu_2)$:
    \begin{align*}
        \W_\nu^2(\mu_1, \mu_2) &= 
        \tfrac{1}{2}\W_2^2\left(\gamma_1^{z_1}, \gamma_2^{z_1}\right)
        + \tfrac{1}{2}\W_2^2\left(\gamma_1^{z_2}, \gamma_2^{z_2}\right) \\
        &= \tfrac{1}{2}\W_2^2\left(\tfrac{2}{3}\delta_{x_1} 
        + \tfrac{1}{3}\delta_{x_2},
        \tfrac{2}{3}\delta_{y_1} + \tfrac{1}{3}\delta_{y_2}\right) +
        \tfrac{1}{2}\W_2^2\left(\tfrac{1}{3}\delta_{x_2} 
        + \tfrac{2}{3}\delta_{x_3},
        \tfrac{1}{3}\delta_{y_2} + \tfrac{2}{3}\delta_{y_3}\right).
    \end{align*}
\end{example}
\begin{figure}[H]
    \centering
    \includegraphics[width=0.7\textwidth]{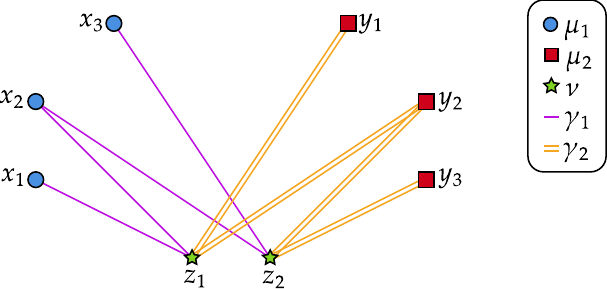}
    \caption{In this example, there is a unique optimal transport plan
    $\gamma_1$ (purple lines) between $\mu_1$ (blue circles) and the pivot $\nu$
    (green stars), and likewise for $\gamma_2$ (orange double lines) between
    $\mu_2$ (red squares) and $\nu$. The disintegration kernel $\gamma_1^{z_1}$
    in the disintegration $\gamma_1(\dd z, \dd x) = \nu(\dd z)\gamma_1^z(\dd x)$
    is the probability measure $\gamma_1^{z_1} = \tfrac{2}{3}\delta_{x_1} +
    \tfrac{1}{3}\delta_{x_2}$, and likewise for $\gamma_1^{z_2}, \gamma_2^{z_1},
    \gamma_2^{z_2}$. }
    \label{fig:W_nu_disintegration}
\end{figure}
\begin{theorem}[\cite{nenna2023transport} Theorem 12 item
1]\label{thm:W_nu_disintegration} Let $\nu, \mu_1, \mu_2 \in
\mathcal{P}_2(\R^d)$. The following equality holds:
    \begin{equation}\label{eqn:W_nu_disintegration}
        \W_\nu^2(\mu_1, \mu_2) = 
        \underset{\gamma_i \in \Pi^*(\nu, \mu_i),\: i\in \{1,2\}}{\min}\ 
        \int_{\R^d}\W_2^2(\gamma_1^y, \gamma_2^y)\dd\nu(y),
    \end{equation}
    where for $i\in \{1, 2\}$, $\gamma_i^y\in \mathcal{P}_2(\R^d)$ is defined
    using the disintegration $\gamma_i(\dd y, \dd x) = \nu(\dd y)\gamma_i^y(\dd
    x)$.
\end{theorem}
\begin{proof}
    We provide a proof in \cref{sec:proof_W_nu_disintegration}, which
    generalises that in \cite{nenna2023transport} to measures in
    $\mathcal{P}_2(\R^d)$, following similar ideas.
\end{proof}

%% file: sections/pivot_sliced.tex
\section{The Pivot Sliced Discrepancy}\label{sec:pivot_sliced}

\subsection{Definition with the \texorpdfstring{$\nu$}{nu}-based Wasserstein
Distance} 

We introduce a generalised version of SWGG introduced in \cite{mahey23fast} for
general measures in $\mathcal{P}_2(\R^d)$ (and fixing the ambiguity issues that
will be discussed in \cref{ex:swgg_ambiguity}), using the $\nu$-based
Wasserstein distance (\cref{eqn:nu_based_Wass}, and see
\cite{nenna2023transport}), where the base measure $\nu$ is taken as a middle of
projected versions of the measures:
\begin{definition}\label{def:S_theta} Let $\mu_1, \mu_2$ of
    $\mathcal{P}_2(\R^d)$, take $\mu_\theta \in \M(Q_\theta\#\mu_1,
    Q_\theta\#\mu_2)$, where $Q_\theta: x\longmapsto (\theta^\top x)\theta$.
    Then, we define
    \begin{equation}\label{eqn:S_theta}
        \PStheta(\mu_1, \mu_2) := \W_{\mu_\theta}(\mu_1, \mu_2).
    \end{equation}
\end{definition}
We consider two related projection operations: $P_\theta: x \longmapsto
\theta^\top x$ and $Q_\theta: x \longmapsto (\theta^\top x)\theta$. The first
one is valued in $\R$, while the second is valued in $\R\theta \subset \R^d$. To
fix ideas, we illustrate the definition of $\PStheta$ in the case of discrete
measures without projection ambiguity in \cref{fig:S_theta_simple}.
\begin{figure}[H]
    \centering
    \includegraphics[width=0.8\textwidth]{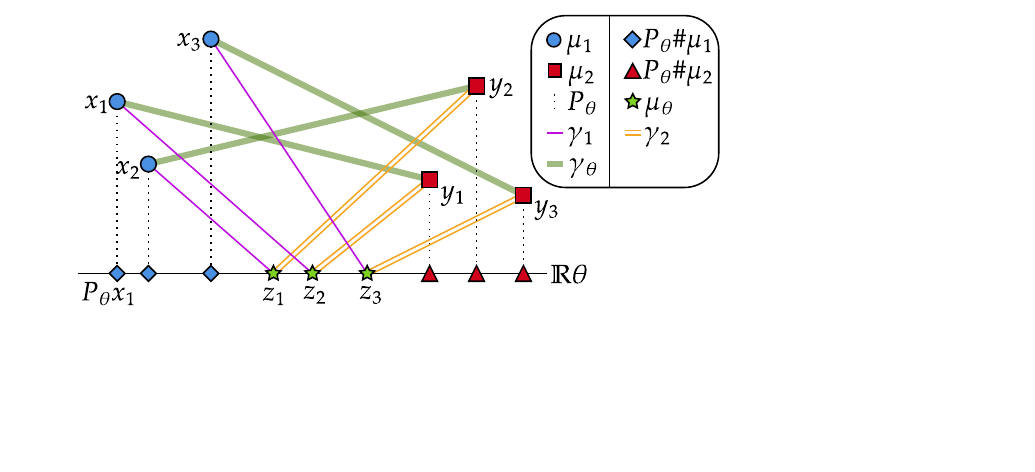}
    \caption{Illustration of the definition of $\PStheta$ in the case of
    discrete measures without projection ambiguity. The measure $\mu_1$ is
    represented by blue circles, and $\mu_2$ by red squares. The projected
    measures $Q_\theta\#\mu_1$ and $Q_\theta\#\mu_2$ are represented by blue
    diamonds and red triangles respectively. The middle $\mu_\theta$ of the
    projections is represented by green stars. Once this middle is determined,
    we compute optimal transport plans $\gamma_1,\gamma_2$ between $\mu_\theta$
    and $\mu_1,\mu_2$ respectively (in this case, they are unique). We represent
    $\gamma_1$ by purple lines and $\gamma_2$ by orange double lines. To obtain
    the coupling corresponding to the cost $\PStheta(\mu_1, \mu_2)$, we look at
    the targets of each point $(z_i)$ of the projected middle $\mu_\theta$:
    since $z_1$ is mapped to $x_1$ in $\mu_1$ and to $y_1$ in $\mu_2$, the
    coupling $\gamma_\theta$ maps $x_1$ to $y_1$, and so on. The coupling
    $\gamma_\theta$ is represented with thick green lines.}
    \label{fig:S_theta_simple}
\end{figure}

The idea of using a pivot measure is to find an optimal manner of correcting
projection ambiguities. To illustrate this, we consider a simple pathological
example in \cref{fig:S_theta_ambiguity_trivial}, where the projections of the
points of the support of $\mu_1$ and $\mu_2$ are not distinct.
\begin{figure}[H]
    \centering
    \includegraphics[width=0.5\textwidth]{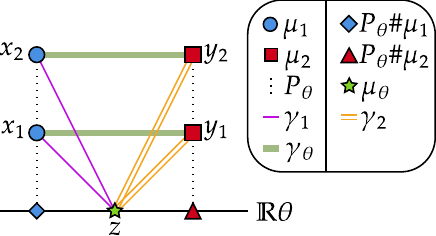}
    \caption{In this example, we notice that $P_\theta\#\mu_1$ and
    $P_\theta\#\mu_2$ are reduced to Dirac masses, thus their middle
    $\mu_\theta$ is the middle Dirac mass. The optimal couplings $\gamma_1$ and
    $\gamma_2$ between $\mu_\theta, \mu_1$ and $\mu_\theta, \mu_2$ are then
    unique. It is then easy to see that the optimal $\rho \in \Gamma(\mu_\theta,
    \mu_1, \mu_2)$ is such that $\rho_{1, 2} =: \gamma_\theta$ is the OT
    coupling between $\mu_1$ and $\mu_2$. In this example, $\PStheta(\mu_1,
    \mu_2) = \W_2(\mu_1, \mu_2)$.}
    \label{fig:S_theta_ambiguity_trivial}
\end{figure}
In \cref{fig:S_theta_ambiguity_medium}, we illustrate another simple example
where the projections of the points of the support of $\mu_1$ are not distinct,
but where they are distinct for $\mu_2$.
\begin{figure}[H]
    \centering
    \includegraphics[width=0.7\textwidth]{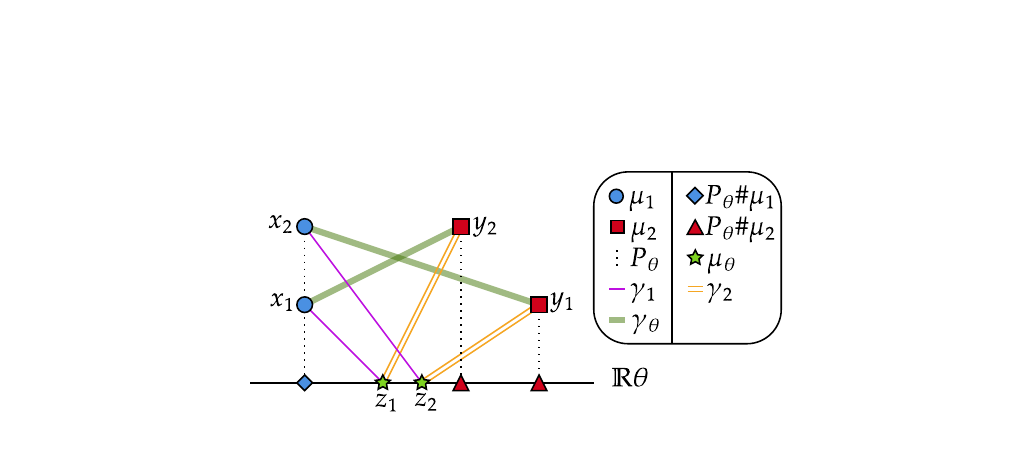}
    \caption{In this illustration, $P_\theta\#\mu_1$ is a Dirac mass but not
    $P_\theta\#\mu_2$. Since we compare the middle $\mu_\theta$ with $\mu_1$ and
    not $P_\theta\#\mu_1$, there is in this case a unique optimal plan
    $\gamma_1$ between $\mu_\theta$ and $\mu_1$. The optimal plan $\gamma_2$
    between $\mu_\theta$ and $\mu_2$ is also unique. The constraint $\rho_{0, 1}
    = \gamma_1,\; \rho_{0, 2} = \gamma_2$ imposes that $\rho_{1, 2} =
    \tfrac{1}{2}\delta_{x_1\otimes y_2} + \tfrac{1}{2}\delta_{x_2\otimes y_1}$
    for any $\rho \in \Gamma(\mu_\theta, \mu_1, \mu_2)$, hence there is no
    choice in the optimisation over $\rho$.}
    \label{fig:S_theta_ambiguity_medium}
\end{figure}

\begin{remark}
    As remarked by \cite[Proposition 16]{nenna2023transport}, when $\nu$ is
    absolutely continuous with respect to the one-dimensional Hausdorff on a
    line, then the $\nu$-based Wasserstein distance equates the
    \textit{layer-wise Wasserstein metric} introduced by \cite{kim2020optimal}.
    We will see in \cref{sec:equality_CWtheta} that $\PStheta$ equals another
    discrepancy that we call $\CWtheta$, and this equality allows us to show
    that $\PStheta$ satisfies the triangle inequality (and thus is a distance)
    on the set of measures with atomless projections, which is a stronger result
    than assuming absolute continuity of the pivot.
\end{remark}
Note that this is a generalisation of SWGG introduced in \cite{mahey23fast} in
the sense that they show that their definition of SWGG coincides with the
expression of \cref{eqn:S_theta} in \cite[Proposition 4.2]{mahey23fast}. To
prove that the quantity $\PStheta^2$ is well-defined, which is to say that it
does not depend on the choice of $\mu_\theta \in \M(Q_\theta\#\mu_1,
Q_\theta\#\mu_2)$, we will show that in fact $\M(Q_\theta\#\mu_1,
Q_\theta\#\mu_2)$ has only one element.
\begin{lemma}\label{lemma:S_theta_well_def} Let $\mu_1, \mu_2 \in
    \mathcal{P}_2(\R^d)$, and $\theta \in \SS^{d-1}$. Then 
    \begin{equation}\label{eqn:means_1D}
        \M(Q_\theta\#\mu_1, Q_\theta\#\mu_2) = 
        \left\{\mu_\theta[\mu_1, \mu_2]\right\},\quad 
        \mu_\theta[\mu_1, \mu_2] := 
        \left[\left(\tfrac{1}{2}F_{\nu_{1}}^{[-1]} 
        + \tfrac{1}{2}F_{\nu_2}^{[-1]}\right)\theta\right]\#\Leb_{[0, 1]},
    \end{equation}
    where for $i=1,2$, the measure $\nu_i$ is defined as $\nu_i =
    P_\theta\#\mu_i$, with $P_\theta: x \longmapsto \theta^\top x$, $\Leb_{[0,
    1]}$ the Lebesgue measure on $[0, 1]$, and where $F_\nu^{[-1]}$ for $\nu \in
    \mathcal{P}(\R)$ denotes the pseudo-inverse of its cumulative distribution:
    \begin{equation}\label{eqn:quantile_function}
        \forall t \in [0, 1],\quad F_\nu^{[-1]}(t) := 
        \inf \{s \in \R : \nu\left((-\infty, s]\right)\geq t\}.
    \end{equation}
\end{lemma}
\begin{proof}
    First, since the $Q_\theta\#\mu_i, i\in \{1,2\}$ are supported on
    $\R\theta$, we have 
    \begin{equation}\label{eqn:means_1d_embedding}
        \M(Q_\theta\#\mu_1, Q_\theta\#\mu_2) =
    \left\{\theta\#\mu : \mu \in \M(P_\theta\#\mu_1,
    P_\theta\#\mu_2)\right\},
    \end{equation}
    which amounts to reducing a problem on a line of direction $\theta$ to a
    problem on $\R$, then embedding the result onto the line $\R\theta$. We
    introduce $\nu_i := P_\theta\#\mu_i$ for $i\in \{1,2\}$ and leverage
    \cref{prop:wass_mean_geodesic}:
    \begin{equation}\label{eqn:apply_wass_mean_geodesic}
        \M(\nu_1, \nu_2) = 
        \left\{\left(\tfrac{1}{2}P_1+\tfrac{1}{2}P_2\right)\#\gamma 
        : \gamma \in \Pi^*(\nu_1, \nu_2)\right\}.
    \end{equation}
    Since the $\nu_i$ are measures on $\R$, by \cite[Theorem
    2.9]{santambrogio2015optimal}, the set of optimal plans $\Pi^*(\nu_1,
    \nu_2)$ is reduced to the plan $(F_{\nu_1}^{[-1]},
    F_{\nu_2}^{[-1]})\#\Leb_{[0,1]}$. Using \cref{eqn:apply_wass_mean_geodesic}
    above and the projection embedding from \cref{eqn:means_1d_embedding}, we
    obtain the result stated in \cref{eqn:means_1D}.
\end{proof}
\begin{remark}\label{remark:discrete_middles}Consider $\mu_1 =
    \tfrac{1}{n}\sum_i\delta_{x_i},\; \mu_2 = \tfrac{1}{n}\sum_i\delta_{y_i}$,
    $\theta\in \SS^{d-1}$ and $\sigma_\theta,\tau_\theta$ two permutations
    sorting respectively $(\theta^\top x_i)_i$ and $(\theta^\top y_i)_i$ (they
    may not be unique if the families $(\theta^\top x_i)_i$ and $(\theta^\top
    y_i)_i$ are not injective). Then the projected middle (computed using
    \cref{eqn:means_1D}) is explicit:
    \begin{equation}\label{eqn:projected_middle_discrete}
        \mu_\theta[\mu_1, \mu_2] = \cfrac{1}{n}\Sum{i=1}{n}\delta\left(\cfrac{\theta^\top \left(x_{\sigma_\theta(i)}+y_{\tau_\theta(i)}\right)}{2} \theta \right).
    \end{equation}
    Note that measure above does not depend on the choice of the sorting
    permutations $(\sigma_\theta, \tau_\theta)$, since the families
    $(\theta^\top x_{\sigma_\theta(i)})_i$ and $(\theta^\top
    y_{\tau_\theta(i)})_i$ do not. This expression is specific to the case where
    $\mu_1$ and $\mu_2$ are uniform discrete measures with the same amount of
    atoms.
\end{remark}
An interesting property of optimal transport between a measure $\mu\in
\mathcal{P}_2(\R^d)$ and another measure $\nu$ supported on a line $\R\theta$ is
that the set of optimal plans and the cost can be related to the one-dimensional
projections of $\mu$ and $\nu$ onto the line $\R\theta$. We remind $Q_\theta:
x\longmapsto (\theta^\top x)\theta$, and introduce $Q_{\theta^\perp} := I -
Q_\theta$. The following result is a generalisation of \cite[Lemma
4.6]{mahey23fast}, which was written in the case of uniform discrete measures.
Note that the exponent 2 in the cost is paramount and allows the separation of
orthogonal terms.
\begin{prop}\label{prop:semi_1D_ot} Let $\mu,\nu \in \mathcal{P}_2(\R^d)$ such
    that $\nu$ is supported on $\R\theta$, where $\theta\in \SS^{d-1}$. Then for
    any plan $\gamma \in \Pi(\nu, \mu)$, we have
    \begin{equation}\label{eqn:semi_1D_ot_split}
        \int_{\R^{2d}}\|x-y\|_2^2\dd\gamma(y,x) = 
        \int_{\R^{2d}}(\theta^\top (x-y))^2\dd\gamma(y,x) 
        + \int_{\R^d}\|Q_{\theta^\perp} x\|_2^2\dd\mu(x),
    \end{equation}
    with the alternate expression
    \begin{equation}\label{eqn:semi_1D_ot_split_push_forward}
        \int_{\R^{2d}}(\theta^\top (x-y))^2\dd\gamma(y,x) = 
        \int_{\R^2}(s-t)^2\dd (P_\theta, P_\theta)\#\gamma(s,t).
    \end{equation}
    This yields the following expression for the OT cost:
    \begin{equation}\label{eqn:semi_1D_ot_cost}
        \W_2^2(\nu ,\mu) = \W_2^2(P_\theta\#\nu, P_\theta\#\mu) +
        \int_{\R^d}\|Q_{\theta^\perp} x\|_2^2\dd\mu(x),
    \end{equation}
    and the following characterisation of the optimal plans:
    \begin{equation}\label{eqn:semi_1D_ot_plan}
        \Pi^*(\nu, \mu) = \left\{\gamma \in \Pi(\nu, \mu): (P_\theta,
        P_\theta)\#\gamma = \left(F_{P_\theta\#\nu}^{[-1]},
        F_{P_\theta\#\mu}^{[-1]}\right)\#\Leb_{[0, 1]}\right\}.
    \end{equation}
\end{prop}
\begin{proof}
    Let $\gamma\in \Pi(\nu, \mu)$. We have
    \begin{align}\label{eqn:semi_1D_ot_split_computation}
        \int_{\R^{2d}}\|x-y\|_2^2\dd\gamma(y,x) 
        &= \int_{\R^{2d}}\left(\|Q_\theta(x-y)\|_2^2 
        + \|Q_{\theta^\perp}(x-y)\|_2^2\right)\dd\gamma(y,x)\nonumber\\
        &= \int_{\R^{2d}}(\theta^\top (x-y))^2\dd\gamma(y,x) 
        + \int_{\R^d}\|Q_{\theta^\perp} x\|_2^2\dd\mu(x),
    \end{align}
    where the last equality comes from the fact that $\nu$ is supported on
    $\R\theta$ and that the second marginal of $\gamma$ is $\mu$. Since the
    second term does not depend on $\gamma$, by taking the infimum in
    \cref{eqn:semi_1D_ot_split}, we obtain \cref{eqn:semi_1D_ot_cost}, where the
    equality
    $$\underset{\gamma\in \Pi(\nu, \mu)}{\inf}\ \int_{\R^{2d}}(\theta^\top
    (x-y))^2\dd\gamma(y,x) = \W_2^2(P_\theta\#\nu, P_\theta\#\mu) $$ is
    justified by \cite[Lemma 2]{dumont22gromovmap}. Furthermore,
    \cref{eqn:semi_1D_ot_split} shows that $$(P_\theta, P_\theta)\#\Pi^*(\nu,
    \mu) \subset \Pi^*(P_\theta\#\nu, P_\theta\#\mu).$$ Indeed, take $\gamma\in
    \Pi^*(\nu, \mu)$, then since $$\pi_\theta := (P_\theta, P_\theta)\#\gamma
    \in \Pi(P_\theta\#\nu, P_\theta\#\mu),$$ \cref{eqn:semi_1D_ot_split} yields
    the optimality of $\pi_\theta$ for the problem $\W_2^2(P_\theta\#\nu,
    P_\theta\#\mu)$. By \cite[Theorem 2.9]{santambrogio2015optimal},
    $$\Pi^*(P_\theta\#\nu, P_\theta\#\mu) =
    \left\{\left(F_{P_\theta\#\nu}^{[-1]},
    F_{P_\theta\#\mu}^{[-1]}\right)\#\Leb_{[0, 1]}\right\} =:
    \{\pi_\theta^*\},$$ hence we have shown that $(P_\theta,
    P_\theta)\#\Pi^*(\nu, \mu) = \{\pi_\theta^*\}$. Conversely, take $\gamma \in
    \Pi(\nu, \mu)$ such that $(P_\theta, P_\theta)\#\gamma = \pi_\theta^*$,
    where $\pi_\theta^*$ is the unique element of $\Pi^*(P_\theta\#\nu,
    P_\theta\#\mu)$. Then by plugging $\gamma$ into \cref{eqn:semi_1D_ot_split},
    we obtain $\gamma \in \Pi^*(\nu, \mu)$ using \cref{eqn:semi_1D_ot_cost}. We
    conclude that
    $$\Pi^*(\nu, \mu) = \left\{\gamma \in \Pi(\nu, \mu): (P_\theta,
    P_\theta)\#\gamma = \pi_\theta^*\right\}.$$
    \vskip -15pt
\end{proof}

\subsection{Semi-Metric Properties of \texorpdfstring{$\PStheta$}{S-theta}}

We begin by stating straightforward properties of the discrepancy $\PStheta$:
\begin{prop}\label{prop:S_theta_semi_metric} Let $\mu_1, \mu_2\in
    \mathcal{P}_2(\R^d)$, and $\theta \in \SS^{d-1}$. Then the following
    properties hold:
    \begin{itemize}
        \item (Separation) $\PStheta(\mu_1, \mu_2) = 0$ if and only if $\mu_1 =
        \mu_2$.
        \item (Symmetry) $\PStheta(\mu_1, \mu_2) = \PStheta(\mu_2, \mu_1)$.
        \item (Upper-bound of $\W_2$) $\PStheta(\mu_1, \mu_2) \geq \W_2(\mu_1,
        \mu_2)$.
    \end{itemize}
\end{prop}
\begin{proof}
    If $\PStheta(\mu_1, \mu_2) = 0$, then by \cref{prop:W_nu_inf_attained},
    there exists $\rho\in \Gamma(\mu_\theta[\mu_1, \mu_2], \mu_1, \mu_2)$ such
    that $\int_{\R^{3d}}\|x_1-x_2\|_2^2\dd\rho(y, x_1, x_2) = 0$, then in
    particular, for $\gamma := \rho_{1, 2} \in \Pi(\mu_1, \mu_2)$, we have
    $\int_{\R^{2d}}\|x_1-x_2\|_2^2\dd\gamma(x_1, x_2) = 0$ and thus $x_1 = x_2$
    for $\gamma$-almost-every $(x_1, x_2) \in \R^{2d}$. For a test function
    $\phi\in \mathcal{C}_b^0$, we compute:
    $$\int_{\R^d} \phi(x_1)\dd\mu_1(x_1) = \int_{\R^d} \phi(x_1)\dd\gamma(x_1,
    x_2) = \int_{\R^d} \phi(x_2)\dd\gamma(x_1, x_2) = \int_{\R^d}
    \phi(x_2)\dd\mu_2(x_2) ,$$ and thus $\mu_1=\mu_2$. The converse is clear,
    but we detail for completeness. We show that $\PStheta(\mu, \mu) = 0$ for
    $\mu\in \mathcal{P}_2(\R^d)$: notice that $\mu_\theta[\mu, \mu] =
    Q_\theta\#\mu =: \mu_\theta$, take any $\gamma \in \Pi^*(\mu_\theta, \mu)$
    and introduce by disintegration $$\rho(\dd y, \dd x_1, \dd x_2) :=
    \mu_\theta(\dd y) \gamma^y(\dd x_1) \delta_{x_1 = x_2}(\dd x_1, \dd x_2).$$
    We have $\rho \in \Gamma(\mu_\theta, \mu, \mu)$, and for $\rho$-almost-every
    $(y, x_1, x_2) \in \R^{3d}$, we have $\|x_1-x_2\|_2^2 = 0$, hence
    $\PStheta(\mu, \mu) = 0$.
    
    Symmetry is immediate from the definition (\cref{eqn:S_theta}). As for the
    upper-bound, by \cref{prop:W_nu_inf_attained} we can take $\rho\in
    \Gamma(\mu_\theta[\mu_1, \mu_2], \mu_1, \mu_2)$ optimal for $\PStheta(\mu_1,
    \mu_2)$, and we have $\rho_{1, 2} \in \Pi(\mu_1, \mu_2)$ and thus:
    $$\PStheta^2(\mu_1, \mu_2) = \int_{\R^{2d}}\|x_1-x_2\|_2^2\dd\rho_{1, 2}(\dd
    x_1, \dd x_2) \geq \W_2^2(\mu_1, \mu_2). $$
    \vskip -20pt
\end{proof}
The triangle inequality is not satisfied in general, as shown in
\cref{ex:ce_S_theta_triangle}.
\begin{example}[$\PStheta$ does not verify the triangle inequality]
    \label{ex:ce_S_theta_triangle} 
    We represent the counter-example in \cref{fig:ce_S_theta_triangle}. Consider
    $\theta := (1, 0)$ and:
    $$x_1 := (-1, 0),\; x_2 := (1, 5),\;\mu_1 := \tfrac{1}{2}\delta_{x_1} +
    \tfrac{1}{2}\delta_{x_2},$$
    $$y_1 := (-1, 5),\; y_2 := (1, 0),\;\mu_2 := \tfrac{1}{2}\delta_{y_1} +
    \tfrac{1}{2}\delta_{y_2},$$
    $$z_1 := (0, 0),\; z_2 := (0, 5),\;\mu_3 := \tfrac{1}{2}\delta_{z_1} +
    \tfrac{1}{2}\delta_{z_2}.$$ First, we compute $\PStheta(\mu_1, \mu_2)$: we
    have
    $$u_1 := (-1, 0),\; u_2 := (1, 0),\; \mu_\theta[\mu_1, \mu_2] =
    \tfrac{1}{2}\delta_{u_1} + \tfrac{1}{2}\delta_{u_2},$$ and then we see that
    there are unique optimal plans between $\mu_\theta[\mu_1, \mu_2]$ and each
    $\mu_1, \mu_2$:
    \begin{align*}
        \Pi^*(\mu_\theta[\mu_1, \mu_2], \mu_1) &= \left\{\gamma_{121} :=
        \tfrac{1}{2}\delta_{u_1\otimes x_1} + \tfrac{1}{2}\delta_{u_2\otimes
        x_2}\right\},\\
        \Pi^*(\mu_\theta[\mu_1, \mu_2], \mu_2) &= \left\{\gamma_{122}
        := \tfrac{1}{2}\delta_{u_1\otimes y_1} + \tfrac{1}{2}\delta_{u_2\otimes
        y_2}\right\}.
    \end{align*}
    Using \cref{thm:W_nu_disintegration}, we compute:
    $$\PStheta(\mu_1, \mu_2) = \sqrt{\tfrac{1}{2}\|x_1-y_1\|_2^2 
    + \tfrac{1}{2}\|x_2-y_2\|_2^2} = 5. $$ We now turn to $\PStheta(\mu_1,
    \mu_3)$. This time, we have 
    $$v_1 := (-\tfrac{1}{2}, 0),\; v_2 := (\tfrac{1}{2}, 0),\; \mu_\theta[\mu_1,
    \mu_3] = \tfrac{1}{2}\delta_{v_1} + \tfrac{1}{2}\delta_{v_2}.$$ There is a
    unique optimal transport plan between $\mu_\theta[\mu_1, \mu_3]$ and
    $\mu_1$:
    $$\Pi^*(\mu_\theta[\mu_1, \mu_3], \mu_1) = \left\{\gamma_{131} :=
    \tfrac{1}{2}\delta_{v_1\otimes x_1} + \tfrac{1}{2}\delta_{v_2\otimes
    x_2}\right\}. $$ On the other hand, there are an infinite number of OT
    between $\mu_\theta[\mu_1, \mu_3]$ and $\mu_3$, which are convex
    combinations of two extremal plans (which correspond to the two permutations
    of $\{1, 2\}$):
    \begin{align*}
        \Pi^*(\mu_\theta[\mu_1, \mu_3], \mu_3) = \Bigl\{\gamma_{133}(t) :=
        &(1-t)\left(\tfrac{1}{2}\delta_{v_1\otimes z_1} +
        \tfrac{1}{2}\delta_{v_2\otimes z_2}\right) \\
        &+ t\left(\tfrac{1}{2}\delta_{v_2\otimes z_1} 
        + \tfrac{1}{2}\delta_{v_1\otimes z_2}\right),\; t\in [0, 1]\Bigr\}.
    \end{align*}
    Following \cref{thm:W_nu_disintegration}, we have
    $$\PStheta(\mu_1, \mu_3) = \underset{t\in[0, 1]}{\min}\
    \tfrac{1}{2}\W_2^2\left(\delta_{x_1}, \tfrac{1-t}{2}\delta_{z_1} +
    \tfrac{t}{2}\delta_{z_2}\right) + \tfrac{1}{2}\W_2^2\left(\delta_{x_2},
    \tfrac{1-t}{2}\delta_{z_2} + \tfrac{t}{2}\delta_{z_1}\right), $$ which is
    clearly minimal at $t=0$, yielding
    $$\PStheta(\mu_1, \mu_3) = \sqrt{\tfrac{1}{2}\|x_1 - z_1\|_2^2 +
    \tfrac{1}{2}\|x_2 - z_2\|_2^2} = 1, $$ and by symmetry, $\PStheta(\mu_2,
    \mu_3) = \PStheta(\mu_1, \mu_3) = 1$. We conclude that the triangle
    inequality does not hold:
    $$\PStheta(\mu_1, \mu_2) = 5 > \PStheta(\mu_2, \mu_3) + \PStheta(\mu_1,
    \mu_3) = 2. $$
\end{example}
\begin{figure}[H]
    \centering
    \includegraphics[width=0.8\textwidth]{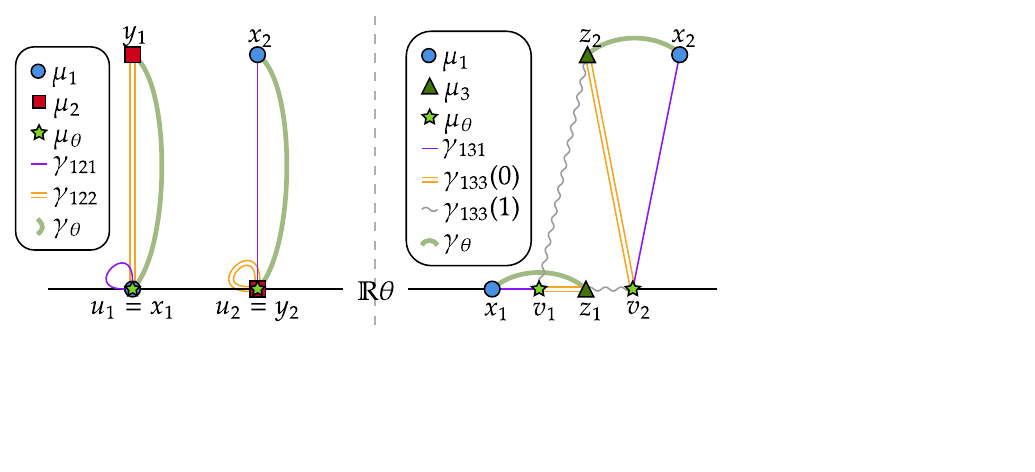}
    \caption{Counter-example from \cref{ex:ce_S_theta_triangle} to the triangle
    inequality for $\PStheta$. Left: illustration of the couplings for
    $\PStheta(\mu_1, \mu_2)$, with the optimal coupling $\gamma_\theta$ between
    $\mu_1$ and $\mu_2$ for $\PStheta(\mu_1, \mu_2)$ represented with thick
    green lines. Right: illustration of the couplings for $\PStheta(\mu_1,
    \mu_3)$. The optimal coupling $\gamma_\theta$ for $\PStheta(\mu_1, \mu_3)$
    corresponds to gluing $\gamma_{131}$ and $\gamma_{133}(0)$.}
    \label{fig:ce_S_theta_triangle}
\end{figure}

In the following, we show that $\PStheta$ is lower semi-continuous with respect
to the weak convergence of measures, along with a result on continuity of the
middle $\mu_\theta[\mu_1, \mu_2]$. We speak of continuity with respect to the
Euclidean topology on $\SS^{d-1}$, and the weak topology on
$\mathcal{P}_2(\R^d)$.

\begin{prop}\label{prop:PStheta_lsc} The map $(\theta, \mu_1, \mu_2) \longmapsto
    \mu_\theta[\mu_1, \mu_2]$ is continuous, and $(\theta, \mu_1,
    \mu_2)\longmapsto \PStheta(\mu_1, \mu_2)$ is lower semi-continuous.
\end{prop}
\begin{proof}
    Take measure sequences $\mu_1^{(n)} \xrightarrow[n\longrightarrow +
    \infty]{w} \mu_1 \in \mathcal{P}_2(\R^d)$ and $\mu_2^{(n)}
    \xrightarrow[n\longrightarrow]{w} \mu_2\in \mathcal{P}_2(\R^d)$ and a
    sequence of projections $\theta_n \xrightarrow[n\longrightarrow + \infty]{}
    \theta \in \SS^{d-1}$. By \cref{lemma:midpoints}, we have
    $$\mu_{\theta_n}[\mu_1^{(n)}, \mu_2^{(n)}] =
    \left[\left(\tfrac{1}{2}F_{P_{\theta_n}\#\mu_1^{(n)}}^{[-1]}+\tfrac{1}{2}F_
    {P_{\theta_n}\#\mu_2^{(n)}}^{[-1]}\right)\theta_n\right]\#\Leb_{[0, 1]} =:
    f_n\#\Leb_{[0, 1]},$$ and:
    $$\mu_\theta[\mu_1, \mu_2] =
    \left[\left(\tfrac{1}{2}F_{P_\theta\#\mu_1}^{[-1]}+\tfrac{1}{2}F_
    {P_\theta\#\mu_2}^{[-1]}\right)\theta\right]\#\Leb_{[0, 1]} =:
    f\#\Leb_{[0,1]}.$$ Let $i\in \{1, 2\}$ and $g_n :=
    F_{P_{\theta_n}\#\mu_i^{(n)}}^{[-1]}$, we show that $(g_n)$ converges
    pointwise $\Leb_{[0, 1]}$-almost-everywhere to $g :=
    F_{P_\theta\#\mu_i}^{[-1]}$. Since
    $P_{\theta_n}\#\mu_i^{(n)}\xrightarrow[n\longrightarrow
    +\infty]{w}P_\theta\#\mu_i$, we have by \cite[Lemma 21.2]{van2000asymptotic}
    that for all $p\in [0, 1]$ such that $g$ is continuous at $p$, $g_n(p)
    \xrightarrow[n\longrightarrow +\infty]{} g(p).$ Since $g$ is non-decreasing,
    it is continuous $\Leb_{[0, 1]}$-almost-everywhere, and thus the convergence
    happens $\Leb_{[0, 1]}$-almost-everywhere. Having shown that $f_n$ converges
    pointwise to $f$ $\Leb_{[0, 1]}$-almost-everywhere, we deduce that 
    $$\mu_{\theta_n}[\mu_1^{(n)}, \mu_2^{(n)}] \xrightarrow[n\longrightarrow
    +\infty]{w} \mu_\theta[\mu_1, \mu_2].$$ By \cref{prop:W_nu_lsc}, we deduce
    that $(\theta, \mu_1, \mu_2)\longmapsto \PStheta(\mu_1, \mu_2)$ is lower
    semi-continuous.
\end{proof}

We show in \cref{ex:ce_swgg_weak_convergence} that full continuity does not
hold.

\begin{example}[$\PStheta$ is not continuous with respect to the weak
    convergence]\label{ex:ce_swgg_weak_convergence} Consider
    $$x_n := (-1-2^{-n}, 5),\; x' := (-1, 0),\;\mu_n := \tfrac{1}{2}\delta_{x_n}
    + \tfrac{1}{2}\delta_{x'},$$
    $$y := (1, 0),\; y' := (2, 5),\;\nu := \tfrac{1}{2}\delta_{y} +
    \tfrac{1}{2}\delta_{y'}.$$ Obviously, $\mu_n \xrightarrow[n\longrightarrow
    +\infty]{w} \mu$, with $\mu = \tfrac{1}{2}\delta_{(-1, 5)} +
    \tfrac{1}{2}\delta_{x'}$. For $n\in \N$, we compute easily that:
    $$\PStheta^2(\mu_n, \nu) = \tfrac{1}{2}\|x_n - y'\|_2^2 +
    \tfrac{1}{2}\|x'-y\|_2^2 = 36 + 3.2^{-n} + \tfrac{4^{-n}}{2}
    \xrightarrow[n\longrightarrow +\infty]{} 36.$$ The limit does not coincide
    with $\PStheta^2(\mu, \nu) = \tfrac{13}{2}$. We summarise this
    counter-example in \cref{fig:ce_swgg_weak_convergence}.
\end{example}

\begin{figure}[H]
    \centering
    \includegraphics[width=0.4\textwidth]{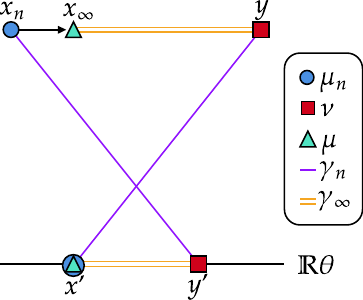}
    \caption{Representation of \cref{ex:ce_swgg_weak_convergence}, showing a
    counter-example to the continuity of $\PStheta^2(\cdot, \nu)$ with respect
    to the weak convergence of measures. At each $n\in \N$, the coupling
    $\gamma_n$ associated to $\PStheta^2(\mu_n, \nu)$ between $\mu_n$ and $\nu$
    (represented by purple lines) is imposed to assign $x_n$ to $y'$ and $x'$ to
    $y$. However, the coupling $\gamma_\infty$ associated to $\PStheta^2(\mu,
    \nu)$ represented by orange double lines has more freedom due to the fact
    that $P_\theta x_\infty = P_\theta x'$, and therefore can perform the less
    costly assignment of $x_\infty$ to $y$ and $x'$ to $y'$.}
    \label{fig:ce_swgg_weak_convergence}
\end{figure}

%% file: sections/equality_CWtheta.tex
\section{Correspondence of Pivot-Sliced and a Constrained Wasserstein Discrepancy}\label{sec:equality_CWtheta}

In this section, we will compare the quantity $\PStheta(\mu_1, \mu_2)$ defined
in \cref{eqn:S_theta} with a particular lifting of the 1D sliced plan between
$\mu_1$ and $\mu_2$. Namely, we will compare the two quantities:
\begin{equation}\label{eqn:S_theta_vs_lift}
    \begin{split}
        \PStheta^2(\mu_1, \mu_2) &:= 
        \underset{\rho \in \Gamma(\mu_\theta[\mu_1, \mu_2], \mu_1, \mu_2)}
        {\min}\ \int_{\R^{3d}}\|x_1-x_2\|_2^2\dd\rho(y, x_1, x_2) \\
        \eqquestion \CWtheta^2(\mu_1, \mu_2) &:= 
        \underset{\substack{\omega \in \Pi(\mu_1, \mu_2) \\ 
        (P_\theta, P_\theta)\#\omega = \pi_\theta[\mu_1, \mu_2]}}{\min}\ 
        \int_{\R^{2d}}\|x_1-x_2\|_2^2\dd\omega(x_1, x_2),
    \end{split}
\end{equation}
where $\pi_\theta[\mu_1, \mu_2]$ is the unique optimal plan between
$P_\theta\#\mu_1$ and $P_\theta\#\mu_2$. We introduce the following notation for
the set of admissible plans $\omega$ for $\CWtheta(\mu_1, \mu_2)$:
\begin{equation}\label{eqn:def_Omega_theta}
    \Omega_\theta(\mu_1, \mu_2) := \left\{\omega \in \Pi(\mu_1, \mu_2) :
    (P_\theta, P_\theta)\#\omega = \pi_\theta[\mu_1, \mu_2]\right\}.    
\end{equation}
To see that $\Omega_\theta(\mu_1, \mu_2)$ is non-empty, we construct an element
by disintegration. Let $\pi_\theta := \pi_\theta[\mu_1, \mu_2] \in
\mathcal{P}_2(\R^2)$, and consider the coupling:
\[
\omega(\dd x_1, \dd x_2) := \pi_\theta(P_\theta \dd x_1, P_\theta \dd x_2)\mu_1^{P_\theta x_1}(\dd x_2)\mu_2^{P_\theta x_2}(\dd x_1).
\] 
It is clear that $\omega \in \Omega_\theta(\mu_1, \mu_2)$ (note that this
coupling will be studied in further detail in \cref{sec:expected_sliced} within
another context). Note that, by compactness of $\Pi(\mu_1, \mu_2)$, continuity of
$\omega \longmapsto (P_\theta, P_\theta)\#\omega$ and lower semi-continuity of
$J := \omega \longmapsto \int_{\R^{2d}}\|\cdot-\cdot\|_2^2\dd\omega$, the
infimum in $\CWtheta^2$ is attained.

To draw a correspondence between $\PStheta$ and $\CWtheta$, we will compare
their optimisation sets, and to this end, we introduce the set $\Gamma_{\theta,
1, 2}\subset \Pi(\mu_1, \mu_2)$ defined as:
\begin{equation}\label{def:Gamma_theta_12}
    \Gamma_{\theta, 1, 2}(\mu_1, \mu_2) := 
    \left\{\rho_{1, 2} : \rho \in \Gamma(\mu_\theta[\mu_1, \mu_2], 
    \mu_1, \mu_2)\right\}.
\end{equation}
We have, by definition:
\begin{equation}\label{eqn:opt_sets_W_theta_S_theta}
    \CWtheta^2(\mu_1, \mu_2) = \underset{\omega \in \Omega_\theta(\mu_1,
    \mu_2)}{\min}\ J(\omega);\quad \PStheta^2(\mu_1, \mu_2) = 
    \underset{\gamma \in \Gamma_{\theta, 1, 2}(\mu_1, \mu_2)}{\min}\ J(\gamma).
\end{equation}

\subsection{First Inequality: \texorpdfstring{$\PStheta \leq \CWtheta$}{PS leq CW}}

To prove a first inequality between $\PStheta$ and $\CWtheta$, we will show that
$\Omega_\theta(\mu_1, \mu_2) \subset \Gamma_{\theta, 1, 2}(\mu_1, \mu_2)$ (these
sets are defined in \cref{eqn:def_Omega_theta} and \cref{def:Gamma_theta_12}).
We start with two Lemmas on Wasserstein means. The first result provides an
explicit optimal coupling between a Wasserstein mean and the two measures. Note
that for legibility, we chose to present and prove the results using random
variables rather than measures.
\begin{lemma}\label{lemma:plan_middle} Let $\mu_1, \mu_2 \in
    \mathcal{P}_2(\R^d)$ and an optimal coupling $\gamma \in \Pi^*(\mu_1,
    \mu_2)$. 
    
    Then $\mu_\half := (\tfrac{1}{2}P_1 + \tfrac{1}{2}P_2)\#\gamma = \Law_{(X_1,
    X_2)\sim\gamma}\left[\tfrac{X_1+X_2}{2}\right]$ belongs to $\M(\mu_1,
    \mu_2)$, and furthermore the coupling $\gamma_\half := \Law_{(X_1, X_2)\sim
    \gamma}\left[\left(\tfrac{X_1+X_2}{2}, X_1\right)\right]$ belongs to
    $\Pi^*(\mu_\half, \mu_1)$.
\end{lemma}
\begin{proof}
    By \cref{prop:wass_mean_geodesic} we have $\mu_\half \in M(\mu_1, \mu_2)$,
    and $\W_2^2(\mu_\half, \mu_1) = \tfrac{1}{4}\W_2^2(\mu_1, \mu_2)$. We
    compute:
    \begin{align*}
        \W_2^2(\mu_\half, \mu_1) 
        &\leq \underset{(X_1, X_2)\sim\gamma}{\mathbb{E}}
        \left[\left\|\tfrac{X_1+X_2}{2} - X_1\right\|_2^2\right] \\
        &= \tfrac{1}{4}\underset{(X_1, X_2)\sim\gamma}{\mathbb{E}}
        \left[\left\|X_1-X_2\right\|_2^2\right] \\
        &=\tfrac{1}{4}\W_2^2(\mu_1, \mu_2),
    \end{align*}
    showing optimality of the coupling $\gamma_\half$, since by
    \cref{prop:wass_mean_geodesic}, $\W_2^2(\mu_\half, \mu_1) =
    \tfrac{1}{4}\W_2^2(\mu_1, \mu_2)$.
\end{proof}
Note that \cref{lemma:plan_middle} is also a consequence of \cite[Lemma
7.2.1]{ambrosio2005gradient} (which states a stronger result with more abstract
language). The following second lemma relates an admissible plan $\omega\in
\Omega_\theta(\mu_1, \mu_2)$ for $\CWtheta(\mu_1, \mu_2)$ to an explicit optimal
coupling between the projected middle $\mu_\theta[\mu_1, \mu_2]$ and the
measures $\mu_1, \mu_2$, which will be useful to construct an admissible 3-plan
for $\PStheta(\mu_1, \mu_2)$.
\begin{lemma}\label{lemma:plan_projected_middle} Let $\omega \in \Pi(\mu_1,
    \mu_2)$ such that $(P_\theta, P_\theta)\#\omega = \Pi^*(P_\theta\#\mu_1,
    P_\theta\#\mu_2)$. Let $(X_1, X_2)\sim \omega$ and $Y := \tfrac{P_\theta X_1
    + P_\theta X_2}{2}\theta$. Then $\Law[Y] = \mu_\theta[\mu_1, \mu_2]$, and
    $\Law[(Y, X_i)] \in \Pi^*(\mu_\theta[\mu_1, \mu_2], \mu_i),\; i\in \{1,
    2\}$.
\end{lemma}
\begin{proof}
    First, we apply \cref{lemma:plan_middle} to the optimal coupling $(P_\theta
    X_1, P_\theta X_2)$, which shows that $\Law[P_\theta Y] =
    P_\theta\#\mu_\theta[\mu_1, \mu_2]$, thus that $\Law[Y] = \mu_\theta[\mu_1,
    \mu_2]$. \cref{lemma:plan_middle} also shows that $(P_\theta Y, P_\theta
    X_1)$ is the optimal coupling between $P_\theta\#\mu_\theta[\mu_1, \mu_2]$
    and $P_\theta\#\mu_1$. Then by \cref{eqn:semi_1D_ot_plan} in
    \cref{prop:semi_1D_ot}, it follows that $\Law[(Y, X_1)] \in
    \Pi^*(\mu_\theta[\mu_1, \mu_2], \mu_1)$, and the same reasoning applies to
    $\Law[(Y, X_2)] \in \Pi^*(\mu_\theta[\mu_1, \mu_2], \mu_2)$.
\end{proof}
Using \cref{lemma:plan_middle} and \cref{lemma:plan_projected_middle}, we can
now show an inequality between $\PStheta$ and $\CWtheta$:

\begin{prop}\label{prop:S_theta_smaller_lift} Let $\mu_1, \mu_2 \in
    \mathcal{P}_2(\R^d)$, and $\theta \in \SS^{d-1}$. The two sets defined in
    \cref{eqn:def_Omega_theta} and \cref{def:Gamma_theta_12} verify
    $\Omega_\theta(\mu_1, \mu_2) \subset \Gamma_{\theta, 1, 2}(\mu_1, \mu_2)$,
    and the two quantities defined in \cref{eqn:S_theta_vs_lift} verify
    $\PStheta(\mu_1, \mu_2) \leq \CWtheta(\mu_1, \mu_2)$.
\end{prop}
\begin{proof}
    Let $\omega \in \Pi(\mu_1, \mu_2)$ such that $(P_\theta, P_\theta)\#\omega =
    \Pi^*(P_\theta\#\mu_1, P_\theta\#\mu_2)$ optimal for $\CWtheta(\mu_1,
    \mu_2)$. Consider $(X_1, X_2)\sim\omega$ and $Y := \tfrac{P_\theta X_1 +
    P_\theta X_2}{2}\theta$. By \cref{lemma:plan_projected_middle}, we have
    $\Law[Y] = \mu_\theta[\mu_1, \mu_2]$, and $\Law[(Y, X_i)] \in
    \Pi^*(\mu_\theta[\mu_1, \mu_2], \mu_i),\; i\in \{1, 2\}$. By definition, the
    3-plan $\rho$ defined by $\rho := \Law[(Y, X_1, X_2)]$ belongs to
    $\Gamma(\mu_\theta[\mu_1, \mu_2], \mu_1, \mu_2)$, thus $\omega \in
    \Gamma_{\theta, 1, 2}(\mu_1, \mu_2)$. We compute:
    $$\PStheta^2(\mu_1, \mu_2) \leq \int_{\R^{3d}}\|x_1-x_2\|_2^2\dd\rho(y, x_1,
    x_2) = \int_{\R^{2d}}\|x_2-x_2\|_2^2\dd\omega(x_1, x_2) = \CWtheta^2(\mu_1,
    \mu_2).$$
\end{proof}

\subsection{Converse Inequality: \texorpdfstring{$\PStheta \geq \CWtheta$}{PS leq CW}}

To show the converse inequality \[\CWtheta(\mu_1, \mu_2)~\leq~\PStheta(\mu_1,
\mu_2),\] we will use more technical arguments from \cite[Lemma
7.2.1]{ambrosio2005gradient}, which will show that (denoting $\mu_\theta :=
\mu_\theta[\mu_1, \mu_2],$) for $i\in \{1, 2\}$, the unique optimal plan $\pi_i$
between $P_\theta\#\mu_\theta$ and $P_\theta\#\mu_i$ is induced by a transport
map $T_i$, i.e. $\pi_i = (I, T_i)\#P_\theta\#\mu_\theta$. This is a consequence
of the fact that $\mu_\theta$ is chosen as the middle of the geodesic between
$P_\theta\#\mu_1$ and $P_\theta\#\mu_2$, and remarkably holds without atomless
assumptions on the $P_\theta\#\mu_i$.

\begin{theorem}\label{thm:S_theta_equals_W_theta} Let $\mu_1, \mu_2 \in
    \mathcal{P}_2(\R^d)$ and $\theta \in \SS^{d-1}$. Then the two sets defined
    in \cref{eqn:def_Omega_theta} and \cref{def:Gamma_theta_12} verify
    $\Gamma_{\theta, 1, 2}(\mu_1, \mu_2)= \Omega_\theta(\mu_1, \mu_2)$,
    and the two quantities defined in \cref{eqn:S_theta_vs_lift} verify
    $\PStheta(\mu_1, \mu_2) = \CWtheta(\mu_1, \mu_2)$.    
\end{theorem}
\begin{proof}
    We have already shown that $\PStheta(\mu_1, \mu_2) \leq \CWtheta(\mu_1,
    \mu_2)$ in \cref{prop:S_theta_smaller_lift}. We now show that
    $\Gamma_{\theta, 1, 2}(\mu_1, \mu_2)\subset \Omega_\theta(\mu_1, \mu_2)$ (we
    write $\mu_\theta := \mu_\theta[\mu_1, \mu_2],$ and $\pi_\theta$ the unique
    element of $\Pi^*(P_\theta\#\mu_1, P_\theta\#\mu_2)$). Let $\rho \in
    \Gamma(\mu_\theta, \mu_1, \mu_2)$. We introduce $\eta := (P_\theta,
    P_\theta, P_\theta)\#\rho \in \Pi(\nu_\half, \nu_1, \nu_2)$, where for
    convenience we write $\nu_i := P_\theta\#\mu_i$ for $i\in \{1, 2\}$ and
    $\nu_\half := P_\theta\#\mu_\theta$. By \cref{eqn:semi_1D_ot_plan} in
    \cref{prop:semi_1D_ot}, we have $\eta_{0, i} \in \Pi^*(\nu_\half, \nu_i)$. 
    
    We now write $\eta$ using OT maps. By \cite[Lemma
    7.2.1]{ambrosio2005gradient}, since $\nu_\half = \M(\nu_1, \nu_2)$ (i.e. it
    is the middle of the constant-speed geodesic between $\nu_1$ and $\nu_2$,
    which is unique since the measures are one-dimensional), for $i\in \{1, 2\}$
    the transport plan $\eta_{0, i} \in \Pi^*(\nu_\half, \nu_i)$ is induced by a
    non-decreasing transport map $T_i$, which is to say that $\eta_{0, i} = (I,
    T_i)\#\nu_\half$. It follows that for $\eta$-almost-every $(t, s_1, s_2) \in
    \R^3$, we have $s_1 = T_1(t)$ and $s_2 = T_2(t)$.

    We now verify that $\eta_{1, 2} \in \Pi^*(\nu_1, \nu_2)$ using the cyclical
    monotonicity criterion: Let $(s_1, s_2), (s_1', s_2') \in \supp \eta_{1, 2}$
    such that $s_1 < s_1'$. Our earlier considerations on $\eta$ show that there
    exists $t, t' \in \R$ verifying $s_1 = T_1(t),\; s_2 = T_2(t)$ and $s_1' =
    T_1(t'),\; s_2' = T_2(t')$. Since $s_1 = T_1(t) < T_1(t') = s_1'$ and $T_1$
    is non-decreasing, we deduce $t<t'$. Now since $T_2$ is non-decreasing,
    $t<t'$ implies that $s_2 = T_2(t) \leq T_2(t') = s_2'$. We have shown the
    following property of $\eta_{1, 2}$:
    \begin{equation}\label{eqn:eta_12_comonotone}
        \forall (s_1, s_2), (s_1', s_2') \in \supp \eta_{1, 2},\; 
        s_1 < s_1' \implies s_2 \leq s_2'.
    \end{equation}
    By \cite[Lemma 2.8]{santambrogio2015optimal}, \cref{eqn:eta_12_comonotone}
    implies that $\eta_{1, 2}$ is the co-monotone plan between $\nu_1$ and
    $\nu_2$, and by \cite[Theorem 2.9]{santambrogio2015optimal}, we conclude
    that $\eta_{1, 2} \in \Pi^*(\nu_1, \nu_2)$.

    Having shown that $\eta_{1, 2} \in \Pi^*(\nu_1, \nu_2)$, we conclude that
    \[(P_\theta, P_\theta)\#\rho_{1, 2} \in \Pi^*(P_\theta\#\mu_1,
    P_\theta\#\mu_2),\] and by definition we conclude $\rho_{1, 2} \in
    \Omega_\theta(\mu_1, \mu_2)$, which shows the inclusion \[\Gamma_{\theta, 1,
    2}(\mu_1, \mu_2)\subset \Omega_\theta(\mu_1, \mu_2),\] and equality is
    obtained by combining with \cref{prop:S_theta_smaller_lift}. By
    \cref{eqn:opt_sets_W_theta_S_theta} we conclude that $\CWtheta(\mu_1, \mu_2)
    = \PStheta(\mu_1, \mu_2)$.
\end{proof}

\subsection{Triangle Inequality for \texorpdfstring{$\PStheta$}{PS} for Projection-Atomless Measures}

Using \cref{thm:S_theta_equals_W_theta} and the following technical lemma on
one-dimensional 3-plans, we will show that $\PStheta$ is a metric on the set of
probability measures whose projections on $\R\theta$ are atomless. We show that
in the one-dimensional atomless case, 3-plans with two optimal bi-marginals
automatically verify that \textit{all} their bi-marginals are optimal. In terms
of random variables, \cref{lemma:1d_3plan_bimarginals} states that if $(X_1,
X_2)$ and $(X_1, X_3)$ are optimal couplings, then so is $(X_2, X_3)$.
\begin{lemma}\label{lemma:1d_3plan_bimarginals} Let $\mu_1, \mu_2, \mu_3 \in
    \mathcal{P}_2(\R)$ such that $\mu_1$ and $\mu_2$ are atomless, and let $\rho
    \in \Pi(\mu_1, \mu_2, \mu_3)$ be a 3-plan such that $\rho_{1, 2} \in
    \Pi^*(\mu_1, \mu_2)$ and $\rho_{1, 3} \in \Pi^*(\mu_1, \mu_3)$. Then
    $\rho_{2, 3} \in \Pi^*(\mu_2, \mu_3)$.
\end{lemma}
\begin{proof}
    For $i\in \{1, 2, 3\}$, introduce the c.d.f. $F_i$ of $\mu_i$. Take $(X_1,
    X_2, X_3)\sim \rho$. By \cite[Theorem 2.9]{santambrogio2015optimal}, since
    $\mu_1$ is atomless and $(X_1, X_3)$ is optimal, we have almost-surely $X_3
    = F_3^{[-1]}\circ F_1(X_1)$. Likewise, since $\mu_2$ is atomless and $(X_2,
    X_1)$ is optimal, we have almost-surely $X_1 = F_1^{[-1]}\circ F_2(X_2)$.
    Combining these equalities yields almost-surely:
    $$X_3 = F_3^{[-1]}\circ F_1(X_1) = F_3^{[-1]}\circ F_1 \circ F_1^{[-1]}\circ
    F_2(X_2). $$ By continuity of $F_1$ (since $\mu_1$ is atomless) and defining
    $F_1(-\infty) := 0$ and $F_1(+\infty):=1$, we have
    $F_1(\R)\cup\{F_1(-\infty)\}\cup\{F_1(+\infty)\} = [0, 1]$, allowing us to
    apply \cite[Proposition 2.3 item 4)]{embrechts2013note}, which yields $F_1
    \circ F_1^{[-1]} = I_{[0, 1]}$.

    We have shown that almost-surely $X_3 = F_3^{[-1]}\circ F_2(X_2)$, which
    shows by \cite[Theorem 2.9]{santambrogio2015optimal} that $(X_2, X_3)$ is
    the optimal coupling between $\mu_2$ and $\mu_3$.
\end{proof}
We can now show that $\PStheta$ verifies the triangle inequality on the set of
probability measures with atomless projections. Combining this statement with
\cref{prop:S_theta_semi_metric} shows that $\PStheta$ is a metric this subset of
$\mathcal{P}_2(\R^d)$.
\begin{prop}\label{prop:S_theta_metric_atomless_projections} Let $\theta\in
    \SS^{d-1}$ and $\mathcal{P}_{2, a}(\R^d, \theta)$ be the set of probability
    measures $\mu \in \mathcal{P}_2(\R^d)$ such that $P_\theta\#\mu$ is
    atomless. The quantity $\PStheta$ is a metric on $\mathcal{P}_{2, a}(\R^d,
    \theta)$.    
\end{prop}
\begin{proof}
    First, by \cref{prop:S_theta_semi_metric}, it only remains to show the
    triangle inequality. Let $\mu_1, \mu_2, \mu_3 \in \mathcal{P}_{2, a}(\R^d,
    \theta)$ and let $\omega_{1, 2} \in \Omega_\theta(\mu_1, \mu_2)$ be optimal
    for $\CWtheta(\mu_1, \mu_2)$, and likewise let $\omega_{2, 3} \in
    \Omega_\theta(\mu_2, \mu_3)$ be optimal for $\CWtheta(\mu_2, \mu_3)$. We
    apply the standard gluing technique (see for example \cite[Lemma
    5.5]{santambrogio2015optimal}): the second marginal of $\omega_{1, 2}$ and
    the first marginal of $\omega_{2, 3}$ are both $\mu_2$, hence we can write
    their disintegrations with respect to $\mu_2$ as:
    $$\omega_{1, 2}(\dd x_1, \dd x_2) = \mu_2(\dd x_2)\omega_{1, 2}^{x_2}(\dd
    x_1),\quad \omega_{2, 3}(\dd x_2, \dd x_3) = \mu_2(\dd x_2)\omega_{2,
    3}^{x_2}(\dd x_3).$$ We now introduce the ``composition'' 3-plan $\rho \in
    \Pi(\mu_1, \mu_2, \mu_3)$ as:
    $$\rho(\dd x_1, \dd x_2, \dd x_3) := \mu_2(\dd x_2) \omega_{1, 2}^{x_2}(\dd
    x_1) \omega_{2, 3}^{x_2}(\dd x_3).$$ Writing $\rho_\theta := (P_\theta,
    P_\theta, P_\theta)\#\rho$, by definition of $\omega_{1, 2}$ and $\omega_{2,
    3}$, we have $[\rho_\theta]_{1, 2} = \pi_\theta[\mu_1, \mu_2]$ and
    $[\rho_\theta]_{2, 3} = \pi_\theta[\mu_2, \mu_3]$. By
    \cref{lemma:1d_3plan_bimarginals}, we deduce that $\rho_{1, 3} =
    \pi_\theta[\mu_1, \mu_3]$, since each $P_\theta\#\mu_i$ is atomless. This
    shows that $\rho_{1, 3} \in \Omega_\theta(\mu_1, \mu_2)$. Denoting $\phi_i
    := (x_1, x_2, x_3) \longmapsto x_i$ for $i \in \{1, 2, 3\}$, we have:
    \begin{align*}
        \CWtheta(\mu_1, \mu_3) &\leq \|\phi_1 - \phi_3\|_{L^2(\rho_{1, 3})} 
        = \|\phi_1 - \phi_3\|_{L^2(\rho)} \\
        &\leq \|\phi_1 -\phi_2\|_{L^2(\rho)} 
        + \|\phi_2 -\phi_3\|_{L^2(\rho)} \\ 
        &= \|\phi_1 -\phi_2\|_{L^2(\omega_{1,2})} 
        + \|\phi_2 -\phi_3\|_{L^2(\omega_{2,3})} \\
        &= \CWtheta(\mu_1, \mu_2) + \CWtheta(\mu_2, \mu_3).
    \end{align*}
    Using \cref{thm:S_theta_equals_W_theta}, we deduce the triangle inequality
    for $\PStheta$.
\end{proof}

%% file: sections/pivot_sliced_discrete.tex
\section{A Monge Formulation of \texorpdfstring{$\PStheta$}{Stheta} Between
Point Clouds}\label{sec:pivot_sliced_discrete}

\subsection{The Case of Non-Ambiguous Projections}

A direct consequence of \cref{thm:W_nu_disintegration} is that, in the case of
point clouds with non-ambiguous projections, the computation of $\PStheta$ can
be done simply by sorting the projections and taking the associated plan between
the projected measures.
\begin{corollary}\label{cor:S_theta_SWGG_as} Let $(x_1, \cdots, x_n) \in
    (\R^d)^n$ and $(y_1, \cdots, y_n) \in (\R^d)^n$, and $\theta \in \SS^{d-1}$
    such that the families $(P_\theta x_i)_i$ and $(P_\theta y_i)_i$ are
    injective. Then for the measures $\mu_1 := \tfrac{1}{n}\sum_{i=1}^n
    \delta_{x_i},\; \mu_2 := \tfrac{1}{n}\sum_{i=1}^n \delta_{y_i}, $ it holds
    $$\PStheta^2(\mu_1, \mu_2) = \cfrac{1}{n}\sum_{i=1}^n\|x_{\sigma_\theta(i)}
    - y_{\tau_\theta(i)}\|_2^2,$$ where $\PStheta$ is introduced in
    \cref{def:S_theta} and where $\sigma_\theta, \tau_\theta$ are the (unique)
    permutations sorting $(P_\theta x_i)$ and $(P_\theta y_i)$. For injective
    families $(x_i)$ and $(y_i)$, the injectivity assumptions holds for
    $\mathcal{U}(\SS^{d-1})$-almost-every $\theta \in \SS^{d-1}$.
\end{corollary}

\begin{proof}
    We begin under the injectivity assumptions, which allows us to define
    $\sigma_\theta, \tau_\theta$ as the (unique) permutations sorting $(P_\theta
    x_i)$ and $(P_\theta y_i)$ respectively, and for $i\in \llbracket 1, n
    \rrbracket,$ let $z_i := \tfrac{1}{2}P_\theta (x_{\sigma_\theta(i)} +
    y_{\tau_\theta(i)})$. We remark that the family $(z_i)$ is increasing by
    construction (we provide further details at the end of the proof for
    almost-sure injectivity) and denote $\nu := \tfrac{1}{n}\sum_i\delta_{z_i}$.
    Let $\gamma_1 \in \Pi^*(\nu, \mu_1)$, and write
    $$\gamma_1 = \sum_{i,j}A_{i,j}\delta_{(z_i, x_{\sigma_\theta(j)})}. $$ By
    \cref{prop:semi_1D_ot} and the injectivity assumptions, we have \[(P_\theta,
    P_\theta)\#\gamma_1 = \tfrac{1}{n}\sum_i \delta_{(P_\theta z_i, P_\theta
    x_{\sigma_\theta(i)})},\] and thus injectivity allows us to identify the
    coefficients $A_{i,j}$, yielding \[\gamma_1 = \tfrac{1}{n} \sum_i
    \delta_{(z_i, x_{\sigma_\theta(i)})},\] and in particular, for any $i\in
    \llbracket 1, n \rrbracket,\; \gamma_1^{z_i} =
    \delta_{x_{\sigma_\theta(i)}}$. The same reasoning applies to $\gamma_2 \in
    \Pi^*(\nu, \mu_2)$, and thus \cref{thm:W_nu_disintegration} yields
    $$\PStheta^2(\mu_1, \mu_2) = \W_{\nu}^2(\mu_1, \mu_2) =
    \cfrac{1}{n}\sum_{i=1}^n\|x_{\sigma_\theta(i)} - y_{\tau_\theta(i)}\|_2^2$$
    Regarding the almost-sure injectivity claim, assume now that the families
    $(x_i)$ and $(y_i)$ are injective, and take $\theta \sim
    \mathcal{U}(\SS^{d-1})$. Then $(P_\theta x_i)$ is almost-surely injective,
    since $\P(P_\theta x_i = P_\theta x_j) = \P(\theta \in (x_i - x_j)^\top)$.
    The same reasoning applies to $(P_\theta y_i)$, and the injectivity of
    $(z_i)$ comes from the fact that almost-surely, for $i<j$, we have $P_\theta
    x_{\sigma_\theta(i)} < P_\theta x_{\sigma_\theta(j)}$, and $P_\theta
    y_{\tau_\theta(i)} < P_\theta x_{\tau_\theta(j)}$, hence by sum $z_i < z_j$,
    almost-surely.
\end{proof}

\subsection{Problem Formulation and Reduction to Sorted Projections}

The natural question that arises is the impact of projection ambiguity, i.e.
non-injectivity of $(P_\theta x_i)$ or $(P_\theta y_j)$. In this section, we
will start from the equality $\PStheta = \CWtheta$ from
\cref{thm:S_theta_equals_W_theta}, to provide the following Monge formulation of
$\PStheta$ between point clouds (without injectivity assumptions), that we will
prove in \cref{thm:CW_theta_point_clouds}:
\begin{equation*}
    \CWtheta^2\left(\frac{1}{n}\sum_{i=1}^n \delta_{x_i}, 
    \frac{1}{n}\sum_{j=1}^n \delta_{y_j}\right) = 
    \underset{(\sigma, \tau) \in \Perm_\theta(X, Y)}{\min}\ 
    \frac{1}{n}\sum_{i=1}^n \|x_{\sigma(i)} - y_{\tau(i)}\|_2^2,
\end{equation*}
where $\Perm_\theta(X, Y)$ is the set of pairs of permutations $(\sigma, \tau)$
such that $\sigma$ sorts $(P_\theta x_i)_{i=1}^n$ and $\tau$ sorts $(P_\theta
y_i)_{i=1}^n$, for given $X := (x_1, \cdots, x_n) \in \R^{n\times d}$ and $Y :=
(y_1, \cdots, y_n) \in \R^{n\times d}$:
\begin{equation}\label{eqn:def_Sigma_theta}
    \begin{split}
        \Perm_\theta(X, Y) &:= \Perm_\theta(X)\times\Perm_\theta(Y), \\ 
        \Perm_\theta(X) &:= \left\{\sigma\in \Perm_n : 
        \forall i \in \llbracket 1, n-1 \rrbracket,\; 
        P_\theta x_{\sigma(i)} \leq P_\theta x_{\sigma(i+1)},\right\},
    \end{split}
\end{equation}
with an analogous definition for $\Perm_\theta(Y)$. We will reduce to the case
where the identity permutation sorts the projections $(P_\theta x_i)_{i=1}^n$
and $(P_\theta y_i)_{i=1}^n$, which will greatly simplify notation and proofs.
The following Lemma states that re-labelling the points does not change the
value of $\CWtheta$ and of the Monge formulation.

\begin{lemma}\label{lemma:CV_theta_monge_invariant_permutation} Let $X,Y\in
    \R^{n\times d}$ and $\sigma_0, \tau_0 \in \Perm_n$. Denote by $X\circ
    \sigma_0 := (x_{\sigma_0(1)}, \cdots, x_{\sigma_0(n)})_{i=1}^n$ and likewise
    $Y\circ \tau_0 := (y_{\tau_0(1)}, \cdots, y_{\tau_0(n)})_{i=1}^n$. Then we
    have:
    $$\CWtheta^2\left(\frac{1}{n}\sum_{i=1}^n \delta_{x_i},
    \frac{1}{n}\sum_{j=1}^n \delta_{y_j}\right) =
    \CWtheta^2\left(\frac{1}{n}\sum_{i=1}^n \delta_{x_{\sigma_0(i)}},
    \frac{1}{n}\sum_{j=1}^n \delta_{y_{\tau_0(j)}}\right), $$ and for the Monge
    formulation, we have the following cost equality:
    \begin{equation}\label{eqn:monge_permutation_invariance}
        \underset{(\sigma, \tau) \in \Perm_\theta(X, Y)}{\min}\
        \frac{1}{n}\sum_{i=1}^n \|x_{\sigma(i)} - y_{\tau(i)}\|_2^2 = 
        \underset{(\sigma, \tau) \in \Perm_\theta(X\circ\sigma_0, Y\circ\tau_0)}
        {\min}\ \frac{1}{n}\sum_{i=1}^n \|x_{\sigma_0\circ\sigma(i)} -
        y_{\tau_0\circ\tau(i)}\|_2^2.
    \end{equation}
\end{lemma}
\begin{proof}
    The first equality is simply a consequence of the equality between measures:
    $$\frac{1}{n}\sum_{i=1}^n \delta_{x_i} = \frac{1}{n}\sum_{i=1}^n
    \delta_{x_{\sigma_0(i)}},\quad \frac{1}{n}\sum_{j=1}^n \delta_{y_j} =
    \frac{1}{n}\sum_{j=1}^n \delta_{y_{\tau_0(j)}}. $$ For the second equality,
    notice that a permutation $\sigma \in \Perm_n$ sorts $(P_\theta
    x_{\sigma_0(i)})_{i=1}^n$ if and only if $P_\theta
    x_{\sigma_0\circ\sigma(1)} \leq \cdots \leq P_\theta
    x_{\sigma_0\circ\sigma(n)}$ if and only if $\sigma_0\circ\sigma$ sorts
    $(P_\theta x_i)_{i=1}^n$, thus we obtain:
    $$ \Perm_\theta(X\circ\sigma_0, Y\circ\tau_0) =
    \left\{(\sigma_0^{-1}\circ\sigma, \tau_0^{-1}\circ\tau), \; (\sigma,
    \tau)\in \Perm_\theta(X, Y)\right\}. $$
    \cref{eqn:monge_permutation_invariance} follows by change of variables.
\end{proof}

Thanks to \cref{lemma:CV_theta_monge_invariant_permutation}, we can assume
without loss of generality (for the cost values) that the identity permutation
sorts the projections $(P_\theta x_i)_{i=1}^n$ and $(P_\theta y_i)_{i=1}^n$. We
formulate this assumption as follows:
\begin{assumption}\label{ass:identity_sorts_projections} The points $X,Y\in
    \R^{n\times d}$ and $\theta\in\SS^{d-1}$ are such that:
    $$P_\theta x_1 \leq \cdots \leq P_\theta x_n\; \text{and}\; P_\theta y_1
    \leq \cdots \leq P_\theta y_n.$$
\end{assumption}
\cref{ass:identity_sorts_projections} holds up to relabelling the points $(x_i)$
and $(y_j)$: taking $\sigma\in\Perm_n$ sorting $(P_\theta x_i)$ and $\tau \in
\Perm_n$ sorting $(P_\theta y_j)$, the relabelled points $\widetilde{X} :=
(x_{\sigma(i)}) =: X\circ\sigma$ and $\widetilde{Y} := (y_{\tau(j)}) =:
Y\circ\tau$ verify the condition.

\subsection{A Kantorovich Formulation of \texorpdfstring{$\CWtheta$}{CWtheta}
Between Point Clouds}

We begin by a characterisation of $\Perm_\theta(X, Y)$, which states that a pair
$(\sigma, \tau)$ belongs to $\Perm_\theta(X, Y)$ if and only if each
``ambiguity'' set $\{i : P_\theta x_i = t\}$ is stable by $\sigma$ and likewise
for $\tau$.

\begin{lemma}\label{lemma:characterisation_Sigma_theta} Let $X := (x_1, \cdots,
    x_n) \in \R^{n\times d},\; Y := (y_1, \cdots, y_n) \in \R^{n\times d}$ and
    $\theta\in \SS^{d-1}$ verifying \cref{ass:identity_sorts_projections}.
    Let $A := \#\{P_\theta x_i\}_{i=1}^n$ and $B := \#\{P_\theta y_j\}_{j=1}^n$.
    Write $\{P_\theta x_i\}_{i=1}^n = (s_a)_{a=1}^A$ where $s_1 < \cdots < s_A$
    and likewise $\{P_\theta y_j\}_{j=1}^n = (t_b)_{b=1}^B$ where $t_1 < \cdots
    < t_B$. Introduce the ``ambiguity group'' sets:
    \begin{equation}\label{eqn:def_Ia_Jb}
        \begin{split}
            \forall a \in \llbracket 1, A \rrbracket,\; 
            I_a &:= \left\{i \in \llbracket 1, n \rrbracket : 
            P_\theta x_i = s_a\right\},\\ 
            \forall b \in \llbracket 1, B \rrbracket,\; 
            J_b &:= \left\{j \in \llbracket 1, n \rrbracket : 
            P_\theta y_j = t_b\right\}.
        \end{split}
    \end{equation} 
    Then the set $\Perm_\theta(X, Y)$ can be re-written as follows:
    \begin{equation}\label{eqn:perm_theta_stability}
        \Perm_\theta(X, Y) = \left\{(\sigma, \tau) \in \Perm_n^2 : 
        \forall a \in \llbracket 1, A \rrbracket, \; \sigma(I_a) = I_a,\; 
        \forall b \in \llbracket 1, B \rrbracket,\; \tau(J_b) = J_b\right\}.
    \end{equation}
\end{lemma}
\begin{proof}
    We show the property $\Perm_\theta(X) = \widetilde{\Perm} := \{\sigma \in
    \Perm_n : \forall a \in \llbracket 1, A\rrbracket,\; \sigma(I_a) = I_a\}$ by
    double inclusion. First, since $s_1 < \cdots < s_A$, the inclusion
    $\widetilde{\Perm} \subset \Perm_\theta(X)$ is clear. For the converse
    inclusion, take $\sigma \in \Perm_\theta(X)$. By definition and by
    \cref{ass:identity_sorts_projections}, we have:
    \begin{align*}
        P_\theta x_1 \leq \cdots \leq P_\theta x_n;\quad 
        P_\theta x_{\sigma(1)} \leq \cdots \leq P_\theta x_{\sigma(n)},
    \end{align*}
    which implies that $\forall i \in \llbracket 1, n \rrbracket,\; P_\theta x_i
    = P_\theta x_{\sigma(i)}$. Take now $a$ the unique element of $\llbracket 1,
    A \rrbracket$ such that $i\in I_a$. We have $s_a = P_\theta x_i = P_\theta
    x_{\sigma(i)}$ and thus $\sigma(i) \in I_a$, and we conclude that $\sigma
    \in \widetilde{\Perm}$. The proof for $\Perm_\theta(Y)$ follows verbatim,
    and \cref{eqn:perm_theta_stability} follows from the definition (see
    \cref{eqn:def_Sigma_theta}).
\end{proof}
To illustrate \cref{lemma:characterisation_Sigma_theta}, we consider an example
with projection ambiguities in \cref{fig:ex_Sigma_theta}.
\begin{figure}[H]
    \centering
    \includegraphics[width=0.9\textwidth]{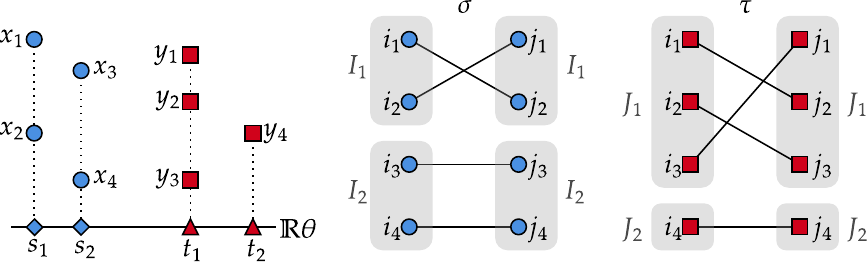}
    \caption{In this example we consider two discrete uniform measures with
    $n:=4$ points with projection ambiguity: $s_1 := P_\theta x_1 = P_\theta x_2
    < s_2 := P_\theta x_3 = P_\theta x_4$ and $t_1 := P_\theta y_1 = P_\theta
    y_2 = P_\theta y_3 < t_2 := P_\theta x_4$. In the notation of
    \cref{lemma:characterisation_Sigma_theta}, we have $A=B=2$ and $I_1 = \{i_1,
    i_2\},\; I_2 = \{i_3, i_4\},\; J_1 = \{j_1, j_2, j_3\}$ and $J_2 = \{j_4\}$.
    We consider a permutation pair $(\sigma, \tau) \in \Perm_\theta(X, Y)$,
    specifically $\sigma := (2, 1, 3, 4)$ and $\tau := (2, 3, 1, 4)$. We see
    that $\sigma$ sorts the sequence $(P_\theta x_i)_{i=1}^n$ and that $I_1$ and
    $I_2$ are stable by $\sigma$, and likewise for
    $\tau$.}\label{fig:ex_Sigma_theta}
\end{figure}

We now write a discrete Kantorovich formulation of $\CWtheta$ between point
clouds, whose expression we will be able to simplify later. The main idea is
that transport plans $P$ are constrained to exchange exactly as much mass
between $I_a$ and $J_b$ as the one-dimensional OT plan $\pi_\theta$ between
$P_\theta\#\mu$ and $P_\theta\#\nu$ sends from $s_a$ to $t_b$, as illustrated in
\cref{fig:ex_CWtheta_discrete_v1}.
\begin{prop}\label{prop:CWtheta_discrete_v1} Under
    \cref{ass:identity_sorts_projections}, let $\mu := \frac{1}{n}\sum_{i=1}^n
    \delta_{x_i}$ and $\nu := \frac{1}{n}\sum_{j=1}^n\delta_{y_j}$ be empirical
    measures. Then the $\CWtheta$ discrepancy introduced in
    \cref{eqn:S_theta_vs_lift} has the following expression:
    \begin{equation}\label{eqn:CWtheta_discrete_v1}
        \CWtheta^2\left(\frac{1}{n}\sum_{i=1}^n \delta_{x_i}, 
        \frac{1}{n}\sum_{j=1}^n\delta_{y_j}\right) = 
        \underset{P\in \U \cap \PthetaXY}{\min}\ 
        \sum_{i=1}^n\sum_{j=1}^n \|x_i-y_j\|_2^2P_{i,j},
    \end{equation}
    \begin{equation}\label{eqn:def_U}
        \U := \{P \in \R_+^{n\times n} : P\mathbf{1} =
        \tfrac{1}{n}\mathbf{1},\; P^\top\mathbf{1} = \tfrac{1}{n}\mathbf{1}\},
    \end{equation}
    \begin{equation}\label{eqn:def_PthetaXY}
        \PthetaXY := \left\{P \in \R^{n\times n} : 
        \forall (a, b) \in \llbracket 1, A \rrbracket \times 
        \llbracket 1, B\rrbracket,\; \sum_{(i,j)\in I_a\times J_b}P_{i,j}
        = \frac{\#(I_a\cap J_b)}{n}\right\}.
    \end{equation}
\end{prop}
\begin{proof}
    Fix $\omega \in \Omega_\theta(\mu, \nu)$ (see \cref{eqn:def_Omega_theta}).
    By the marginal constraints, we have $\supp \omega \subset \{(x_i,
    y_j)\}_{i,j}$, allowing us to define $P \in \R_+^{n\times n}$ by $\forall
    (i,j) \in \llbracket 1, n \rrbracket^2,\; P_{i,j} = \omega(\{(x_i, y_j)\})$.
    Since $\omega \in \Pi(\mu, \nu)$, we verify immediately that $P \in \U$. As
    for the constraint $(P_\theta, P_\theta)\#\omega = \pi_\theta[\mu, \nu] =:
    \pi_\theta$, notice that by construction (see the notations in
    \cref{lemma:characterisation_Sigma_theta}), we can write by \cite[Theorem
    2.9]{santambrogio2015optimal} for \textit{any} $(\sigma, \tau) \in
    \Perm_\theta(X, Y)$ that $\pi_\theta = \frac{1}{n}\sum_i \delta_{(P_\theta
    x_{\sigma(i)}, P_\theta y_{\tau(i)})}$. Since $\{P_\theta x_i\}_{i=1}^n =
    (s_a)_{a=1}^A$ and $\{P_\theta y_j\}_{j=1}^n = (t_b)_{b=1}^B$, it follows
    that for any $(a, b) \in \llbracket 1, A \rrbracket \times \llbracket 1, B
    \rrbracket$:
    \begin{align*}
        \pi_\theta(\{(s_a, t_b)\}) = \sum_{k=1}^n
        \frac{1}{n}\mathbbold{1}(\sigma(k) \in I_a) 
        \mathbbold{1}(\tau(k) \in J_b) =
        \sum_{k=1}^n \frac{1}{n}\mathbbold{1}(k \in I_a \cap J_b) =
        \frac{\#(I_a \cap J_b)}{n},
    \end{align*}
    where we have used that $\sigma(I_a) = I_a$ and $\tau(J_b)=J_b$, which holds
    by \cref{lemma:characterisation_Sigma_theta}. We can now show that $P \in
    \PthetaXY$ using the constraint $(P_\theta, P_\theta)\#\omega = \pi_\theta$:
    $$\frac{\#(I_a \cap J_b)}{n} = \pi_\theta(\{(s_a, t_b)\}) = (P_\theta,
    P_\theta)\#\omega(\{(s_a, t_b)\}) = \sum_{(i,j)\in I_a\times J_b}P_{i,j}. $$
    The cost $\int_{\R^{2d}}\|x-y\|_2^2\dd\omega(x,y)$ writes
    $\sum_{i,j}\|x_i-y_j\|_2^2P_{i,j}$ by definition of $P$. Conversely, it can
    readily be checked with the same computations that for any $P\in
    \U\cap\PthetaXY$, the coupling $\omega := \sum_{i,j}P_{i,j}\delta_{(x_i,
    y_j)}$ belongs to $\Omega_\theta(\mu, \nu)$, and yields the same
    transportation cost. We conclude that the equality in
    \cref{eqn:CWtheta_discrete_v1} holds.
\end{proof}
\begin{figure}[H]
    \centering
    \includegraphics[width=0.65\textwidth]{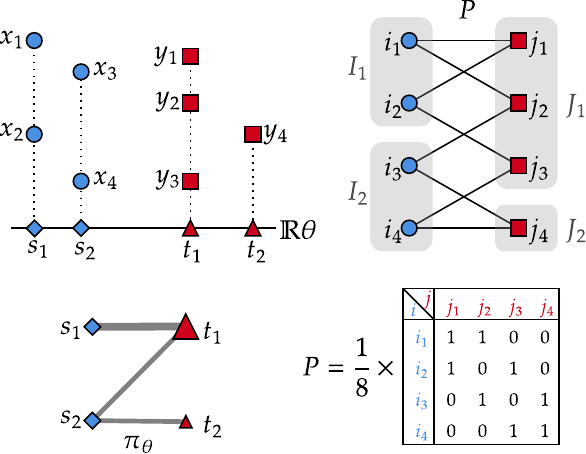}
    \caption{We continue with the example from \cref{fig:ex_Sigma_theta} and
    illustrate the unique optimal transport plan $\pi_\theta =
    \frac{1}{2}\delta_{(s_1, t_1)} + \frac{1}{4}\delta_{(s_2, t_1)} +
    \frac{1}{4}\delta_{(s_2, t_2)}$ between $P_\theta\#\mu =
    \frac{1}{2}\delta_{s_1}+\frac{1}{2}\delta_{s_2}$ and $P_\theta\#\nu =
    \frac{3}{4}\delta_{t_1}+\frac{1}{4}\delta_{t_2}$. We show a particular
    transport plan $P\in \U\cap\PthetaXY$ which is not a permutation matrix. For
    the constraints, notice for example that $\pi_\theta(\{(s_1, t_1)\}) =
    \frac{\#(I_1\cap J_1)}{4}= \frac{1}{2} = \sum_{i=1}^2\sum_{j=1}^3P_{i,j}$.
    }\label{fig:ex_CWtheta_discrete_v1}
\end{figure}
The discrete problem in \cref{eqn:CWtheta_discrete_v1} can be seen as a
constrained Kantorovich problem. Our goal is now to show that it admits a
constrained Monge formulation, which is to say a minimisation over the
constrained set of permutations $\Perm_\theta(X, Y)$. To show this, we will
adapt the proof of the Birkhoff Von Neumann Theorem \cite{birkhoff1946three}
(see also \cite[Theorem 2]{peyre2019course} and \cite{hurlbert2008short} for
other proofs which inspired our method). Our objective is now to build up
definitions and technical lemmas to adapt Birkhoff's Theorem, and prove the
generalisation stated in \cref{thm:extr_U_cap_PthetaXY}. We will consider
particular elements of $\U$ called permutation matrices:
\begin{equation}\label{eqn:def_permutation_matrix}
    \forall (\alpha, \beta) \in \Perm_n^2,\; P^{\alpha, \beta} := 
    \left[\frac{1}{n}\mathbbold{1}(\alpha(i)=\beta(j))\right]_{i,j}.
\end{equation}
This method of writing permutation matrices differs from the usual
$P^\sigma_{i,j} := \frac{1}{n}\mathbbold{1}(\sigma(i)=j)$, and will be more
convenient for our purposes. An elementary property of permutation matrices is
that:
\begin{equation}\label{eqn:invariance_P_alpha_beta}
    \forall (\alpha, \beta, \varphi) \in \Perm_n^3,\; 
    P^{\alpha, \beta} = P^{\varphi\circ\alpha, \varphi\circ\beta},
\end{equation}
since $\varphi\circ\alpha(i) = \varphi\circ\beta(j) \Longleftrightarrow
\alpha(i) = \beta(j)$. For $S \subset \Perm_n^2$, we will write $\CPerm{S} :=
\{P^{\alpha, \beta} : (\alpha, \beta) \in S\}$. The Birkhoff Von Neumann Theorem
\cite{birkhoff1946three} states that $\Extr \U = \CPerm{\Perm_n^2}$, where the
set of extreme points of a convex set is defined as:

\begin{definition}\label{def:extreme_point} The set of extreme points $\Extr C$
    of a convex set $C$ is the set of points $c\in C$ that cannot be written $c
    = \frac{1}{2}a + \frac{1}{2}b$ for some $a,b\in C$.
\end{definition}

Our objective is to show that $\Extr (\U \cap \PthetaXY) =
\CPerm{\Perm_\theta(X, Y)}$. We begin with a Lemma showing a condition for
$P^{\alpha, \beta}$ to belong to $\PthetaXY$.
\begin{lemma}\label{lemma:condition_Psigmatau_in_Ptheta} Under
    \cref{ass:identity_sorts_projections}, for any $(\alpha, \beta) \in
    \Perm_n^2$, we have 
    $$P^{\alpha, \beta} \in \PthetaXY \Longleftrightarrow \exists \varphi \in
    \Perm_n : (\varphi \circ \alpha, \varphi\circ \beta) \in \Perm_\theta(X,
    Y).$$ In other words, $P^{\alpha, \beta} \in \PthetaXY$ if and only if
    $P^{\alpha, \beta} = P^{\sigma, \tau}$ for some $(\sigma, \tau) \in
    \Perm_\theta(X, Y)$.
\end{lemma}
\begin{proof}
    Suppose that $P^{\alpha, \beta} \in \PthetaXY$. Applying the definition of
    $\PthetaXY$ from \cref{eqn:def_PthetaXY}, we see that (using the notation of
    \cref{lemma:characterisation_Sigma_theta}), for any $(a,b) \in \llbracket 1,
    A \rrbracket \times \llbracket 1, B \rrbracket,$
    $$\sum_{(i,j)\in I_a\times J_b}\frac{\mathbbold{1}(\alpha(i)=\beta(j))}{n} =
    \frac{\#(I_a\cap J_b)}{n},\; \text{thus}\; \# (\alpha(I_a) \cap \beta(J_b))
    = \#(I_a\cap J_b).$$ Let $E := \{(a,b) : I_a\cap J_b \neq \varnothing\}$.
    For any $(a,b) \in E$, we have $\# (\alpha(I_a) \cap \beta(J_b)) =
    \#(I_a\cap J_b)$, and thus we can introduce a bijection
    $\varphi_{a,b}:\alpha(I_a) \cap \beta(J_b) \longrightarrow I_a\cap J_b$. We
    have the partition $\llbracket 1, n \rrbracket = \cup_{(a,b)\in E} I_a \cap
    J_b$ where the union is disjoint and the elements are non-empty. Since
    $\alpha, \beta$ are permutations and by the property $\# (\alpha(I_a) \cap
    \beta(J_b)) = \#(I_a\cap J_b)$, we have the partition $\llbracket 1, n
    \rrbracket = \cup_{(a,b)\in E} \alpha(I_a) \cap \beta(J_b)$, again with
    disjoint unions and non-empty terms. We can define $\psi: \llbracket 1, n
    \rrbracket \longrightarrow E$ a map such that $\forall i \in \llbracket 1, n
    \rrbracket,\; i \in \alpha(I_a)\cap \beta(J_b)$ where $\psi(i) = (a,b)$. The
    map $\varphi := i \longmapsto \varphi_{\psi(i)}(i)$ is therefore
    well-defined, we verify easily that it is a permutation of $\llbracket 1, n
    \rrbracket$ using the partition $\llbracket 1, n \rrbracket = \cup_{(a,b)\in
    E} \alpha(I_a) \cap \beta(J_b)$.

    We now fix $a \in \llbracket 1, A \rrbracket$ and show that
    $\varphi\circ\alpha(I_a) = I_a$. Let $i\in I_a$, we have $\alpha(i) \in
    \alpha(I_a)$, and there exists (a unique) $b \in \llbracket 1, B \rrbracket$
    such that $\alpha(i) \in \alpha(I_a)\cap\beta(J_b)$. By definition, we get
    $\psi(\alpha(i)) = (a, b)$, and thus $\varphi(\alpha(i)) =
    \varphi_{a,b}(\alpha(i)) \in I_a\cap J_b \subset I_a$. We conclude that
    $\varphi\circ\alpha(I_a)\subset I_a$ and similarly that
    $\varphi\circ\beta(J_b)\subset J_b$ for any $b \in \llbracket 1, B
    \rrbracket$. By \cref{lemma:characterisation_Sigma_theta}, we conclude that
    $(\varphi\circ\alpha, \varphi\circ\beta)\in \Perm_\theta(X, Y)$, concluding
    the ``left to right'' implication.
    
    Conversely, let $\varphi \in \Perm_n$ and $(\sigma, \tau) \in
    \Perm_\theta(X, Y)$. Notice that $P^{\sigma, \tau} = P^{\varphi \circ
    \sigma, \varphi\circ \tau}$ by \cref{eqn:invariance_P_alpha_beta}. We check
    that $P^{\sigma, \tau} \in \PthetaXY$ by applying the definition: let $(a,b)
    \in \llbracket 1, A \rrbracket \times \llbracket 1, B \rrbracket$, we have:
    $$\sum_{(i,j)\in I_a\times J_b}P_{i,j}^{\sigma, \tau} =
    \frac{\#(\sigma(I_a)\cap\tau(J_b))}{n} = \frac{\#(I_a\cap J_b)}{n}, $$ where
    we used that $\sigma(I_a) = I_a$ and $\tau(J_b)=J_b$, which is a consequence
    of \cref{lemma:characterisation_Sigma_theta}.
\end{proof}

\subsection{Technical Lemmas on Bipartite Graphs Associated to Couplings}

To study the extreme points of $\U\cap\PthetaXY$ we will adapt the techniques
from \cite{peyre2019course,hurlbert2008short} and consider the bipartite graph
associated to a matrix in $P\in\R_+^{n\times m}$, which we define in
\cref{def:bipartite_P}.
\begin{definition}\label{def:bipartite_P} The bipartite directed graph $G_P$
    associated to a matrix $P\in\R_+^{n\times m}$ is the graph with vertices
    $V_P := \{i_k,\; k\in\llbracket 1, n \rrbracket\} \cup \{j_l,\; \ell\in
    \llbracket 1, m \rrbracket\}$ and directed edges:
    \begin{align*}
        E_P &:= \left\{(i_k,j_l),\; (k,\ell)\in \llbracket 1, n \rrbracket \times
        \llbracket 1, m \rrbracket,\; : P_{i_k,j_l}>0\right\} \\
        &\quad \cup\;
        \left\{(j_l,i_k),\; (k,\ell)\in \llbracket 1, n \rrbracket \times \llbracket
        1, m \rrbracket,\; P_{i_k,j_l}>0\right\}.
    \end{align*}
\end{definition}
By slight abuse of notation, we will often write $\{i_k,\; k \in \llbracket 1, n
\rrbracket\}$ as simply $\llbracket 1, n \rrbracket$ and $\{j_l,\; \ell\in
\llbracket 1, m \rrbracket\}$ as $\llbracket 1, m \rrbracket$, seeing them as
disjoint sets of labels. The $i$'s will be called ``left'' vertices, and the
$j$'s ``right'' vertices. Edges $(i,j)$ being directed from left to right, we
will call them ``right'' edges, and likewise edges $(j,i)$ will be referred to
as ``left'' edges. We continue the example from
\cref{fig:ex_CWtheta_discrete_v1} in \cref{fig:ex_bipartite_P} showing the
bipartite graph associated to the matrix $P$.
\begin{figure}[H]
    \centering
    \includegraphics[width=0.6\textwidth]{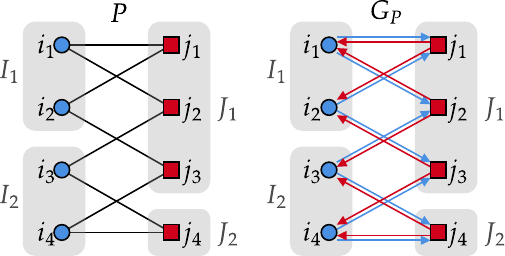}
    \caption{We consider the matrix $P$ from \cref{fig:ex_CWtheta_discrete_v1}
    and show the associated bipartite graph $G_P$. The ``right'' edges from an
    $i\in \llbracket 1, n \rrbracket$ on the left to a $j\in \llbracket 1, n
    \rrbracket$ on the right are represented in blue, and the ``left edges'' are
    represented in red. Note that by construction, for each $(i,j)$ such that
    $P_{i,j}>0$, there is both a left edge $(i,j)$ and a right edge $(j,i)$ in
    $G_P$.}\label{fig:ex_bipartite_P}
\end{figure}

In the following, we will extract a particular cycle $i_1\to j_1\to i_2\to\cdots
\to i_{p+1}=i_1$ from the graph $G_P$ of an element $P\in \U\setminus
\CPerm{\Perm_n^2}$. In proofs of Birkhoff's theorem, this is commonly used to
show that $P$ is not an extreme point of $\U$. In our setting, we will also make
use of this property, in addition to strategies specific to $\PthetaXY$.
\begin{lemma}\label{lemma:extract_cycle_P} Assume $n\geq 2$ and let $P\in
    \U\setminus \CPerm{\Perm_n^2}$. 
    
    Then there exists a cycle $(i_1, j_1,
    \cdots, i_p, j_p, i_{p+1}) \in \llbracket 1, n \rrbracket^{2p+1}$ in $G_P$
    with $p\geq 2$ verifying:
    \begin{equation}\label{eqn:exists_adequate_cycle}
        \begin{split}
            &i_{p+1} = i_1;\; (i_k)_{k=1}^p\ 
            \text{and}\ (j_k)_{k=1}^p\ \text{are\ injective};\\
            &\text{and}\  
            \forall k \in \llbracket 1, p \rrbracket,\; 
            P_{i_k, j_k} \in (0, \tfrac{1}{n}),\; 
            P_{i_{k+1}, j_k} \in (0, \tfrac{1}{n}). 
        \end{split}
    \end{equation}
\end{lemma}
\begin{proof}
    First, we show a weaker result:
    \begin{equation}\label{eqn:exists_adequate_cycle_noninj}
        \begin{split}
            &\exists p \geq 2,\; \exists (i_1, j_1, \cdots, i_p, j_p, i_{p+1}) 
            \in \llbracket 1, n \rrbracket^{2p+1} \\
            &\text{such\ that}\ i_{p+1} = i_1;\; 
            \forall k \in \llbracket 1, p \rrbracket,\; i_k\neq i_{k+1},\;
            \forall k \in \llbracket 1, p-1 \rrbracket,\; j_k\neq j_{k+1};\\
            &\text{and}\  
            \forall k \in \llbracket 1, p \rrbracket,\; 
            P_{i_k, j_k} \in (0, \tfrac{1}{n}),\; 
            P_{i_{k+1}, j_k} \in (0, \tfrac{1}{n}). 
        \end{split}
    \end{equation}
    Since $P\in \U\setminus \CPerm{\Perm_n^2}$, there exists $(i_1, j_1) \in
    \llbracket 1, n \rrbracket^2$ such that $P_{i_1, j_1} \in (0, \frac{1}{n})$.
    Since $0<P_{i_1, j_1} < \sum_{i}P_{i,j_1} = \frac{1}{n}$, there exists
    $i_2\neq i_1$ such that $P_{i_2, j_1} \in (0, \frac{1}{n})$. Likewise, since
    $0<P_{i_2, j_1} < \sum_{j}P_{i_2,j} = \frac{1}{n}$, there exists $j_2 \neq
    j_1$ such that $P_{i_2, j_2}\in (0, \frac{1}{n})$. We continue and show the
    existence of $i_3\neq i_2$ such that $P_{i_3, j_2} \in (0, \frac{1}{n})$. So
    far, we have built a chain $i_1\to j_1\to i_2 \to j_2 \to i_3$. If $i_3=i_1$
    then we have shown \cref{eqn:exists_adequate_cycle_noninj}. Otherwise we
    continue the process up to $i_k,\; k\geq 4$ while $i_k\neq i_1$, and there
    are two exclusive possibilities:
    \begin{enumerate}
        \item The process terminates with $ (i_1, j_1, \cdots, i_p, j_p,
        i_{p+1})$ such that $i_{p+1}=i_1$, and by construction the cycle
        verifies the conditions of \cref{eqn:exists_adequate_cycle_noninj};
        \item The process continues at least up to $k=n+1$, yielding $(i_1, j_1,
        \cdots, i_n, j_n, i_{n+1})$ verifying the conditions of
        \cref{eqn:exists_adequate_cycle_noninj} except $i_{n+1}\neq i_1$. Then
        by the pigeonhole principle, there exists $k_1<k_2 \in \llbracket 1, n+1
        \rrbracket^2$ such that $i_{k_1}=i_{k_2}$. Consider the cycle $(i_{k_1},
        j_{k_1}, i_{k_1+1}, j_{k_1+1}, \cdots, i_{k_2})$, it verifies
        \cref{eqn:exists_adequate_cycle_noninj} (the length is sufficient since
        by construction $i_{k_1}\neq i_{k_1+1}$).
    \end{enumerate}
    Now that we have shown \cref{eqn:exists_adequate_cycle_noninj}, we deduce
    \cref{eqn:exists_adequate_cycle} by taking $p\geq 2$ minimal in
    \cref{eqn:exists_adequate_cycle_noninj}.
\end{proof}

As an illustration, in \cref{fig:ex_bipartite_P}, by following the edges of
$G_P$ starting from the edge $(i_1, j_1)$, we observe the cycle $(i_1, j_1, i_2,
j_3, i_4, j_4, i_3, j_2, i_1)$ which satisfies the criteria of
\cref{eqn:exists_adequate_cycle}. 

We will also require the following technical result about extracting injective
cycles from (possibly) redundant cycles in a graph. For a set $S$ and $n,m \in
\N$, we say that two families $(s_i)_{i=1}^n\in S^n$ and $(t_j)_{j=1}^n\in S^n$
are equipotent if $n=m$ and there exists a permutation $\varphi \in \Perm_n$
such that $\forall i \in \llbracket 1, n \rrbracket,\; s_{\varphi(i)}=t_i$. We
write this property $(s_i)\simeq (t_j)$. This concept is particularly useful
when the families are not injective, which will sometimes be the case in the
following.

\begin{lemma}\label{lemma:split_multicycle} Let $G :=
    (V:=\mathcal{A}\cup\mathcal{B}, E)$ be a directed bipartite graph, set
    $p\geq 1$ and consider a cycle written $(a_1, b_1, \cdots, a_p, b_p,
    a_{p+1}) \in V^{2p+1}$ with $\forall k \in \llbracket 1, p \rrbracket, \;
    (a_k, b_k) \in E,\; (b_k, a_{k+1})\in E$. Then there exists $L\geq 1$ cycles
    of $G$ of the form $(a_1^\ell, b_1^\ell, \cdots, a_{p_\ell}^\ell,
    b_{p_\ell}^\ell, a_{p_\ell+1}^\ell)$ (with $a_1^\ell = a_{p_\ell+1}^\ell$
    and each $(a_k^\ell, b_k^\ell), (b_k^\ell, a_{k+1}^\ell) \in E$) whose
    combined elements (without the last vertex) are exactly the elements of
    $(a_1, b_1, \cdots, a_p, b_p)$:
    \begin{equation}\label{eqn:split_multicycle_concatenation}
        (a_1, b_1, \cdots, a_p, b_p) \simeq (a_1^1, b_1^1, \cdots, a_{p_1}^1,
        b_{p_1}^1, \cdots \cdots, a_1^L, b_1^L, \cdots, a_{p_L}^L, b_{p_L}^L),
    \end{equation}
    and such that for each $\ell\in \llbracket 1, L \rrbracket,$ the families of
    edges $((a_k^\ell, b_k^\ell))_{k=1}^{p_\ell}$ and $((b_k^\ell,
    a_{k+1}^\ell))_{k=1}^p$ are injective. 
\end{lemma}
\begin{proof}
    Given such a cycle $\mathcal{C}:=(a_1, b_1, \cdots, a_p, b_p, a_{p+1})$ we
    consider the two following ``splitting'' operators:
    \begin{itemize}
        \item $\mathrm{Split}_R$ takes the first pair $i<j \in \llbracket 1,
        p\rrbracket^2$ (for the lexicographic order) such that $(a_i, b_i) =
        (a_j, b_j)$ if such a pair $(i,j)$ exists (if not,
        $\mathrm{Split}_R(\mathcal{C})$ returns $\mathcal{C}$).
        $\mathrm{Split}_R(\mathcal{C})$ then returns the two following
        sub-cycles:
        $$\mathcal{C}_1 := (a_1, b_1, \cdots, a_{i-1}, b_{i-1}, a_j, b_j,
        \cdots, a_p, b_p, a_{p+1}),\; \mathcal{C}_2 := (a_i, b_i, \cdots, ,
        a_{j-1}, b_{j-1}, a_j). $$ Obviously, their concatenation without
        endpoints is exactly $\mathcal{C}$ without its endpoint: 
        $$(a_1, b_1, \cdots, a_{i-1}, b_{i-1}, a_j, b_j, \cdots, a_p, b_p, a_i,
        b_i, \cdots, , a_{j-1}, b_{j-1}) \simeq (a_1, b_1, \cdots, a_p, b_p)
        .$$
        \item $\mathrm{Split}_L$ takes the first pair $i<j \in \llbracket 1,
        p\rrbracket^2$ such that $(b_i, a_{i+1}) = (b_j, a_{j+1})$ if such a
        pair $(i,j)$ exists (if not, $\mathrm{Split}_L(\mathcal{C})$ returns
        $\mathcal{C}$). $\mathrm{Split}_L(\mathcal{C})$ then returns the two
        following sub-cycles, (which verify the equipotence condition):
        $$\mathcal{C}_1 := (a_1, b_1, \cdots, a_i, b_i, a_{j+1}, b_{j+1},
        \cdots, a_p, b_p, a_{p+1}),\; \mathcal{C}_2 := (a_{i+1}, b_{i+1},
        \cdots, a_j, b_j, a_{j+1}). $$
    \end{itemize}
    To split an initial $\mathcal{C}$, we construct a family
    $(\mathcal{C}_\ell)$ of cycles iteratively starting with $(\mathcal{C})$ by
    applying $\mathrm{Split}_R$ and $\mathrm{Split}_L$ to the cycles to the
    family $\mathcal{C}_\ell$ until no cycle can be split. This process
    terminates since each iteration increases the number of cycles (they are
    non-empty), which is bounded because $\mathcal{C}$ is finite and the
    concatenation of the cycles $(\mathcal{C}_\ell)$ without endpoints is
    exactly $\mathcal{C}$ without its endpoint. At the end of the process, the
    equipotence condition remains and each cycle $\mathcal{C}_\ell$ has
    injective edges $((a_k^\ell, b_k^\ell))_k, ((b_k^\ell, a_{k+1}^\ell))_k$
    since the splitting process could not continue.
\end{proof}
In \cref{fig:ex_extract_cycles} we illustrate the splitting process of
\cref{lemma:split_multicycle}.
\begin{figure}[H]
    \centering
    \includegraphics[width=.8\linewidth]{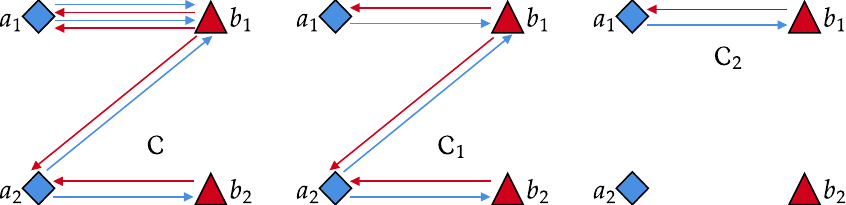}
    \caption{Extracting two cycles $\mathcal{C}_1,\mathcal{C}_2$ from
    the cycle $\mathcal{C}$ such that each cycle $\mathcal{C}_\ell$ has distinct
    (directed) edges.}
    \label{fig:ex_extract_cycles}
\end{figure}
The cycle from \cref{fig:ex_extract_cycles} is a cycle of $G_{\oll{P}}$ (where
$\oll{P}$ is the OT plan matrix between the measures $P_\theta\#\mu$ and
$P_\theta\#\nu$), constructed using $P$ from \cref{fig:ex_bipartite_P} with $a_1
\in \llbracket 1, A \rrbracket$ where $i\in I_{a_1}$, then $b_1 \in \llbracket
1, B \rrbracket$ such that $j_1 \in I_{b_1}$ and so on. This example case is
paramount since it will be the use case of \cref{lemma:split_multicycle} in the
proof of \cref{thm:extr_U_cap_PthetaXY}.

The following lemma is in essence a cyclical monotonicity property, and concerns
a property of cycles in the bipartite graph associated to the matrix $\oll{P}\in
\R_{+}^{A\times B}$ which is the unique optimal transport plan matrix between
the one-dimensional measures $P_\theta\#\mu = \sum_{a=1}^A\frac{\#
I_a}{n}\delta_{s_a}$ and $P_\theta\#\nu = \sum_{b=1}^B\frac{\#
J_b}{n}\delta_{t_b}$. The idea is that by monotonicity of $\oll{P}$, no edges of
$G_{\oll{P}}$ can cross one another, which constrains cycles to have a left edge
$(b,a)$ corresponding to each right edge $(a,b)$. We remind that by assumption
$s_1 < \cdots < s_A$ and $t_1 < \cdots < t_B$ (the notation was introduced in
\cref{lemma:characterisation_Sigma_theta}).

\begin{lemma}\label{lemma:cycle_ollP} Let $\oll{P}\in \R_{+}^{A\times B}$ be the
    OT matrix between $\sum_{a=1}^A\frac{\# I_a}{n}\delta_{s_a}$ and
    $\sum_{b=1}^B\frac{\# J_b}{n}\delta_{t_b}$. 
    
    If $\mathcal{C}:=(a_1, b_1, \cdots, a_p, b_p, a_{p+1})$ is a cycle in
    $G_{\oll{P}}$ (i.e. $(a_k)_{k=1}^{p+1} \in \llbracket 1, A
    \rrbracket^{p+1},\; (b_k)_{k=1}^p \in \llbracket 1, B \rrbracket^p,\;
    a_{p+1}=a_1$ and $\forall k \in \llbracket 1, p \rrbracket,\; \oll{P}_{a_k,
    b_k} >0,\; \oll{P}_{a_{k+1}, b_k} >0$) such that the families of edges
    $((a_k, b_k))_{k=1}^{p}$ and $((b_k, a_{k+1}))_{k=1}^p$ are injective, then
    $((a_k, b_k))_{k=1}^{p} \simeq ((a_{k+1}, b_k))_{k=1}^p$.
\end{lemma}
\begin{proof}
    First, by optimality of $\oll{P}$, by \cite[Lemma
    2.8]{santambrogio2015optimal}, $\oll{P}$ is monotone in the sense that:
    $$\forall (a,b), (a', b') \in \llbracket 1, A \rrbracket \times \llbracket
    1, B \rrbracket,\ \text{such\ that}\ \oll{P}_{a,b}>0, \oll{P}_{a',b'}>0,\;
    a<a' \implies b\leq b'. $$ Note that the contrapositive yields the
    symmetrical property that if $b<b'$ then $a\leq a'$. Furthermore, we remind
    that since each $(a_k, b_k)$ and $(b_k, a_{k+1})$ are edges of the graph
    $G_{\oll{P}}$, we have $\oll{P}_{a_k, b_k} >0$ and $\oll{P}_{a_{k+1}, b_k}
    >0$. We can understand the monotonicity property as the fact that the edges
    of the cycle cannot cross one another.
    
    By injectivity of the edge families, to show the equipotence result, it
    suffices to show that $\forall k \in \llbracket 1, p \rrbracket,\; \exists
    k' \in \llbracket 1, p \rrbracket$ such that $(a_k, b_k) = (a_{k'+1},
    b_{k'})$. Since the vertices $a_k$ and $b_k$ are part of the cycle $a_1 \to
    b_1 \to a_2 \to \cdots \to a_{p+1} = a_1$, there exists a sub-cycle $b_k \to
    a'_1 \to b'_1 \to \cdots \to b'_q \to a_k$, which is to say that there
    exists, for some $q\geq 0$, $(b_k, a'_1)\in \mathcal{C},\; \forall k' \in
    \llbracket 1, q \rrbracket,\; (a'_{k'}, b'_{k'}) \in \mathcal{C},\;
    (b'_{k'}, a'_{k'+1}) \in \mathcal{C}$ (writing $a'_{q+1} := a_k$), and we
    now take $q\geq 0$ minimal. We will show that $q=0$ by contradiction: assume
    $q\geq 1$, which implies that $a'_1\neq a_k$ by minimality. Assume that
    $a_1' < a_k$ (the case $a_1'>a_k$ is analogous). By monotonicity of
    $\oll{P}$, we deduce $b_1'\leq b_k$, and even $b_1'< b_k$ since $b_1'=b_k$
    would violate the minimality of $q$. By monotonicity, we deduce that
    $a_2'\leq a_1'$ and again, even $a_2' < a_1'$ by minimality of $a$.
    Continuing this process we find that $a_{q+1}'<a_k$ contradicting
    $a_{q+1}'=a_k$. We conclude that the edge $(b_k, a_k)$ belongs to the cycle,
    which is to say that there exists $k'\in \llbracket 1, p \rrbracket$ such
    that $(a_k, b_k) = (a_{k'+1}, b_{k'})$, finishing the proof.
\end{proof}
In \cref{fig:ex_cycle_ollP} we illustrate the result of \cref{lemma:cycle_ollP}
in the use case of the proof of \cref{thm:extr_U_cap_PthetaXY}, regrouping the
continued example from
\cref{fig:ex_Sigma_theta,fig:ex_CWtheta_discrete_v1,fig:ex_bipartite_P,fig:ex_extract_cycles}.
\begin{figure}[H]
    \centering
    \includegraphics[width=.8\linewidth]{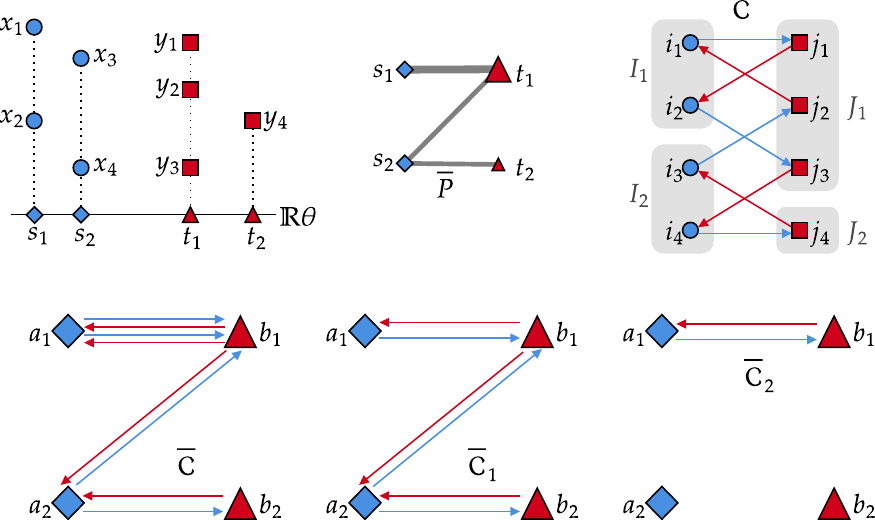}
    \caption{We take two discrete uniform measures $\mu := \frac{1}{n}\sum_i
    \delta_{x_i}$ and $\nu := \frac{1}{n}\sum_j\delta_{y_j}$ such that $s_1 :=
    P_\theta x_1 = P_\theta x_2 < s_2 := P_\theta x_3 = P_\theta x_4$ and $t_1
    := P_\theta y_1 = P_\theta y_2 = P_\theta y_3 < t_2 := P_\theta y_4$. We
    represent the OT matrix $\oll{P}$ between the measures $P_\theta\#\mu$ and
    $P_\theta\#\nu$ and consider the bipartite graph $G_P$ associated to a
    coupling $P \in \mathbb{U}\cap \PthetaXY$ (presented in
    \cref{fig:ex_bipartite_P}). In this case the graph $G_P$ (top-right)
    contains the cycle $\mathcal{C} := (i_1, j_1, i_2, j_3, i_4, j_4, i_3, j_2,
    i_1)$. We consider for each $k$ the ``right'' edge $(a_k, b_k)$ in
    $G_{\oll{P}}$ such that $i_k \in I_{a_k}$ and $j_k \in J_{b_k}$, and the
    ``left'' edge $(b_k, a_{k+1})$ such that $j_k \in J_{b_k}$ and $i_{k+1} \in
    I_{a_{k+1}}$. This defines the cycle $\oll{\mathcal{C}}$ in $G_{\oll{P}}$,
    that we decompose into cycles with distinct edges $(\oll{\mathcal{C}}_\ell)$
    using \cref{lemma:split_multicycle}. \cref{lemma:cycle_ollP} then applies to
    each $\oll{\mathcal{C}}_\ell$ and we observe indeed that in
    $\oll{\mathcal{C}}_\ell$, each ``left'' edge $(b, a)$ has a corresponding
    ``right'' edge $(a, b)$ in the cycle.}
    \label{fig:ex_cycle_ollP}
\end{figure}

\subsection{A Constrained Version of the Birkhoff von Neumann Theorem}

We are now ready to prove a constrained version of the Birkhoff von Neumann
Theorem \cite{birkhoff1946three}. We remind that $\PthetaXY$ is defined in
\cref{eqn:def_PthetaXY}, $\U$ in \cref{eqn:def_U} and $\CPerm{\Perm_\theta(X,
Y)} = \{P^{\sigma, \tau},\; (\sigma, \tau)\in \Perm_\theta(X, Y)\}$, with
$P^{\sigma, \tau}$ the permutation matrix introduced in
\cref{eqn:def_permutation_matrix} and $\Perm_\theta(X, Y)$ defined in
\cref{eqn:def_Sigma_theta}. Finally, the notion of extreme points is defined in
\cref{def:extreme_point}.
\begin{theorem}\label{thm:extr_U_cap_PthetaXY} Let $(X, Y) \in \R^{n\times d}$
    and $\theta \in \SS^{d-1}$ verifying \cref{ass:identity_sorts_projections}.
    Then $$\Extr(\U\cap\PthetaXY) = \CPerm{\Perm_\theta(X, Y)}.$$
\end{theorem}
\begin{proof}
    \step{1}{$\CPerm{\Perm_\theta(X, Y)} \subset \Extr(\U\cap\PthetaXY)$}

    First, for $(\sigma, \tau)\in \Perm_\theta(X, Y)$, we have $P^{\sigma,
    \tau}\in \PthetaXY$ by \cref{lemma:condition_Psigmatau_in_Ptheta}, which
    shows that $P^{\sigma, \tau} \in \U\cap\PthetaXY$. Now if $P^{\sigma, \tau}
    = \frac{1}{2}Q + \frac{1}{2}R$ for some $(Q, R) \in \U\cap\PthetaXY$, then
    for any $(i,j)\in \llbracket 1, n \rrbracket^2$, we have $P^{\sigma,
    \tau}_{i,j} \in \{0, \frac{1}{n}\}$, thus $\frac{1}{2}Q_{i,j} +
    \frac{1}{2}R_{i,j} = P^{\sigma, \tau}_{i,j}$ implies that $Q_{i,j} =
    R_{i,j}$ since $Q_{i,j}$ and $R_{i,j}$ are both in $[0, \frac{1}{n}]$ (since
    they belong to $\U$). This shows that $P^{\sigma, \tau} \in
    \Extr(\U\cap\PthetaXY)$.

    \step{2}{Writing $P \in (\U\cap\PthetaXY) \setminus \CPerm{\Perm_\theta(X,
    Y)}$ as $P=(Q+R)/2$ with $Q, R \in [0, \frac{1}{n}]^{n\times n}$}
    
    To show $\Extr(\U\cap\PthetaXY) \subset \CPerm{\Perm_\theta(X, Y)}$, we will
    show that 
    $$(\U\cap\PthetaXY) \setminus \CPerm{\Perm_\theta(X, Y)} \subset
    (\U\cap\PthetaXY) \setminus\Extr(\U\cap\PthetaXY).$$ Note that when $n=1$
    the entire Theorem is trivial, and in the following we assume $n\geq 2$. We
    take $P \in (\U\cap\PthetaXY) \setminus \CPerm{\Perm_\theta(X, Y)}$ and
    apply \cref{lemma:extract_cycle_P}, allowing us to introduce $(i_1, j_1,
    \cdots, i_p, j_p, i_{p+1}) \in \llbracket 1, n \rrbracket^{2p+1}$ such that
    $i_{p+1}=i_1$, the families $(i_k)_{k=1}^p$ and $(j_k)_{k=1}^p$ are
    injective and $\forall k \in \llbracket 1, p \rrbracket,\; P_{i_k, j_k} \in
    (0, \tfrac{1}{n}),\; P_{i_{k+1}, j_k} \in (0, \tfrac{1}{n})$. We consider
    the set of ``right edges'' $E_R := ((i_k, j_k))_{k=1}^p$ and ``left edges'':
    $E_L := ((i_{k+1}, j_k))_{k=1}^p$. By injectivity, we have $E_R \cap E_L =
    \varnothing$ and that $E_R$ and $E_L$ are themselves injective. Note that
    our cycle construction is illustrated on an example in
    \cref{fig:ex_cycle_ollP}. We take the smallest margin $\varepsilon>0$ that
    $P$ has to be in $\U$ (within the cycle):
    $$\varepsilon := \underset{k\in \llbracket 1, p \rrbracket}{\min}\{P_{i_k,
    j_k}, 1-P_{i_k, j_k}, P_{i_{k+1},j_k}, 1-P_{i_{k+1},j_k}\} \in (0,
    \tfrac{1}{n}), $$ and introduce the matrices $Q,R \in \R^{n\times n}$
    defined as, for $(i,j)\in \llbracket 1, n \rrbracket^2$:
    \begin{align*}
        Q_{i,j} &:= 
        \left\{\begin{array}{ccc} 
            P_{i,j} & \text{if} & (i,j)\notin E_R\cup E_L \\
            P_{i,j} + \varepsilon & \text{if} & (i,j)\in E_R \\
            P_{i,j} - \varepsilon & \text{if} & (i,j)\in E_L
        \end{array}\right.,\\
        R_{i,j} &:= \left\{\begin{array}{ccc}
            P_{i,j} & \text{if} & (i,j)\notin E_R \cup E_L \\
            P_{i,j} - \varepsilon & \text{if} & (i,j)\in E_R \\
            P_{i,j} + \varepsilon & \text{if} & (i,j)\in E_L 
        \end{array}\right. .
    \end{align*}
    For visualisation purposes, in the example of \cref{fig:ex_cycle_ollP}, we
    represent ``left'' edges of the cycle in blue (on these edges, we add
    $+\varepsilon$ in $Q$ and $-\varepsilon$ in $R$) and ``right'' edges in red
    (on which we do the opposite). By definition of $\varepsilon$, we have $Q, R
    \in [0,\frac{1}{n}]^{n\times n}$. By construction, we also have $P =
    \frac{1}{2}(Q+R)$.

    \step{3}{Showing that $Q, R \in \U$}

    Fix $j\in \llbracket 1, n \rrbracket,$ we show that $\sum_iQ_{i,j} =
    \frac{1}{n}$. Since $E_L$ and $E_R$ are injective and disjoint, we compute
    \begin{align*}
        \sum_{i=1}^nQ_{i,j} &= \sum_{i : (i,j)\not\in E_R\cup E_L} P_{i,j} +
        \sum_{i : (i,j)\in E_R} (P_{i,j} +\varepsilon) + \sum_{i : (i,j)\in E_L}
        (P_{i,j} -\varepsilon) \\
        &= \frac{1}{n} + \varepsilon(\# I_R(j) - \#I_L(j)), 
    \end{align*}
    where $I_R(j) := \{i\in \llbracket 1, n \rrbracket : (i,j)\in E_R\}$ and
    $I_L(j) := \{i \in \llbracket 1, n \rrbracket : (i,j)\in E_L\}$. Since $E_R$
    and $E_L$ are injective, we deduce that $I_R$ and $I_L$ are also injective.
    Take $i \in I_R(j)$ and write $(i, j) = (i_k, j_k) \in E_L$ for some $k\in
    \llbracket 1, p \rrbracket$. We notice that $(i_k, j_{k-1}) \in E_R$ where
    if $k=1$ we write $j_{k-1} := j_p$. We conclude that $\# I_R(j) = \# I_L(j)$
    and thus that $\sum_iQ_{i,j} = \frac{1}{n}$. The same reasoning shows that
    $\sum_jQ_{i,j} = \frac{1}{n}$ for all $i\in \llbracket 1, n \rrbracket$, and
    we conclude that $Q\in \U$. The same computations show that $R \in \U$ as
    well.

    \step{4}{Showing that $Q,R \in \PthetaXY$}

    We now show that $Q,R \in \PthetaXY$ using the definition
    (\cref{eqn:def_PthetaXY}). Take $(a, b) \in \llbracket 1, A \rrbracket
    \times \llbracket 1, B \rrbracket$. We have:
    \begin{align}
        \sum_{(i,j)\in I_a\times J_b}Q_{i,j} 
        &= \sum_{(i,j)\in (I_a\times J_b)\cap E_R^c \cap E_L^c}P_{i,j} 
        + \sum_{(i,j)\in (I_a\times J_b)\cap E_R}(P_{i,j} + \varepsilon)\nonumber\\
        &\quad + \sum_{(i,j)\in (I_a\times J_b)\cap E_L}(P_{i,j} - \varepsilon)
        \nonumber \\
        &= \cfrac{\# I_a\cap J_b}{n} 
        + \varepsilon\left(\# ((I_a \times J_b)\cap E_R) 
        - \# ((I_a \times J_b)\cap E_L)\right). 
        \label{eqn:sum_Q_condition_Ptheta}
    \end{align}
    Let $\oll{P}\in \R_+^{A\times B}$ be the OT matrix between
    $\sum_{a=1}^A\frac{\# I_a}{n}\delta_{s_a}$ and $\sum_{b=1}^B\frac{\#
    J_b}{n}\delta_{t_b}$. Consider the family $\oll{\mathcal{C}} := (a_1, b_1,
    \cdots, a_p, b_p, a_{p+1})$ defined by the condition $\forall k \in
    \llbracket 1, p \rrbracket,\; i_k \in I_{a_k},\; j_k \in J_{b_k}$ and
    $a_{p+1} := a_1$. Since $\mathcal{C} := (i_1, i_2, \cdots, i_p, j_p,
    i_{p+1})$ is a cycle in $G_{P}$, it follows that $\oll{\mathcal{C}}$ is a
    cycle in $G_{\oll{P}}$ since the condition $P \in \PthetaXY$ implies:
    $$\forall (a,b) \in \llbracket 1, A \rrbracket \times \llbracket 1, B
    \rrbracket,\; \sum_{(i,j)\in I_a\times J_b} P_{i,j} = \oll{P}_{a,b}, $$ thus
    if $P_{i,j}>0$ for some $(i,j)\in I_a\times J_b$ then $P_{a,b}>0$. See also
    \cref{fig:ex_cycle_ollP} for an example. We now apply
    \cref{lemma:split_multicycle} to show that $\oll{\mathcal{C}}$ is the
    ``concatenation'' of $L\geq 1$ cycles $\oll{\mathcal{C}}_\ell$ of
    $G_{\oll{P}}$ of the form:
    $$\oll{\mathcal{C}}_\ell := (a_1^\ell, b_1^\ell, \cdots, a_{p_\ell}^\ell,
    b_{p_\ell}^\ell, a_{p_\ell+1}^\ell),$$ where ``concatenation'' means that
    \cref{eqn:split_multicycle_concatenation} holds with the same notation, and
    where each $\oll{\mathcal{C}}_\ell$ is such that the edge families
    $((a_k^\ell, b_k^\ell))_{k=1}^{p_\ell}$ and $((b_k^\ell,
    a_{k+1}^\ell))_{k=1}^{p_\ell}$ are injective. For each $\ell\in \llbracket
    1, L \rrbracket$, we apply \cref{lemma:cycle_ollP}, which shows in
    particular that for any $(a,b) \in \llbracket 1, A \rrbracket \times
    \llbracket 1, B \rrbracket$:
    \begin{equation}\label{eqn:count_groups_in_Cl}
        \#\left\{k\in \llbracket 1, p_\ell \rrbracket : 
        (a_k^\ell, b_k^\ell) = (a,b)\right\} 
        = \#\left\{k\in \llbracket 1, p_\ell \rrbracket : 
        (a_{k+1}^\ell, b_k^\ell) = (a,b)\right\}.
    \end{equation}
    We understand \cref{eqn:count_groups_in_Cl} as the fact that for any left
    edge from group $a$ to group $b$ in $\mathcal{C}_\ell$, there corresponds
    exactly as many right edges from group $b$ to group $a$. This will allow us
    to show that the terms in $+\varepsilon$ and $-\varepsilon$ are in the same
    number. We now re-write the sets from the condition on $Q$
    (\cref{eqn:sum_Q_condition_Ptheta}) for a fixed $(a, b)\in \llbracket 1, A
    \rrbracket \times \llbracket 1, B \rrbracket$:
    \begin{align*}
        \#\left((I_a\times J_b)\cap E_R\right) &= 
        \left\{(i_k, j_k),\; k \in \llbracket 1, p \rrbracket,
        \; (i_k, j_k) \in I_a\times J_b \right\} \\
        &= \#\left((a_k, b_k),\; k \in \llbracket 1, p \rrbracket,\; 
        (a_k, b_k) = (a,b) \right) \\
        &= \sum_{\ell=1}^L \#\left((a_k^\ell, b_k^\ell),\; 
        k \in \llbracket 1, p_\ell \rrbracket,
        \; (a_k^\ell, b_k^\ell) = (a,b) \right) \\
        &= \sum_{\ell=1}^L \#\left((b_k^\ell, a_{k+1}^\ell),\; 
        k \in \llbracket 1, p_\ell \rrbracket,
        \; (a_{k+1}^\ell, b_k^\ell) = (a,b) \right) \\
        &= \#\left((I_a\times J_b)\cap E_L\right),
    \end{align*}
    where the first equality uses the definition of $E_R$, the second inequality
    comes from associating to each pair $(i_k, j_k)$ its group pair $(a_k, b_k)$
    and counting the group pairs \textit{with repetition}, the third equality
    the concatenation property of the cycles $\mathcal{C}_\ell$
    (\cref{eqn:split_multicycle_concatenation}), the fourth equality from
    \cref{eqn:count_groups_in_Cl} and the last inequality from the definition of
    $E_L$ (doing the same computations as for $E_R$ in reverse order).

    Combining with \cref{eqn:sum_Q_condition_Ptheta} shows that $\sum_{(i,j)\in
    I_a\times J_b} Q_{i,j} = \tfrac{\#(I_a\cap J_b)}{n}$ and thus that $Q \in
    \PthetaXY$. Likewise we show $R \in \PthetaXY$ and thus we have found $Q,R
    \in \U\cap \PthetaXY$ such that $P = \frac{1}{2}(Q+R)$, and we conclude that
    $P$ does not belong to $\Extr(\U\cap\PthetaXY)$, finishing the proof.
\end{proof}
From \cref{thm:extr_U_cap_PthetaXY} we deduce the following theorem, which is a
Monge formulation of the constrained Kantorovich problem in $\CWtheta$:

\begin{theorem}\label{thm:CW_theta_point_clouds} Let $(X, Y) \in \R^{n\times d}$
    and $\theta \in \SS^{d-1}$, we have:
    \begin{equation}\label{eqn:CW_theta_point_clouds}
        \CWtheta^2\left(\frac{1}{n}\sum_{i=1}^n \delta_{x_i}, 
        \frac{1}{n}\sum_{j=1}^n \delta_{y_j}\right) = 
        \underset{(\sigma, \tau) \in \Perm_\theta(X, Y)}{\min}\ 
        \frac{1}{n}\sum_{i=1}^n \|x_{\sigma(i)} - y_{\tau(i)}\|_2^2,
    \end{equation}
\end{theorem}
\begin{proof}
    Beginning under \cref{ass:identity_sorts_projections}, we combine
    \cref{thm:extr_U_cap_PthetaXY} with the expression of $\CWtheta^2$ from
    \cref{prop:CWtheta_discrete_v1}:
    \begin{align*}
         \CWtheta^2\left(\frac{1}{n}\sum_{i=1}^n \delta_{x_i},
        \frac{1}{n}\sum_{j=1}^n \delta_{y_j}\right) 
        &= \underset{P\in\U\cap\PthetaXY}{\min}\ 
        \sum_{i,j}\|x_{i} - y_{j}\|_2^2P_{i,j} \\
        &=\underset{P\in \Extr(\U\cap\PthetaXY)}{\min}\ \sum_{i,j}\|x_{i} -
        y_{j}\|_2^2P_{i,j},
    \end{align*}
    since the solution of a linear program over a non-empty convex compact set
    is attained at an extreme point \cite[Theorem
    2.7]{bertsimas1997introduction}, and we conclude that the expression in
    \cref{eqn:CW_theta_point_clouds} holds thanks to
    \cref{thm:extr_U_cap_PthetaXY}. For the general case without
    \cref{ass:identity_sorts_projections}, we use
    \cref{lemma:CV_theta_monge_invariant_permutation}.
\end{proof}

%% file: sections/min_pivot_sliced.tex
\section{Min-Pivot Sliced}
\label{sec:min_PS}

\subsection{Min-Pivot Sliced Discrepancy: Definition}

A specificity of the Pivot Sliced Wasserstein discrepancy is the dependence on
the axis $\theta\in\SS^{d-1}$, which can overly constrain the choice of
transport plans. In this section, we study the Min-Pivot Sliced Discrepancy
which minimises $\PStheta$ over $\theta\in\SS^{d-1}$. This object was first
introduced in \cite{mahey23fast} on the set of discrete uniform measures with
$n$ points. 
\begin{equation}\label{eqn:def_min_sliced_discrepancy}
    \minS^2(\mu_1, \mu_2) := \underset{\theta\in\SS^{d-1}}{\min}\
    \PStheta^2(\mu_1, \mu_2)
    = \underset{\substack{\theta\in\SS^{d-1} \\ 
    \omega \in \Omega_\theta(\mu_1, \mu_2)}}
    {\min}\ \int_{\R^{2d}}\|x_1-x_2\|_2^2\dd\omega(x_1, x_2),
\end{equation}
where we used the notation $\Omega_\theta(\mu_1, \mu_2)$ defined in
\cref{eqn:def_Omega_theta}, and \cref{thm:S_theta_equals_W_theta}. We show below
that the infimum is attained:
\begin{prop}\label{prop:minS_attained} Let $\mu_1, \mu_2 \in
    \mathcal{P}_2(\R^d)$. Then the minimum in
    \cref{eqn:def_min_sliced_discrepancy} is attained.
\end{prop}
\begin{proof}
    Take a sequence $(\theta_n)_{n\in\N} \in \SS^{d-1}$ such that
    \[\PStheta[\theta_n](\mu_1, \mu_2)
    \xrightarrow[n\longrightarrow+\infty]{}\minS(\mu_1, \mu_2).\] By compactness
    of $\SS^{d-1}$, we can extract a converging subsequence of $(\theta_n)$: up
    to extraction we can assume that $\theta_n
    \xrightarrow[n\longrightarrow+\infty]{}\theta \in \SS^{d-1}$. Denoting
    $\mu_{\theta_n}:=\mu_{\theta_n}[\mu_1, \mu_2]$, by
    \cref{prop:W_nu_inf_attained}, for each $n\in \N$ we can choose $\rho_n \in
    \Gamma(\mu_{\theta_n}, \mu_1, \mu_2)$ optimal for $\PStheta[\theta_n](\mu_1,
    \mu_2)$. By \cref{prop:PStheta_lsc}, we have $\mu_{\theta_n}
    \xrightarrow[n\longrightarrow +\infty]{w}
    \mu_\theta.$ Using \cref{lemma:tightness_Gamma3} item 1) we obtain that the
    set of $\Gamma(\{\mu_{\theta_n}\}, \mu_1, \mu_2)$ is tight in
    $\mathcal{P}_2(\R^{3d})$, and since $(\rho_n) \in \Gamma(\{\mu_{\theta_n}\},
    \mu_1, \mu_2)^\N$, there exists an extraction $\alpha$ such that
    $\rho_{\alpha(n)} \xrightarrow[n\longrightarrow +\infty]{w}\rho \in
    \mathcal{P}_2(\R^{3d})$. Applying \cref{lemma:tightness_Gamma3} item 2)
    shows that $\rho \in \Gamma(\mu_\theta, \mu_1, \mu_2)$.    

    The cost function $J := \rho \in \mathcal{P}_2(\R^{3d}) \longmapsto
    \int_{\R^{3d}}\|x_1-x_2\|_2^2\dd \rho_{1, 2}(y, x_1, x_2)$ is lower
    semi-continuous by \cite[Lemma 1.6]{santambrogio2015optimal}, which provides
    the following inequality:
    $$\PStheta^2(\mu_1, \mu_2) \leq J(\rho)\leq \underset{n\longrightarrow
    +\infty}{\liminf}\ J(\rho_{\alpha(n)}) = \underset{n\longrightarrow
    +\infty}{\lim}\ S_{\theta_{\alpha(n)}}^2(\mu_1, \mu_2) = \minS^2(\mu_1,
    \mu_2),$$ where the first inequality holds by the property  $\rho \in
    \Gamma(\mu_\theta, \mu_1, \mu_2)$, the second inequality comes from the
    lower semi-continuity of $J$ and the first equality comes from the fact that
    $\forall n \in \N, \; J(\rho_n) = \PStheta[\theta_n]^2(\mu_1, \mu_2)$. We
    conclude that $\minS(\mu_1, \mu_2) = \PStheta(\mu_1, \mu_2)$ and thus the
    infimum is attained.
\end{proof}

From \cref{prop:minS_attained}, we conclude that the properties of $\PStheta$
stated in \cref{prop:S_theta_semi_metric} are inherited by $\minS$. In
\cref{ex:ce_triangle_minPS}, we show an example which numerically contradicts
the triangle inequality.

\begin{example}[$\minS$ does not verify the triangle
    inequality]\label{ex:ce_triangle_minPS} We consider a setting with three
    measures $\mu_1, \mu_2, \mu_3 \in \mathcal{P}(\R^2)$ with 10 points each,
    obtained with five rotations of the example from
    \cref{ex:ce_S_theta_triangle}, which we represent in
    \cref{fig:ce_triangle_minPS}. Extensive numerical approximation with
    $L:=10^5$ directions yields the following violation of the triangle
    inequality:
    \[\minS(\mu_1, \mu_3) + \minS(\mu_3, \mu_2) - \minS(\mu_1, \mu_2) \approx
    -0.612. \] While the expressions are not tractable in closed form, this
    numerical experiment strongly suggests that the triangle inequality does not
    hold for $\minS$.
\end{example}

\begin{figure}[ht]
    \centering
    \includegraphics[width=0.6\textwidth]{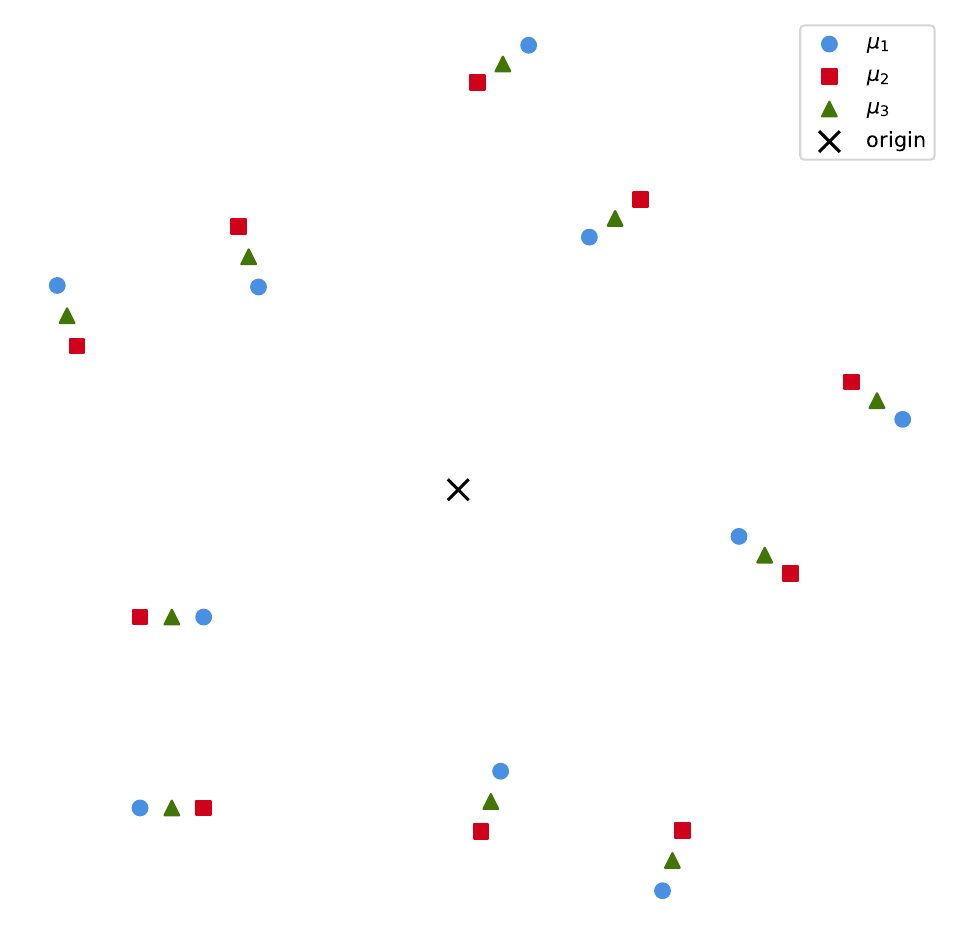}
    \caption{Counter-example from \cref{ex:ce_triangle_minPS} to the triangle
    inequality for $\minS$.}
    \label{fig:ce_triangle_minPS}
\end{figure}

\subsection{Equality with the Wasserstein Distance for Certain Discrete Measures}

In \cite[Proposition 3.2]{mahey23fast}, the authors show (proof in \cite[Section
11.1]{mahey23fast}) that the Min-Sliced Discrepancy $\minS$ equals $\W_2$ on the
set $\mathcal{P}^n(\R^d)$ of uniform discrete measures with $n$ points under a
condition on $n$ and $d$. Their proof relies on an application of
\cite{cover1967number} which requires the points to be in general position (see
\cref{def:general_position}), however the condition is not stated in
\cite{mahey23fast}. For the sake of clarity, we restate the result and provide a
detailed proof. First, we remind the notion of points of $\R^d$ in general
position in \cref{def:general_position}.

\begin{definition}\label{def:general_position} Let $x_1, \cdots, x_n \in \R^d$.
    We say that the points are in general position if for all $k\in \llbracket
    1, d \rrbracket$, there is no subset $I\subset\llbracket 1, n \rrbracket$
    with $k+2$ elements such that $\{x_i\}_{i\in I}$ is contained in a
    $k$-dimensional affine subspace of $\R^d$.
\end{definition}

\begin{prop}\label{prop:condition_minS_equals_W2} Let $\mu :=
    \tfrac{1}{n}\sum_{i=1}^n\delta_{x_i}$ and $\nu :=
    \tfrac{1}{n}\sum_{i=1}^n\delta_{y_i}$ such that the union of supports
    $(x_i)\cup(y_j)$ is in general position. If $2n\leq d+1$, then:
    \begin{align}
        &\Bigl\{(\sigma, \tau) \in \Perm_n^2 : \exists \theta \in \SS^{d-1}: 
        \forall i\in \llbracket 1, n-1 
        \rrbracket,\nonumber\\ 
        &\quad P_\theta x_{\sigma(i)}<  P_\theta x_{\sigma(i+1)},\;
        P_\theta y_{\tau(i)}<  P_\theta y_{\tau(i+1)}
        \Bigr\} = \Perm_n^2.\label{eqn:projected_permutations_all}
    \end{align}
    As a result, $\minS(\mu, \nu) = \W_2^2(\mu, \nu).$
\end{prop}
\begin{proof}
    By the Theorem in \cite[Section 2]{cover1967number} and \cite[Equation (12)
    in Section 3]{cover1967number}, since $(x_1, \cdots, x_n, y_1, \cdots, y_n)$
    are in general position and $d\geq 2n-1$, we have
    \begin{equation}\label{eqn:projected_permutations_all_2}
        \left\{\alpha \in \Perm_{2n} : \exists \theta \in \SS^{d-1}:
        \forall k\in \llbracket 1, 2n-1 \rrbracket,\;
        P_\theta z_{\alpha(i)}<  P_\theta z_{\alpha(i+1)} \right\} = \Perm_{2n},
    \end{equation}
    where $(z_1, \cdots, z_{2n}) := (x_1, \cdots, x_n, y_1, \cdots, y_n)$. Take
    now $(\sigma, \tau) \in \Perm_n^2$ and define $\alpha\in \Perm_{2n}$ by:
    $$\forall i \in \llbracket 1, n \rrbracket,\; \alpha(i) := \sigma(i),\quad
    \forall j \in \llbracket 1, n \rrbracket,\; \alpha(n+j) := \tau(j).$$ By
    \cref{eqn:projected_permutations_all_2}, there exists $\theta \in \SS^{d-1}$
    such that $\forall k\in \llbracket 1, 2n-1 \rrbracket,\; P_\theta
    z_{\alpha(i)}<  P_\theta z_{\alpha(i+1)}$, showing
    \cref{eqn:projected_permutations_all}.

    Now by definition of $\Perm_\theta(X, Y)$, the RHS term of
    \cref{eqn:projected_permutations_all} is a subset of $\cup_{\theta \in
    \SS^{d-1}} \Perm_\theta(X, Y)$, which shows that $\cup_{\theta \in
    \SS^{d-1}} \Perm_\theta(X, Y) = \Perm_n^2$. Using
    \cref{thm:S_theta_equals_W_theta} and \cref{thm:CW_theta_point_clouds} we
    conclude:
    \begin{align*}
        \minS(\mu, \nu) 
        &= \underset{(\sigma, \tau) \in \cup_{\theta \in
        \SS^{d-1}}\Perm_\theta(X, Y)}{\min}\ \frac{1}{n}\sum_{i=1}^n
        \|x_{\sigma(i)}-y_{\tau(i)}\|_2^2 \\
        &= \underset{(\sigma, \tau) \in
        \Perm_n^2}{\min}\ \frac{1}{n}\sum_{i=1}^n \|x_{\sigma(i)}-y_{\tau(i)}\|_2^2\\
        &= \W_2^2(\mu, \nu),
    \end{align*}
    where the last equality comes from the Monge formulation of $\W_2$ in the
    case of uniform measures with the same number of points (see
    \cite[Proposition 2.1]{computational_ot} for instance).
\end{proof}

%% file: sections/expected_sliced.tex
\section{Expected Sliced Wasserstein}\label{sec:expected_sliced}

In \cite{liu2024expected}, Liu et al. present a variant of the Sliced
Wasserstein distance, consisting in taking the transport cost of a coupling that
is an average of lifted sliced couplings. In this section, we will explain how
to define these notions for general measures of $\mathcal{P}_2(\R^d)$ instead of
discrete ones.

\subsection{Lifting Sliced Plans}

To lift a 1D transport plan onto $\R^{d}$, we will require the notion of
disintegration of measures with respect to a map reminded in
\cref{sec:nu_based_wass_disintegration}. Let $\mu_1, \mu_2 \in
\mathcal{P}_2(\R^d)$ and $\theta \in \SS^{d-1}$. Consider the disintegration of
$\mu_1$ with respect to $P_\theta := x\longmapsto x\cdot\theta$ as in
\cref{def:disintegration_P}: $\mu_1(\dd x) = (P_\theta\#\mu_1)(P_\theta\dd x)
\mu_1^{P_\theta x}(\dd x)$. The kernel $\mu_1^{P\theta x}$ is a measure on
$\R^d$ supported on the slice $\{x'\in \R^d\mid x'\cdot \theta = x\cdot\theta\}
= x +\theta^\perp$. Denoting similarly the disintegration of $\mu_2$ by
$\mu_2(\dd y) = (P_\theta\#\mu_2)(P_\theta\dd y) \mu_2^{P_\theta y}(\dd y)$, we
first notice that the disintegration of $\mu_1\otimes\mu_2$ with respect to
$(P_\theta, P_\theta)$ writes simply as a product:
\begin{equation}\label{eqn:disintegration_otimes}
    \mu_1\otimes\mu_2(\dd x, \dd y) = (P_\theta\#\mu_1)(P_\theta\dd
    x)(P_\theta\#\mu_2)(P_\theta\dd y) \mu_1^{P_\theta x}(\dd x)\mu_2^{P_\theta
    y}(\dd y),
\end{equation}
noticing that $(P_\theta, P_\theta)\#(\mu_1\otimes \mu_2) =
(P_\theta\#\mu_1)\otimes(P_\theta\#\mu_2)$. 

Take now the 1D OT plan $\pi_\theta := \pi_\theta[\mu_1, \mu_2] \in
\Pi^*(P_\theta\#\mu_1, P_\theta\#\mu_2)$, the idea behind the lift is to replace
the independent coupling $(P_\theta, P_\theta)\#(\mu_1\otimes \mu_2) =
(P_\theta\#\mu_1)\otimes(P_\theta\#\mu_2)$ in \cref{eqn:disintegration_otimes}
by $\pi_\theta$: we define the lifted plan through its disintegration as:
\begin{equation}\label{eqn:lifted_plan_symbolic}
    \gamma_\theta(\dd x, \dd y) := \pi_\theta(P_\theta\dd x, P_\theta\dd y) \mu_1^{P_\theta x}(\dd x)\mu_2^{P_\theta y}(\dd y).
\end{equation}
More formally, we can define $\gamma_\theta$ using test functions $\phi\in
\mathcal{C}_b^0(\R^{d}\times\R^{d})$:
\begin{equation}\label{eqn:lifted_plan_test_functions}
    \int_{\R^{2d}}\phi(x, y)\dd\gamma_\theta(x, y) = 
    \int_{\R^2}\left(\int_{P_\theta^{-1}(s)\times P_\theta^{-1}(t)}
    \phi(x, y)\dd\mu_1^{s}(x)\dd\mu_2^t(y)\right)\dd\pi_\theta(s, t).
\end{equation}
We illustrate the definition of the lifted plan on a simple example in
\cref{fig:independent_lifted_plan}.
\begin{figure}[H]
    \centering
    \includegraphics[width=0.7\textwidth]{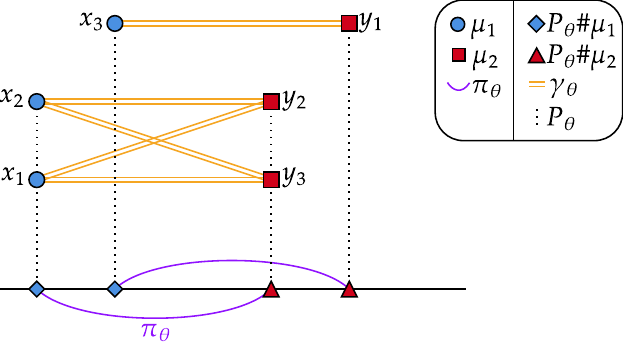}
    \caption{Example of the lifted plan $\gamma_\theta$ between two measures
    $\mu_1$ and $\mu_2$. In this case, we notice that $P_\theta x_1 = P_\theta
    x_2$ and $P_\theta y_2 = P_\theta y_3$. As a result, the optimal plan
    $\pi_\theta$ between $P_\theta\#\mu_1$ and $P_\theta\#\mu_2$ does not allow
    us to deduce an assignment between $(x_1, x_2)$ and $(y_2, y_3)$. The lifted
    coupling $\gamma_\theta$ chooses the independent coupling: $x_1$ is assigned
    uniformly to $(y_2, y_3)$ and likewise for $x_2$. As for $x_3$ and $y_1$,
    the coupling $\pi_\theta$ assigns $P_\theta x_3$ to $P_\theta y_1$ which
    imposes that $\gamma_\theta$ send $x_3$ to $y_1$.}
    \label{fig:independent_lifted_plan}
\end{figure}

In \cref{prop:elementary_props_lifted_plan} we show that the lifted plan is a
valid coupling between $\mu_1$ and $\mu_2$. We also provide an explicit
expression for discrete measures $\mu_1, \mu_2$, which coincides with the
expression in \cite[Equation 9]{liu2024expected}, which serves as their
definition of lifted plans.
\begin{prop}\label{prop:elementary_props_lifted_plan} Let $\mu_1, \mu_2 \in
    \mathcal{P}_2(\R^d)$, $\theta \in \SS^{d-1}$ and $\gamma_\theta :=
    \gamma_\theta[\mu_1, \mu_2]$ the lifted plan defined in
    \cref{eqn:lifted_plan_test_functions}. Then:
    \begin{enumerate}
        \item $\gamma_\theta \in \Pi(\mu_1, \mu_2)$.
        \item If $\mu_1 = \sum_{i=1}^na_i\delta_{x_i}$ and $\mu_2 =
        \sum_{j=1}^mb_j\delta_{y_j}$, let $\pi_\theta \in \Pi^*(P_\theta\#\mu_1,
        P_\theta\#\mu_2)$. For $(i,j) \in \llbracket 1, n \rrbracket \times
        \llbracket 1, m \rrbracket$, we define $Q_{i,j} :=
        \pi_\theta(\{(P_\theta x_i, P_\theta y_j)\})$, which allows us to see
        $\pi_\theta$ as a matrix of size $n\times m$.
        \begin{equation}\label{eqn:lifted_plan_discrete}
            \gamma_\theta = \sum_{i=1}^n\sum_{j=1}^m 
            \cfrac{a_ib_j}{A_iB_j} Q_{i,j}\delta_{(x_i, y_j)},
            \quad A_i := \sum_{i': x_{i'}\cdot\theta = x_i\cdot\theta} a_{i'},
            \; B_j := \sum_{j': y_{j'}\cdot\theta = y_j\cdot\theta} b_{j'}.
        \end{equation}
    \end{enumerate}
\end{prop}
\begin{proof}
    For 1. we verify the property using a test function $\phi \in
    \mathcal{C}_b^0(\R^d)$ with \cref{eqn:lifted_plan_test_functions}:
    \begin{align*}
        \int_{\R^{2d}}\phi(x)\dd\gamma_\theta(x, y) &= 
        \int_{\R^2}\left(\int_{P_\theta^{-1}(s)\times P_\theta^{-1}(t)}
        \phi(x)\dd\mu_1^{s}(x)\dd\mu_2^t(y)\right)\dd\pi_\theta(s, t) \\
        &= \int_{\R^2}\left(\int_{P_\theta^{-1}(s)}\phi(x)\dd\mu_1^{s}(x)\right)
        \underbrace{\left(\int_{P_\theta^{-1}(t)}\dd\mu_2^{t}(y)\right)}_{=1}
        \dd\pi_\theta(s, t) \\
        &= \int_\R \left(\int_{P_\theta^{-1}(s)}\phi(x)\dd\mu_1^{s}(x)\right) 
        \dd(P_\theta\#\mu_1)(s) \\
        &= \int_{\R^d}\phi(x)\dd\mu_1(x),
    \end{align*}
    where we use the fact that the first marginal of $\pi_\theta$ is
    $P_\theta\#\mu_1$, and then use disintegration of $\mu_1$ with respect to
    $P_\theta$. The same method shows that the second marginal of
    $\gamma_\theta$ is $\mu_2$, concluding $\gamma_\theta \in \Pi(\mu_1,
    \mu_2)$.

    For 2. we begin by writing explicitly the disintegration of $\mu_1$ with
    respect to $P_\theta$. For $\mu_1 = \sum_{i=1}^na_i\delta_{x_i}$, we have
    $P_\theta\#\mu = \sum_i a_i \delta_{P_\theta x_i}$, and for $s=P_\theta x_i
    \in \supp(P_\theta\#\mu)$, we have $\mu_1^{s} = A_i^{-1}\sum_{i': x_i'\cdot
    \theta = s} a_{i'}\delta_{x_{i'}}$. We establish
    \cref{eqn:lifted_plan_discrete} by testing on $\phi \in
    \mathcal{C}_b^0(\R^{2d})$. The support of $\pi_\theta$ is (at most) the
    family of pairs $((x_i\cdot\theta, y_j\cdot\theta))_{i,j}$. We choose $I
    \subset \llbracket 1, n \rrbracket \times \llbracket 1, m \rrbracket$ such
    that the support of $\pi_\theta$ is the injective family
    $\left((x_i\cdot\theta, y_j\cdot\theta)\right)_{(i,j) \in I}$. We then have:
        \begin{align*}
            \int_{\R^{2d}}\phi(x, y)\dd\gamma_\theta(x, y) &= 
            \int_{\R^2}\left(\int_{P_\theta^{-1}(s)\times P_\theta^{-1}(t)}
            \phi(x, y)\dd\mu_1^{s}(x)\dd\mu_2^t(y)\right)\dd\pi_\theta(s, t) \\
            &= \sum_{(i,j) \in I} \Biggl(\sum_{\substack{
            i': P_\theta x_{i'} = P_\theta x_i \\ 
            j': P_\theta y_{j'} = P_\theta y_j}}\phi(x_{i'}, y_{j'})
            \cfrac{a_{i'}b_{j'}}{A_iB_j}\Biggr) 
            \pi_\theta(\{(P_\theta x_i, P_\theta y_j)\}) \\
            &= \sum_{i=1}^n\sum_{j=1}^m \phi(x_i, y_j)\cfrac{a_i b_j}{A_i B_j}
            Q_{i,j},
        \end{align*}
    where we use the fact that for $(i,j)\in I$ and $(i', j')$ such that
    $P_\theta x_{i'} = P_\theta x_i$ and $P_\theta y_{j'} = P_\theta y_j$, it
    holds that $A_i = A_{i'},\; B_j = B_{j'}$ and $Q_{i,j} = Q_{i',j'}$.
\end{proof}
The discrete expression in \cref{eqn:lifted_plan_discrete} shows that the
definition of listed plans in \cref{eqn:lifted_plan_test_functions} is a
generalisation of the plan lift from \cite[Equation 9]{liu2024expected}. We now
study the transport cost associated to the lifted plan $\gamma_\theta$:
\begin{definition}\label{def:independent_lifted_cost} For $\theta\in \SS^{d-1}$
    and $\mu_1, \mu_2 \in \mathcal{P}_2(\R^d)$. With $\gamma_\theta[\mu_1,
    \mu_2]$ the lifted plan defined in \cref{eqn:lifted_plan_test_functions}, we
    define the lifted cost as:
    $$\LStheta^2(\mu_1, \mu_2) :=
    \int_{\R^{2d}}\|x_1-x_2\|_2^2\dd\gamma_\theta[\mu_1, \mu_2](x_1, x_2). $$
\end{definition}
We will see that $\LStheta$ defines a discrepancy on $\mathcal{P}_2(\R^d)$ that
is almost a distance.
\begin{prop}\label{prop:lifted_cost_almost_distance} Fix $\theta\in \SS^{d-1}$.
    The quantity $\LStheta$ is non-negative, symmetric, verifies the triangle
    inequality, and if $\mu_1,\mu_2 \in \mathcal{P}_2(\R^d)$ verify
    $\LStheta(\mu_1, \mu_2) = 0$ then $\mu_1=\mu_2$. Furthermore, we have the
    inequality $\LStheta \geq \W_2$.
\end{prop}
\begin{proof}
    Non-negativity and symmetry are immediate. For $\mu_1,
    \mu_2\in\mathcal{P}_2(\R^d)$, by \cref{prop:elementary_props_lifted_plan},
    we have $\gamma_\theta[\mu_1, \mu_2] \in \Pi(\mu_1, \mu_2)$, hence
    $\LStheta(\mu_1, \mu_2)\geq\W_2(\mu_1, \mu_2)$. Suppose now that
    $\mu_1,\mu_2 \in \mathcal{P}_2(\R^d)$ are such that $\LStheta(\mu_1, \mu_2)
    = 0$, then $\W_2(\mu_1, \mu_2) =0$ and therefore $\mu_1 = \mu_2$.

    We now show the triangle inequality: let $\mu_1, \mu_2, \mu_3\in
    \mathcal{P}_2(\R^d)$. We consider the sliced 3-plan $\eta_\theta$ defined
    by:
    $$\eta_\theta := \left(F_{P_\theta\#\mu_1}^{[-1]},
    F_{P_\theta\#\mu_2}^{[-1]}, F_{P_\theta\#\mu_3}^{[-1]}\right) \#\Leb_{[0,
    1]}. $$ For $i< j\in \{1, 2, 3\}$, introduce $\pi_\theta^{(i,j)}$ the unique
    optimal transport plan between $P_\theta\#\mu_i$ and $P_\theta\#\mu_j$. By
    \cite[Theorem 2.9]{santambrogio2015optimal}, we see that
    $[\eta_\theta]_{i,j}=\pi_\theta^{(i,j)}$. We now lift the sliced plan
    $\eta_\theta$ in the same manner as in
    \cref{eqn:lifted_plan_test_functions}, defining a plan $\rho_\theta \in
    \mathcal{P}_2(\R^{3d})$ by disintegration:
    $$\rho_\theta(\dd x_1, \dd x_2, \dd x_3) := \eta_\theta(P_\theta\dd x_1,
    P_\theta \dd x_2, P_\theta \dd x_3)\mu_1^{P_\theta x_1}(\dd x_1)
    \mu_2^{P_\theta x_2}(\dd x_2)\mu_3^{P_\theta x_3}(\dd x_3).$$ By computing
    the expectation against test functions, $\forall i<j\in \{1, 2, 3\},\;
    [\rho_\theta]_{i,j} = \gamma_\theta[\mu_i, \mu_j]$. We now use the classical
    gluing method (as in \cite[Lemma 5.5]{santambrogio2015optimal}) to show the
    triangle inequality, introducing the functions $\phi_i := (x_1, x_2, x_3)
    \longmapsto x_i$ for $i\in \{1, 2, 3\}$:
    \begin{align*}
        \LStheta(\mu_1, \mu_3)
        &= \sqrt{\int_{\R^{2d}}\|x_1-x_3\|_2^2\dd
        \gamma_\theta[\mu_1, \mu_3](x_1, x_3)} \\
        &= \sqrt{\int_{\R^{3d}}\|x_1-x_3\|_2^2
        \dd\rho_\theta(x_1, x_2, x_3)} \\
        &= \|\phi_1 - \phi_3\|_{L^2(\rho_\theta)} \\
        &\leq \|\phi_1 - \phi_2\|_{L^2(\rho_\theta)} 
        + \|\phi_2 - \phi_3\|_{L^2(\rho_\theta)} \\
        &= \sqrt{\int_{\R^{2d}}\|x_1-x_2\|_2^2
        \dd\gamma_\theta[\mu_1, \mu_2](x_1, x_2)}\\
        &\quad+\sqrt{\int_{\R^{2d}}\|x_2-x_3\|_2^2
        \dd\gamma_\theta[\mu_2, \mu_3](x_1, x_2)} \\
        &= \LStheta(\mu_1, \mu_2) + \LStheta(\mu_2, \mu_3).
    \end{align*}
    \vskip -20pt
\end{proof}

The discrepancy $\LStheta$ is not a distance on $\mathcal{P}_2(\R^d)$: in
\cref{ex:EStheta_positive_self_distance}, we introduce a particular case in
dimension two where $\LStheta(\mu, \mu)>0$. 
\begin{example}[$\LStheta(\mu, \mu)$ can be
    non-zero]\label{ex:EStheta_positive_self_distance} Take $\theta := (1, 0)$
    and $\mu := \frac{1}{2}(\delta_{x_0} + \delta_{x_1}),\; x_0 := (0, 0),\; x_1
    := (0, 1)$. We have $P_\theta\#\mu = \delta_0$ and thus $\gamma_\theta[\mu,
    \mu] = \mu\otimes\mu$. The lifted cost is then:
    $$\LStheta^2(\mu, \mu) = \cfrac{1}{4}\left(\|x_0 - x_0\|_2^2 + \|x_0 -
    x_1\|_2^2 + \|x_1 - x_0\|_2^2 + \|x_1 - x_1\|_2^2\right) = \cfrac{1}{2}>0.
    $$ 
\end{example}
For probability measures $\mu$ with countable support, for almost-every 
$\theta\in \SS^{d-1}$, there is no ambiguity in the projections, and thus 
$\LStheta(\mu, \mu)=0$, as shown in \cref{prop:ES_theta_discrete_as_distance}. 
To state the result, we introduce the following notation for the set of 
probability measures with countable\footnote{by ``countable", we mean a set 
that is either finite or equipotent to $\N$.} support:
\begin{equation}\label{eqn:countably_discrete}
    \mathcal{P}_{\mathrm{DC}}(\R^d) := 
    \left\{\mu = \sum_{x\in X} a_x\delta_x :
    X\subset\R^d\text{ countable},\; (a_x)_{x\in X} \in (0, 1]^d,
    \; \sum_{x\in X} a_x = 1\right\}.
\end{equation}
\begin{prop}\label{prop:ES_theta_discrete_as_distance}
    Consider $\mu\in \mathcal{P}_{\mathrm{DC}}(\R^d)$, then for almost-every 
    $\theta\in \SS^{d-1}$, we have $\LStheta(\mu, \mu)=0$.
\end{prop}
\begin{proof}
    Denoting $\bbsigma_u$ the uniform measure on the unit sphere $\SS^{d-1}$, we
    have by countable additivity $\P_{\theta\sim\bbsigma_u}(\exists x\neq y \in
    X^2 : P_\theta x = P_\theta y) \leq \sum_{x\neq y\in
    X^2}\bbsigma_u((x-y)^\perp) = 0$. We now fix $\Theta\subset \SS^{d-1}$ the
    set $\Theta := \{\theta \in \SS^{d-1} : \forall x \neq y \in X^2,\; \theta
    \not\in (x-y)^\perp \}$, we have shown that $\bbsigma_u(\Theta) = 1$. For
    $\theta \in \Theta$, the family $(P_\theta x)_{x\in X}$ is injective, and
    the disintegration kernel $\mu$ with respect to $P_\theta$ at $P_\theta x$
    is simply $\delta_{Q_{\theta^\perp} x}$, therefore the lifted plan
    $\gamma_\theta[\mu, \mu]$ is $\sum_{x\in X}a_x\delta_{(x, x)}$. It
    follows from the definition that $\LStheta(\mu, \mu)=0$ for any $\theta \in
    \Theta$, concluding the proof.
\end{proof}

\subsection{Averaging Lifted Plans}

Let $\mu_1, \mu_2 \in \mathcal{P}_2(\R^d)$ and $\theta \in \SS^{d-1}$. We have
constructed a lifted plan $\gamma_\theta \in \Pi(\mu_1, \mu_2)$ (see
\cref{eqn:lifted_plan_test_functions,prop:elementary_props_lifted_plan}). We now
define the expected lifted plan as the ``average'' of lifted plans over all
directions $\theta \in \SS^{d-1}$ through a probability measure $\bbsigma \in
\mathcal{P}(\SS^{d-1})$. We define $\oll{\gamma}[\mu_1, \mu_2, \bbsigma]$ by 
duality on test functions $\phi\in \mathcal{C}_b^0(\R^d\times\R^d)$:
\begin{equation}\label{eqn:def_expected_lifted_plan}
    \int_{\R^{2d}}\phi(x, y)\dd\oll{\gamma}[\mu_1, \mu_2, \bbsigma](x, y) := 
    \int_{\SS^{d-1}}\int_{\R^{2d}}\phi(x, y)
    \dd\gamma_\theta[\mu_1, \mu_2](x, y)\dd\bbsigma(\theta).
\end{equation}
Having defined the expected lifted plan, we can now define the expected sliced
discrepancy:
\begin{definition}\label{def:expected_sliced} Let $\mu_1, \mu_2 \in
    \mathcal{P}_2(\R^d)$ and $\bbsigma \in \mathcal{P}(\SS^{d-1})$. The expected
    sliced discrepancy between $\mu_1$ and $\mu_2$ is defined as:
    \begin{align*}
        \ES_{\bbsigma}^2(\mu_1, \mu_2) &:= \int_{\R^{2d}}\|x-y\|_2^2
        \dd\oll{\gamma}[\mu_1, \mu_2, \bbsigma](x, y) \\
        &= \int_{\SS^{d-1}}\int_{\R^{2d}}\|x-y\|_2^2 
        \dd\gamma_\theta[\mu_1, \mu_2](x, y)\dd\bbsigma(\theta) \\
        &= \int_{\SS^{d-1}}\LStheta^2(\mu_1, \mu_2)\dd\bbsigma(\theta).
    \end{align*}
    where $\oll{\gamma}[\mu_1, \mu_2, \bbsigma]$ is the expected lifted plan
    between $\mu_1$ and $\mu_2$ for the measure $\bbsigma$ on $\SS^{d-1}$,
    defined in \cref{eqn:def_expected_lifted_plan}, and $\gamma_\theta[\mu_1,
    \mu_2]$ is the lifted plan defined in \cref{eqn:lifted_plan_test_functions}.
\end{definition}
The properties of $\LStheta$ are passed on to $\ES_\bbsigma$ by integration.
\begin{corollary}\label{cor:expected_sliced_almost_distance} For any probability
    measure $\bbsigma$ on $\SS^{d-1}$, the quantity $\ES_\bbsigma$ is
    non-negative, symmetric, verifies the triangle inequality, and if
    $\mu_1,\mu_2 \in \mathcal{P}_2(\R^d)$ verify $\ES_\bbsigma(\mu_1, \mu_2) =
    0$ then $\mu_1=\mu_2$. Furthermore, we have the inequality $\ES_\bbsigma
    \geq \W_2$.
\end{corollary}
\begin{proof}
    Non-negativity, symmetry and the property $\ES_\bbsigma \geq \W_2$ are 
    immediate by applying the definition and 
    \cref{prop:lifted_cost_almost_distance}. For the triangle inequality, let 
    $\mu_1, \mu_2, \mu_3 \in \mathcal{P}_2(\R^d)$ and for 
    $i<j \in \{1, 2, 3\}$, introduce $f_{i,j} := \theta \longmapsto 
    \LStheta(\mu_i, \mu_j).$ By \cref{prop:lifted_cost_almost_distance} we have 
    $0\leq f_{1,3} \leq f_{1, 2} + f_{2, 3}$. We write:
    \begin{align*}
        \ES_\bbsigma(\mu_1, \mu_3) &= \|f_{1,3}\|_{L^2(\bbsigma)} \\
        &\leq \|f_{1,2} + f_{2,3}\|_{L^2(\bbsigma)} \\
        &\leq \|f_{1,2}\|_{L^2(\bbsigma)} + \|f_{2,3}\|_{L^2(\bbsigma)}\\
        &= \ES_\bbsigma(\mu_1, \mu_2) + \ES_\bbsigma(\mu_2, \mu_3).
    \end{align*}
    \vskip -10pt
\end{proof}
The discrepancy $\ES_\bbsigma$ is not a distance on $\mathcal{P}_2(\R^d)$.
First, if $d=2$ and $\bbsigma = \delta_{(1, 0)}$, $\ES_\bbsigma = \LStheta$ and
the counter-example from \cref{ex:EStheta_positive_self_distance} earlier with 
$\mu := \tfrac{1}{2}\delta_{(0, 0)} + \tfrac{1}{2}\delta_{(0, 1)}$ yields 
$\ES_\bbsigma(\mu, \mu)>0$.

Even for probability measures $\bbsigma$ that are absolutely continuous with 
respect to the uniform measure on $\SS^{d-1}$, we can find examples where
$\ES_\bbsigma(\mu, \mu)>0$, as presented in 
\cref{ex:ES_positive_self_distance}.

\begin{example}[Case where $\ES_\bbsigma(\mu, \mu)>0$ for any $\bbsigma$]
    \label{ex:ES_positive_self_distance}
    Take $\bbsigma$ any probability measure on
    $\SS^{1}$ and $\mu := \mathcal{U}(B_{\R^2}(0, 1))$ the uniform measure on 
    the Euclidean unit ball of $\R^2$. We have for any $\theta \in \SS^1,\; 
    P_\theta\mu = \nu$, where $\nu(\dd t) = \frac{2\sqrt{1-t^2}}{\pi}\mathbbold
    {1}_{[-1, 1]}(t) \dd t$. The disintegration of $\mu$ with respect to 
    $P_\theta$ is covariant with respect to $\theta$, and the disintegration 
    kernel at $t=P_\theta x$ is $\mu^t=\mathcal{U}\left(\left\{t\theta + 
    v\theta_\perp,\; v \in [-\sqrt{1-t^2}, \sqrt{1-t^2}]\right\}\right)$, where 
    we have fixed $\theta_\perp$ a unit orthogonal vector to $\theta$. The 
    disintegration kernel $\mu^t$ is the uniform measure on the ball slice of 
    $B_{\R^2}(0, 1) \cap (t\theta + \theta^\perp)$ (with $\theta^\perp := \{x 
    \in \R^d : \theta \cdot x = 0\}$), and can simply be understood as the 
    uniform measure on the segment $[-\sqrt{1-t^2},\sqrt{1-t^2}]$, cast into 
    $\R^2$. The optimal transport plan between $\nu$ and itself is $\pi_\theta 
    := (I, I)\#\nu$, and it follows that the lifted plan between $\mu$ and 
    itself is (denoting $t := P_\theta x_1$ for legibility):
    \begin{align*}
        \gamma_\theta(\dd x_1, \dd x_2) &= \delta_{P_\theta x_1 = P_\theta x_2}
        (\dd P_\theta x_1, \dd P_\theta x_2)\ \nu(\dd P_\theta x_1)\\
        &\otimes \mathcal{U}\left(\left\{
        (t\theta + v_1\theta_\perp, t\theta + v_2\theta_\perp),
        \; (v_1,v_2) \in [-\sqrt{1-t^2},\sqrt{1-t^2}]^2\right\}\right) 
        (\dd x_1, \dd x_2).
    \end{align*}
    We provide a visualisation of the disintegration $\mu^t$ and the coupling
    $\gamma_\theta$ in \cref{fig:ce_expected_sliced_ball}.

    By symmetry $\LStheta(\mu, \mu)$ does not depend on $\theta$, we compute it 
    for $\theta:=(1, 0)$:
    \begin{align*}
        \LStheta^2(\mu, \mu) &= \int_{\R^3}\Biggr[
        \|(u, v_1) - (u, v_2)\|_2^2 \mathbbold{1}_{[-1, 1]}(u)
        \cfrac{2\sqrt{1-u^2}}{\pi} 
        \mathbbold{1}_{[-\sqrt{1-u^2}, \sqrt{1-u^2}]^2}(v_1, v_2) \\
        &\qquad\qquad \times\left(\cfrac{1}{2\sqrt{1-u^2}}\right)^2\Biggl]\dd v_1 \dd v_2\dd u  \\
        &= \int_{u=-1}^{u=1}\cfrac{\sqrt{1-u^2}}{2\pi}
        \int_{v_1=-\sqrt{1-u^2}}^{v_1=\sqrt{1-u^2}}
        \int_{v_2=-\sqrt{1-u^2}}^{v_2=\sqrt{1-u^2}} (v_1-v_2)^2 
        \dd v_1 \dd v_2 \dd u \\
        &= \cfrac{5\pi}{12} > 0.
    \end{align*}
    We conclude that $\ES_\bbsigma(\mu, \mu) > 0$ (for any $\bbsigma$), and 
    thus $\ES_\bbsigma$ is not a distance on $\mathcal{P}_2(\R^d)$.
\end{example}

\begin{figure}[ht]
    \centering
    \includegraphics[width=0.7\textwidth]{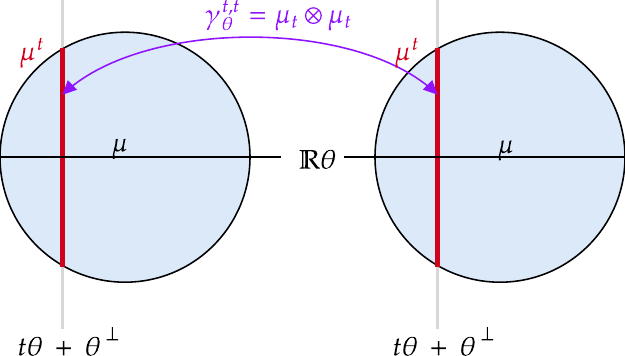}
    \caption{Illustration of the lifted plan $\gamma_\theta$ from 
    \cref{ex:ES_positive_self_distance} between $\mu$ the
    uniform measure on the unit Euclidean ball of $\R^2$, and itself. The plan
    is defined by disintegration: the coupling between $P_\theta\#\mu$ and
    $P_\theta\#\mu$ is simply $(I, I)$, the coupling induced by the identity
    map. As for the orthogonal part, the disintegration kernel of $\mu$ at
    $t\theta$ is $\mu^t$, the uniform measure on the ball slice $B_{\R^2}(0, 1)
    \cap (t\theta + \theta^\perp)$, represented as a thick red vertical line. 
    The lifted plan couples the disintegration
    kernel $\mu^t$ with itself with the independent coupling: writing
    $\gamma_\theta^{t,t}$ as the disintegration kernel of $\gamma_\theta$ at
    $(t, t) \in [-1, 1]^2$, we have $\gamma_\theta^{(t, t)} = 
    \mu_t\otimes\mu_t$, which corresponds to the uniform measure on the square 
    $\{(t\theta + v_1\theta_\perp, t\theta + v_2\theta_\perp),\; (v_1,v_2) \in 
    [-\sqrt{1-t^2}, \sqrt{1-t^2}]^2\}\subset \R^4$.}
    \label{fig:ce_expected_sliced_ball}
\end{figure}

Using \cref{prop:ES_theta_discrete_as_distance}, we can show that the expected
sliced distance is a distance on the set of ``countably discrete'' probability
measures defined in \cref{eqn:countably_discrete}.
\begin{corollary}\label{cor:expected_sliced_distance_countably_discrete} For any
    $\bbsigma$ a probability measure on $\SS^{d-1}$ that is absolutely
    continuous with respect to $\bbsigma_u$, the quantity $\ES_\bbsigma$ is a
    distance on $\mathcal{P}_{\mathrm{DC}}(\R^d)$.   
\end{corollary}
\begin{proof}
    Thanks to \cref{cor:expected_sliced_almost_distance}, the only axiom to
    verify to show that $\ES_\bbsigma$ is a distance on
    $\mathcal{P}_{\mathrm{DC}}(\R^d)$ is to show that $\forall \mu \in
    \mathcal{P}_{\mathrm{DC}}(\R^d),\; \ES_\bbsigma(\mu, \mu)= 0$. We now fix
    $\mu \in \mathcal{P}_{\mathrm{DC}}(\R^d)$. Since $\bbsigma \ll \bbsigma_u$,
    we have by \cref{prop:ES_theta_discrete_as_distance} that for
    $\bbsigma$-almost-every $\theta\in \SS^{d-1},\; \LStheta(\mu, \mu)=0$. We
    conclude $\ES_\bbsigma^2(\mu, \mu) = \int_{\SS^{d-1}}\LStheta^2(\mu,
    \mu)\dd\bbsigma(\theta) = 0$.
\end{proof}

\begin{remark}\label{rmk:centred_ES}
    Given that $\ES_\bbsigma(\mu, \mu)$ can be non-zero for some $\mu \in 
    \mathcal{P}_2(\R^d)$, a natural idea is to introduce a centred (or 
    ``de-biased'') version of $\ES_\bbsigma$ by subtracting the self-similarity 
    terms. For $\mu_1, \mu_2 \in \mathcal{P}_2(\R^d)$ and 
    $\bbsigma \in \mathcal{P}(\SS^{d-1})$, we define the
    centred expected sliced discrepancy as:
    $$\widetilde{\ES}^2_{\bbsigma}(\mu_1, \mu_2) 
    := \ES_{\bbsigma}^2(\mu_1, \mu_2) - \frac{1}{2}\ES_{\bbsigma}^2(\mu_1,
    \mu_1) - \frac{1}{2}\ES_{\bbsigma}^2(\mu_2, \mu_2).$$ By definition, for any
    $\mu \in \mathcal{P}_2(\R^d)$, we have $\widetilde{\ES}_\bbsigma(\mu, \mu) =
    0$. By symmetry of $\ES_\bbsigma$, we also have symmetry of
    $\widetilde{\ES}_\bbsigma$. Due to
    \cref{cor:expected_sliced_distance_countably_discrete}, this centering is
    only of interest for probability measures with non-countable support (most
    commonly absolutely continuous measures).
\end{remark}

%% file: sections/numerics.tex
\section{Numerics}
\label{sec:numerics}

In this section, we evaluate the efficiency and practicability of the
sliced-based transport plans, namely min-Pivot Sliced Wasserstein ($\minS$) and
expected Sliced Wasserstein ($\ES$), in both synthetic and real-world scenarios.
We begin by presenting quantitative and qualitative results on toy datasets,
evaluating their ability to generate meaningful transport plans and costs across
various settings. We continue with a colour transfer task, which is simple to
assess qualitatively yet can be computationally challenging in classic OT due to
the large sample size ($n \geq 500^2$). We finish with a more complex task that
involves large-scale datasets where a transport plan is required, namely point
cloud registration. For these experiments, we employ the POT
toolbox~\cite{flamary2021pot}. Note that we report experimental results only in
the context of distributions with the same number of samples, but the 
results can be easily extended to the case of a different number of samples. All
experiments were run on a CPU on a MacBook Pro with an M1 chip.  

\subsection{Evaluation of the Transport Losses and Plans}

\subsubsection{Gradient Flows} \label{sec:GF}

We begin by analysing sliced-based transport plans in the context of the
minimisation problem $\min_{\mu} \mathcal{F}^{\nu}(\mu) ,$ where
\(\mathcal{F}^{\nu}\) denotes the objective functional to be optimised, \(\nu\)
is a discrete target distribution, and \(\mu\) is a discrete source
distribution. One way to solve this problem is to construct a gradient flow
$\partial_t \mu_t = -\nabla \mathcal{F}^{\nu}(\mu_t)$ and to define an explicit
Euler scheme $\mu_{t+1}= \mu_t - h \nabla \mathcal{F}^{\nu}(\mu_t)$, with $h$ a
step size or learning rate.
This procedure yields a flow $(\mu_t)_{t}$ that decreases the functional
$\mathcal{F}^{\nu}(\mu)$ over time $0\leq t \leq 1$ \cite{bonnotte}. We consider
several functionals here: the Wasserstein distance $\W_2^2$, the Sliced
Wasserstein distance $\SW_2^2$, $\minS^2$, and $\ES^2$. For $\minS$ and $\ES$,
at each step $t$ we draw randomly $L$ directions $\theta_\ell \in \SS^{d-1}$ and
compute $\minS^2 \approx \min_\ell \PStheta[\theta_\ell]^2$ and $\ES \approx
\frac{1}{L} \sum_{\ell = 1}^L \LSthetaL^2$. Moreover, for $\minS$, we use an
optimisation scheme described in \cite{chapel2025differentiable} to obtain an
approximation $\hat{\theta}^\star$ of an optimal direction $\theta^\star \in
\argmin_{\theta \in \mathbb{S}^{d-1}}\PStheta^2$. In what follows, it is denoted
$\PStheta[\hat{\theta}^\star]^2$.

Following the experimental setting of~\cite{chapel2025differentiable}, we consider several target distributions of $n=50$ samples, shown in the first and third columns of \cref{fig:fig_gf}: a Gaussian distribution (in 2 and 500
dimensions), a spiral, two moons, a circle, and eight Gaussians of different
means. The source distribution is chosen to be a uniform distribution. We use
Adam as an optimisation scheme, with a learning rate of 0.02 for all methods,
and consider $L = 50$ directions for the sliced approaches. We report the
2-Wasserstein distance between $\mu_t$ and $\nu$ at each iteration of the
optimisation procedure and repeat each experiment 10 times. 

One can observe that the Expected Sliced discrepancy does not converge in any
setting. This finding is consistent with that of \cite[Section
3.4]{liu2024expected}. In contrast, all other methods enable convergence to the
target distribution, i.e. $\mu_t \to \nu$ as $t \to 1$, when working in two
dimensions. When considering a 500-dimensional Gaussian distribution, only
Wasserstein and $\PStheta[\hat{\theta}^\star]$ achieve convergence: with a fixed
number of samples $n$, we suspect that the required number of directions to
obtain a good approximation must grow exponentially with the dimension, making
$\minS$ (with $L=50$ directions), $\ES$ and Sliced Wasserstein inadequate for
this context. Using optimisation techniques in $\minS$ provides a single
meaningful direction $\hat{\theta}^\star$, even when $n$ is small compared to
the dimension. One can notice that Wasserstein and
$\PStheta[\hat{\theta}^\star]$ have very similar behaviours, which is backed by
\cref{prop:condition_minS_equals_W2} which states that
$\PStheta[\hat{\theta}^\star]$ is equal to the 2-Wasserstein distance when
$\hat{\theta}^\star$ is an optimal direction and when $d \geq 2n - 1$. This
encourages the use of the minimisation method proposed in
\cite{chapel2025differentiable}, which outperforms the search over $L$ random
projections.

\begin{figure}[H]
    \centering
    \includegraphics[width=1\textwidth]{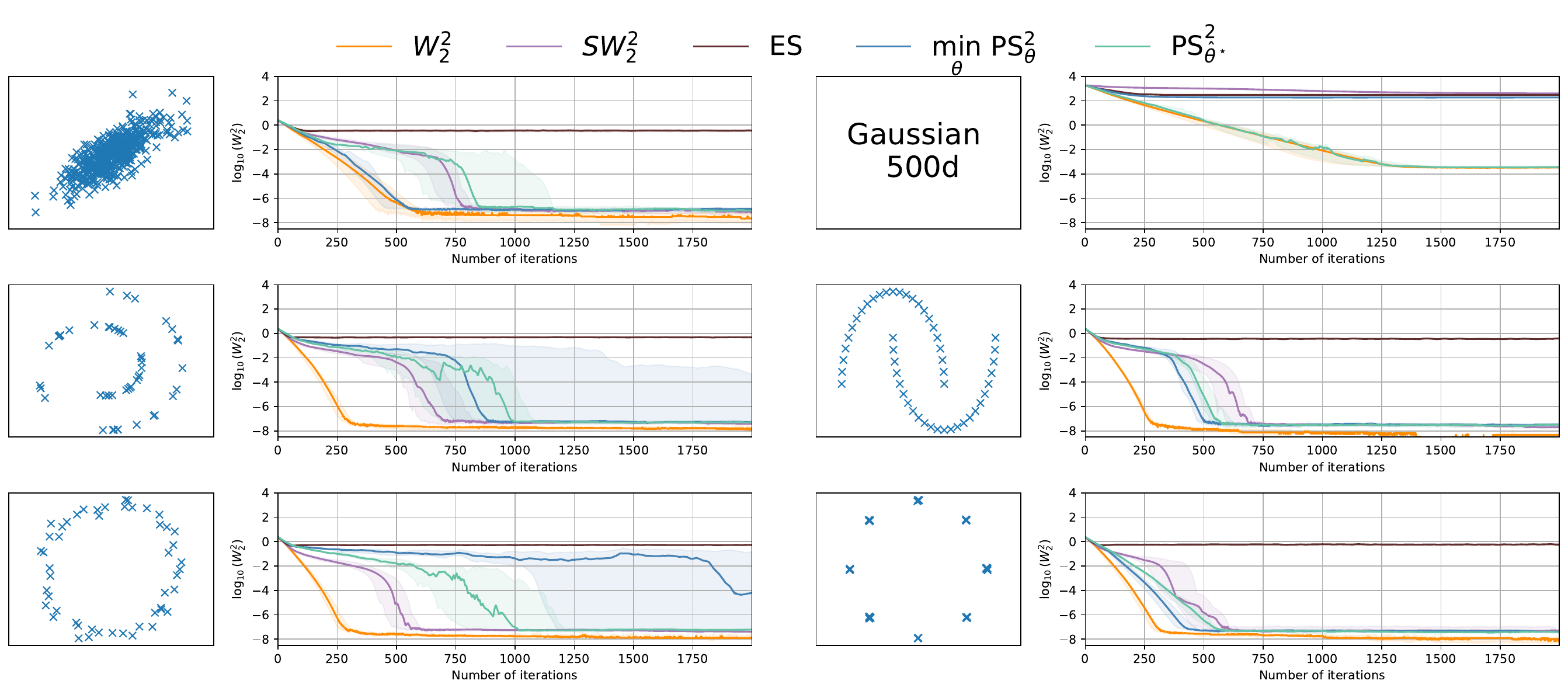}
    \caption{Log 2-Wasserstein distance measured between a source and different
    target distributions as a function of number of iterations. Plain lines
    represent the median over 10 iterations while shaded regions indicate 0.25
    and 0.75 quantiles.}
    \label{fig:fig_gf}
\end{figure}

\subsubsection{Comparison of Transport Plans and Discrepancies}
We now provide a quantitative assessment of the transport plans that can be
estimated from sliced-based methods. 

\paragraph{Qualitative assessment of the transport plans.}
We illustrate some transport plans in several two-dimensional settings. The
first one corresponds to transporting samples of a source Gaussian distribution
to samples of a target Gaussian distribution with different parameters. The
second one considers two distributions sampled on circles of the same centre and
different radii, with $n=24$ samples. The last one considers a more challenging
and non-linear setting, in which the source distribution is composed of 8
Gaussians of several means and the target is composed of two moons. 

\cref{fig:illus_plans_2d} presents the plans obtained with 2-Wasserstein,
Expected Sliced and min-Pivot Sliced, together with the associated discrepancy.
We choose $L=50$ directions and fix $n=10$ samples for the first scenario and
$n=24$ otherwise. One can notice that, in the simple case of 2 Gaussians as
source and target distributions (first line), the transport cost is close to the
2-Wasserstein one. Min-Pivot Sliced provides a plan that is close to the OT one;
Expected Sliced provides a highly non-deterministic coupling, associating each
source point to numerous targets. When it comes to non-linear settings (third
and fifth lines), one can notice that the sliced estimated costs deviate from
their OT counterpart: as $\minS$ and $\ES$ rely on plans obtained by projecting
on a line then lifted to the original space, and because none of these
projections capture the true matching, the approximation is quite poor, with
spurious matchings between the two parts of the moon. Dedicated variants of
Sliced Wasserstein have been proposed in this non-linear setting, for instance
\emph{generalised} versions in which the data are projected onto a non linear
surface, e.g. \cite{kolouri2019generalized}, and \emph{augmented} ones
\cite{chen2020augmented} that first embed the data into a higher dimensional
space in which a linear surface better captures the distances. These variants
are out of the scope of this paper, but note that a non-linear variant of
$\minS$ has been proposed in~\cite{chapel2025differentiable}. 

\paragraph{Comparing plans obtained by flows}
To avoid relying on one single direction and to better take into account the non-linearities
 in the distributions, we propose here to build on \emph{flows}, for
which different directions can be chosen at each iteration. The second, fourth
and sixth lines of \cref{fig:illus_plans_2d} present trajectories obtained when
considering such flows, with an SGD optimiser and a fixed learning rate equal to
2 (as recommended by \cite[under Equation 44]{bonneel2015sliced}, we take a
learning rate equal to the dimension). If the flow has converged after 200 steps
(that is to say, when the Wasserstein distance between two consecutive steps is
less than $10^{-6}$), we infer a transport plan as the map linking the source
and the target samples reached by the flow. This strategy also allows considering
Sliced Wasserstein to obtain a plan, as proposed in \cite[Section
3.3]{rabin2012wasserstein}. One can notice that, as expected, 2-Wasserstein
flows plan recovers the transport plan and that Sliced Wasserstein based plan is
close to the actual one. As observed in \cref{sec:GF}, even in the simple case
of 2 Gaussians, Expected Sliced does not converge. When considering $\minS$,
flow-based transport allows enhancing the approximation of the plan, avoiding
spurious couplings between the two moons. One further notices that this strategy
comes with an extra computational cost as several iterations for computing the
flow are needed to obtain the approximation. We present this method to highlight
the benefits of stochastic algorithms when using sliced-based methods.

\begin{figure}[H]
    \centering
    \includegraphics[width=0.8\textwidth]{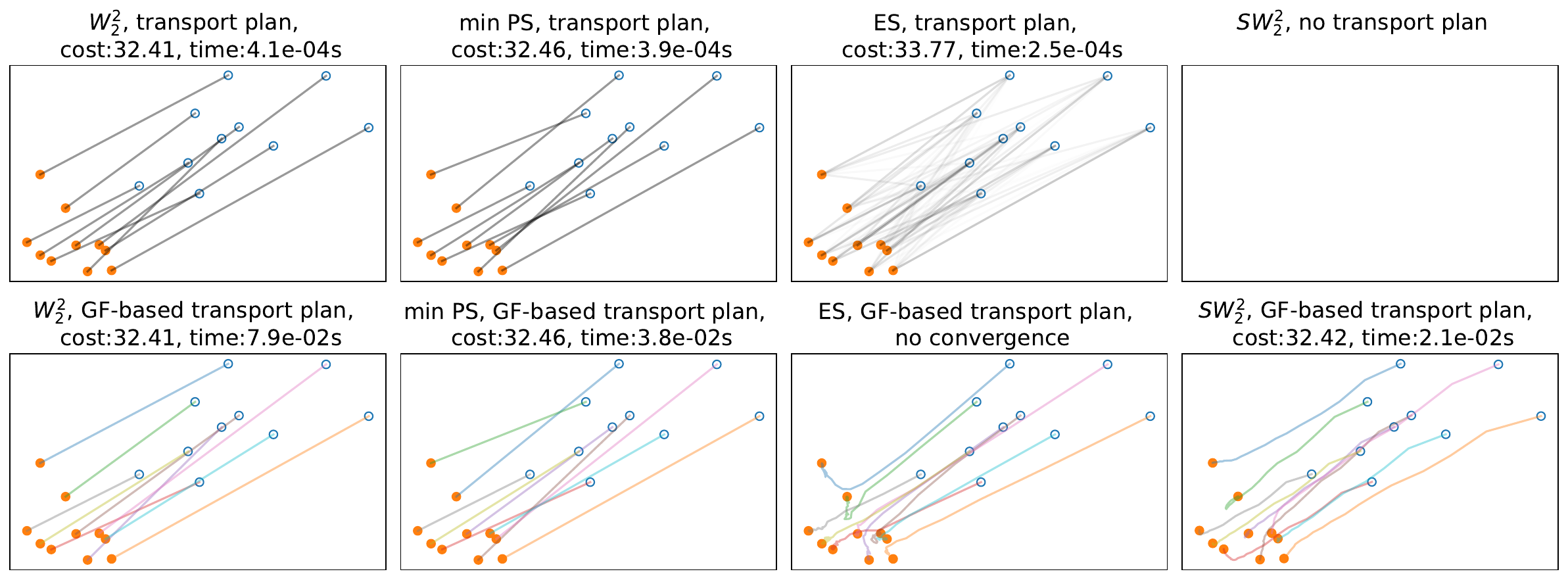}\\
    \includegraphics[width=0.8\textwidth]{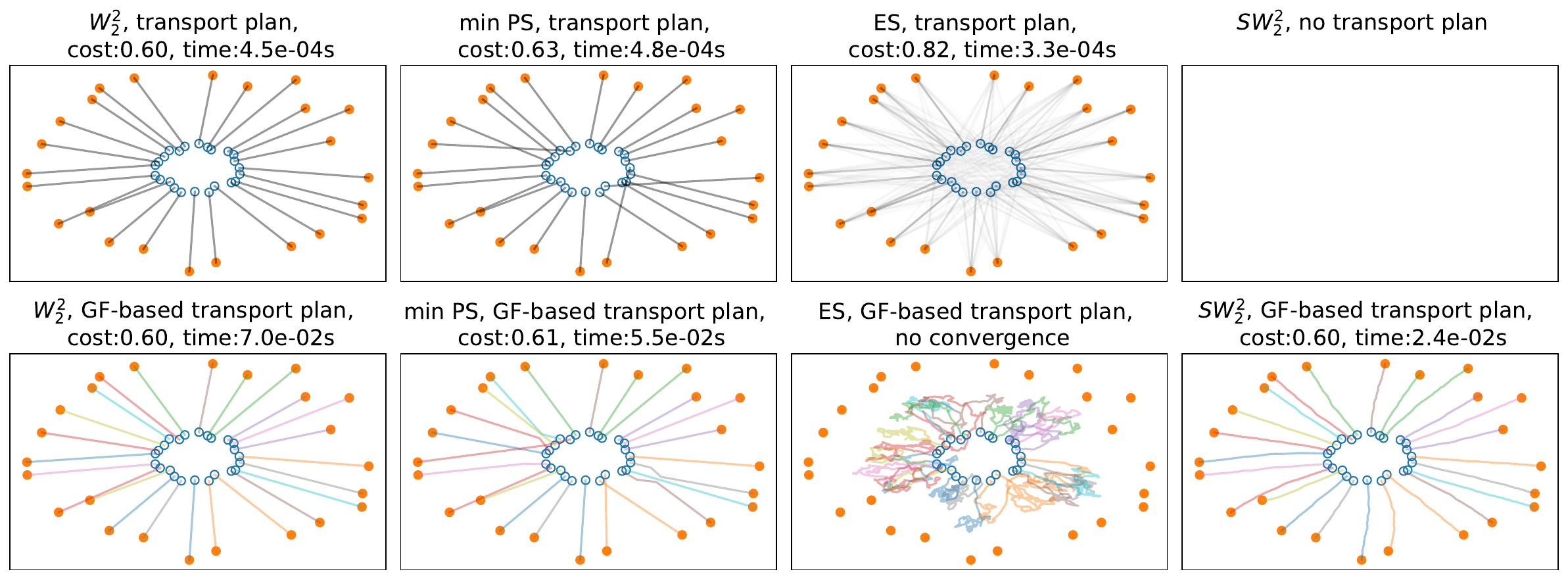}\\
    \includegraphics[width=0.8\textwidth]{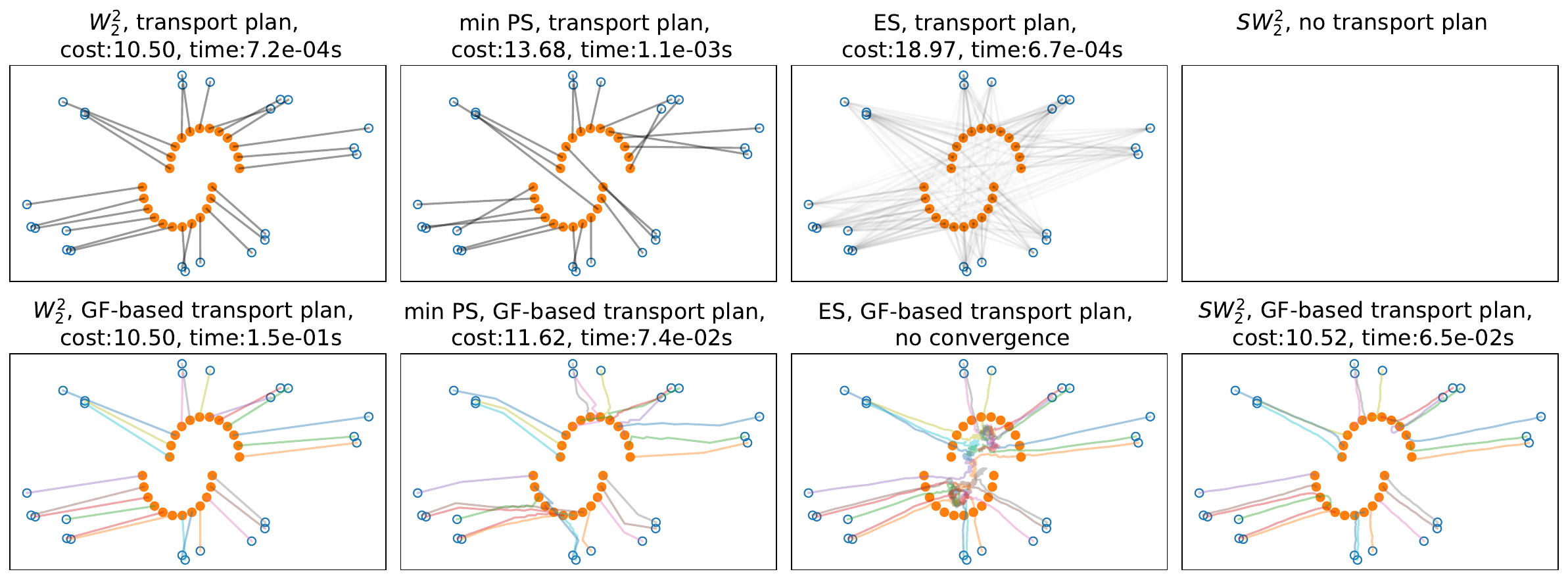}\\
    \caption{Comparison of the plans obtained by sliced plans methods and
    2-Wasserstein between a source (blue samples) and a target (orange samples)
    distributions. First, third and fifth lines: transport plans obtained by
    solving Wasserstein, min-Pivot Sliced and Expected Sliced. Second, fourth
    and sixth lines: trajectories obtained by solving a gradient flow for
    Wasserstein, min-Pivot Sliced, Expected Sliced and Sliced Wasserstein. In
    that case, the associated cost is computed by mapping the source sample to
    the target sample that is reached by the flow.}
    \label{fig:illus_plans_2d}
\end{figure}

\paragraph{Timings}
We report some timings for the different methods in order to assess their
computational efficiency. We consider the same settings as the first scenario of
the previous section (two Gaussians as a source and target distribution). For
the Sliced Wasserstein flow, we perform 10 steps, with an extra complexity
linear with the number of steps. We vary the number of samples from $n = 10$ to
$n = 10^7$, and present the results in \cref{fig:timings}. One can notice that
sliced-based methods are significantly faster to compute when $n$ grows. Note
that Wasserstein cannot be computed for $n\geq 10^5$ due to memory issues, as
it requires storing the full cost matrix $C\in \mathbb{R}^{n \times n}$ of size
$n^2$; there is no need to store $C$ for sliced-based methods, which require a
memory of $2n$ for $\minS$ and at most $2Ln$ for $\ES$. The time complexities
of all flow variants are proportional to the number of flow steps, and we notice
that all sliced methods have comparable complexities in $\O(Lnd + Ln\log(n))$,
which is substantially advantageous compared to the $\O(n^3 \log n)$ complexity
of standard OT.

\begin{figure}[H]
    \centering
    \includegraphics[width=0.5\textwidth]{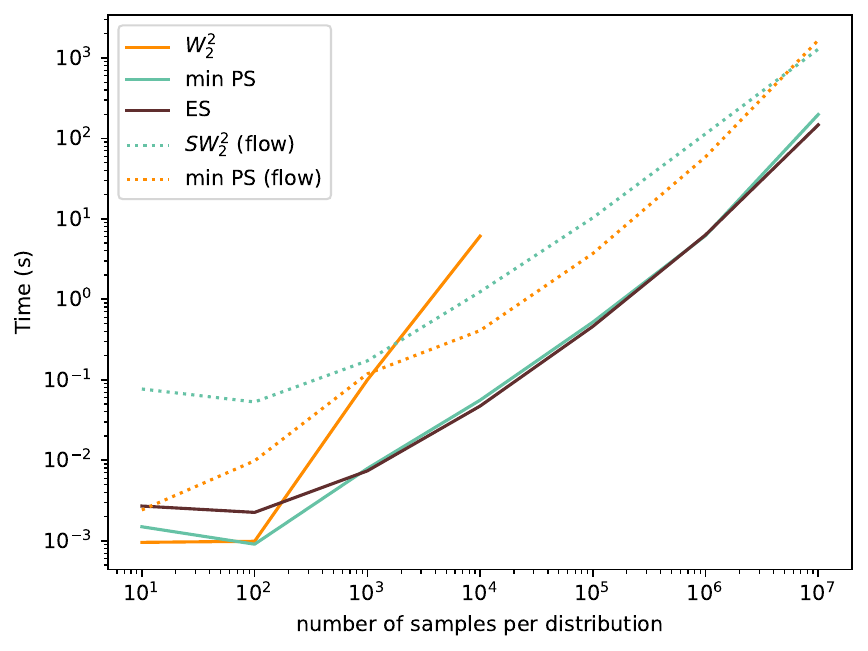}
    \caption{Running time comparisons of different methods for varying number 
    of samples $n$.}
    \label{fig:timings}
\end{figure}

\subsection{Illustration on Colour Transfer}

Colour transfer consists of transferring the colour distribution of a source
image onto a target image while preserving the structure of the source. We see
an RGB image $I \in \R^{w\times h\times 3}$ as the uniform measure of its pixels
in the RGB space $\mu_I := \tfrac{1}{wh}\sum_{i=1}^w\sum_{j=1}^h
\delta_{I_{i,j,\cdot}} \in \mathcal{P}(\R^3)$. Given a source image $I$ and a
target image $J$ of the same size, our objective is to match (in a certain sense)
each pixel $(i, j)$ of $I$ to a pixel $(i', j')$ of $J$. We consider three
different approaches: first, we compute a permutation that is (approximately)
optimal for the $\minS$ discrepancy, approximated by searching over $L=50$
directions. Using this permutation, we replace each pixel of $I$ with its
corresponding pixel in $J$. Second, we approximate the Expected Sliced plan by
averaging over $L=50$ directions. Since this does not yield a permutation but
only a transport plan $\oll{\gamma}$, we use the barycentric projection (i.e.
, conditional expectation) of $\oll{\gamma}$, which provides only an approximate
matching to $\mu_J$. Finally, we compare these methods with the Sliced
Wasserstein (SW) flow proposed in \cite{rabin2012wasserstein}, which operates 10
steps of Stochastic Gradient Descent with a learning rate of 1 on $X \longmapsto
\SW_2^2(\mu_X, \mu_J)$ initialised at $X_0 := I$ and samples a batch of 3
orthonormal directions at each step. Note that while the final iteration is
expected to verify $\mu_X \approx \mu_J$, it may not be the case in practice
depending on the hyperparameter choices. We report our results on three
different image pairs in
\cref{fig:ocean_sunset_to_ocean_day,fig:mountain_to_bridge,fig:A1_to_C1}. 

\begin{figure}[H]
    \centering
    \includegraphics[width=1\textwidth]{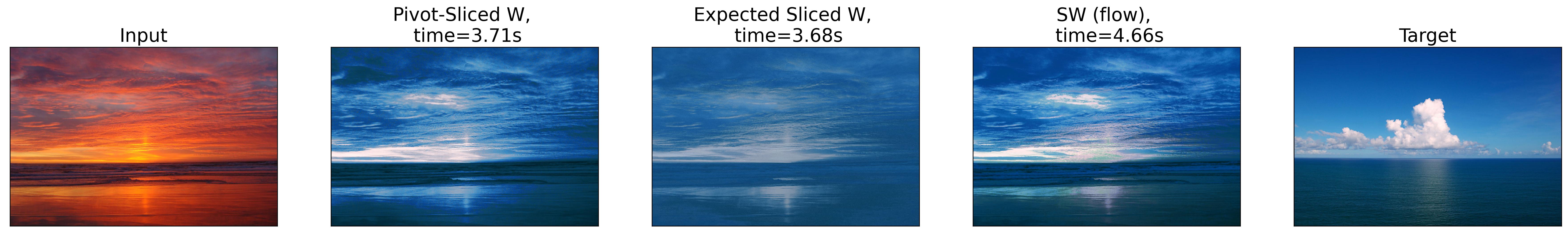}
    \caption{Colour transfer example on images of size $1000\times 669$.}
    \label{fig:ocean_sunset_to_ocean_day}
\end{figure}

\begin{figure}[H]
    \centering
    \includegraphics[width=1\textwidth]{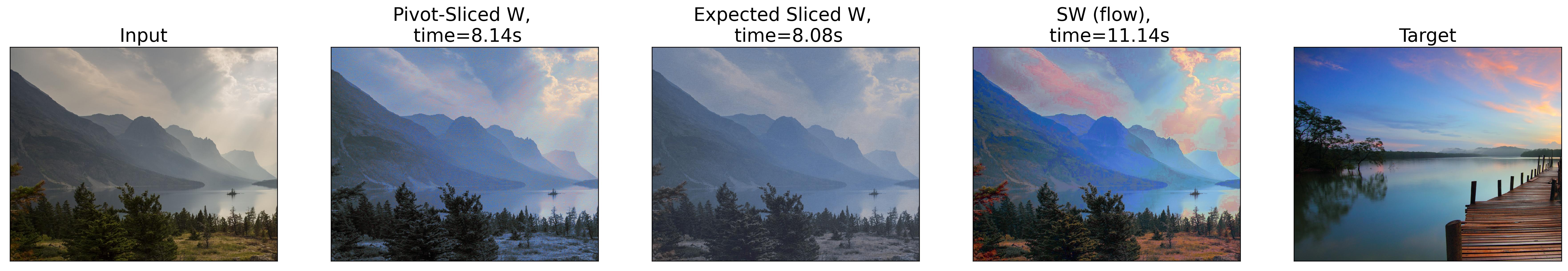}
    \caption{Colour transfer example on images of size $1280\times 1024$.}
    \label{fig:mountain_to_bridge}
\end{figure}

\begin{figure}[H]
    \centering
    \includegraphics[width=1\textwidth]{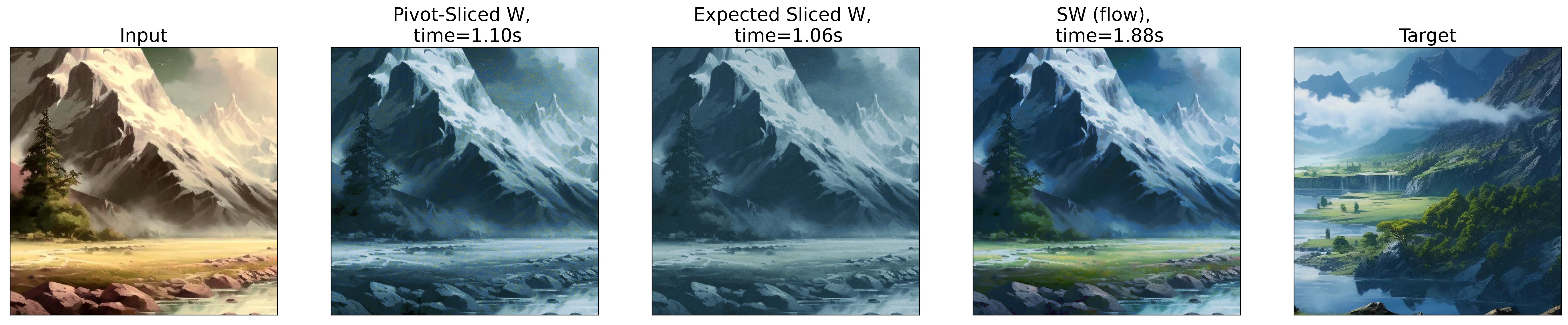}
    \caption{Colour transfer example on images of size $500\times 500$.}
    \label{fig:A1_to_C1}
\end{figure}

In \cref{fig:ocean_sunset_to_ocean_day}, the source and target images are
relatively monochrome, which makes the colour transfer task easier. We observe
that the Pivot-Sliced and SW methods are comparable, while the Expected Sliced
results in duller colours. Contrastingly, in \cref{fig:mountain_to_bridge}, the
colour palettes are more diverse, and Pivot-Sliced yields a visually worse result
than SW, while SW matches the colour distributions less faithfully, with some
artefacts in the sky. As for Expected Sliced, the results are again duller and
quite different from the target colour distribution. Finally, in
\cref{fig:A1_to_C1}, only the SW method produces visually consistent results;
the matchings provided by $\minS$ and $\ES$ fail to preserve sufficient spatial
structure, particularly in the green colours. Overall, while the plan
associated to $\minS$ can suffice in practice, it appears that iterative methods
such as the SW flow are better suited for this task. Our experiments suggest
that the barycentric projection of the Expected Sliced plan does not provide a
sound transportation. 

\FloatBarrier

\subsection{Experiments on a Shape Registration Task}
We now consider a shape registration task, with a rigid transformation that
involves a translation and a rotation. Most approaches to solve this problem are
concerned with finding the right correspondences between the points. For
instance, the Iterative Closest Point (ICP) algorithm~\cite{besl1992method}
relies on nearest neighbour correspondences, considering the Euclidean distance
between points. Optimal transport is now a workhorse for this task, as it
provides a principled way to find correspondences between two point clouds;
see~\cite{bonneel2023survey} for a review of OT-based methods for point cloud
registration. We here evaluate the performance of the sliced-based methods,
namely $\minS$ and $\ES$, in this context. We compare them to the 2-Wasserstein
distance, which is a standard benchmark for point cloud registration, and also
to Sliced Wasserstein, using a gradient flow as described in \cref{sec:GF} to
get an approximated transport plan. Note that Expected Sliced does not provide a
one-to-one correspondence, but they can be inferred from the blurred transport
plan~\cite{solomon2015convolutional}. 

We consider two point clouds of 3D shapes, which are subsampled from the
\emph{bunny} and \emph{armadillo} shapes of the \emph{open3d}
library~\cite{Zhou2018}. We first subsample both shapes with $n=500$ points, and
then we apply 10 different rigid transformations to the source shape to obtain the
target shape. We then run the ICP algorithm, with several alignment methods: the
nearest neighbour correspondence, Wasserstein, Sliced Wasserstein, min-Pivot
Sliced and Expected Sliced, to realign the two shapes. \cref{fig:results_ICP}
presents the two shapes, subsampled with $n=2000$ points for visualisation
purposes. The second line presents the Wasserstein distance between the
(registered) source and target point clouds along the iterations. One can notice
that  Min-Pivot Sliced yields the best registration among all methods: this
conclusion was also reached by~\cite{mahey23fast} who conjecture that it allows
exiting local minima of the ICP algorithm by finding an approximated matching.

We also consider the case where the shapes are not subsampled, which is a more
computationally challenging setup, especially for the armadillo shape. One can
draw similar conclusions, with $\minS$ yielding the best registration, with
little variation around the different repetitions of the experiment. 

\begin{figure}[H]
    \centering
        \begin{tabular}{lr}
    \includegraphics[width=0.25\textwidth]{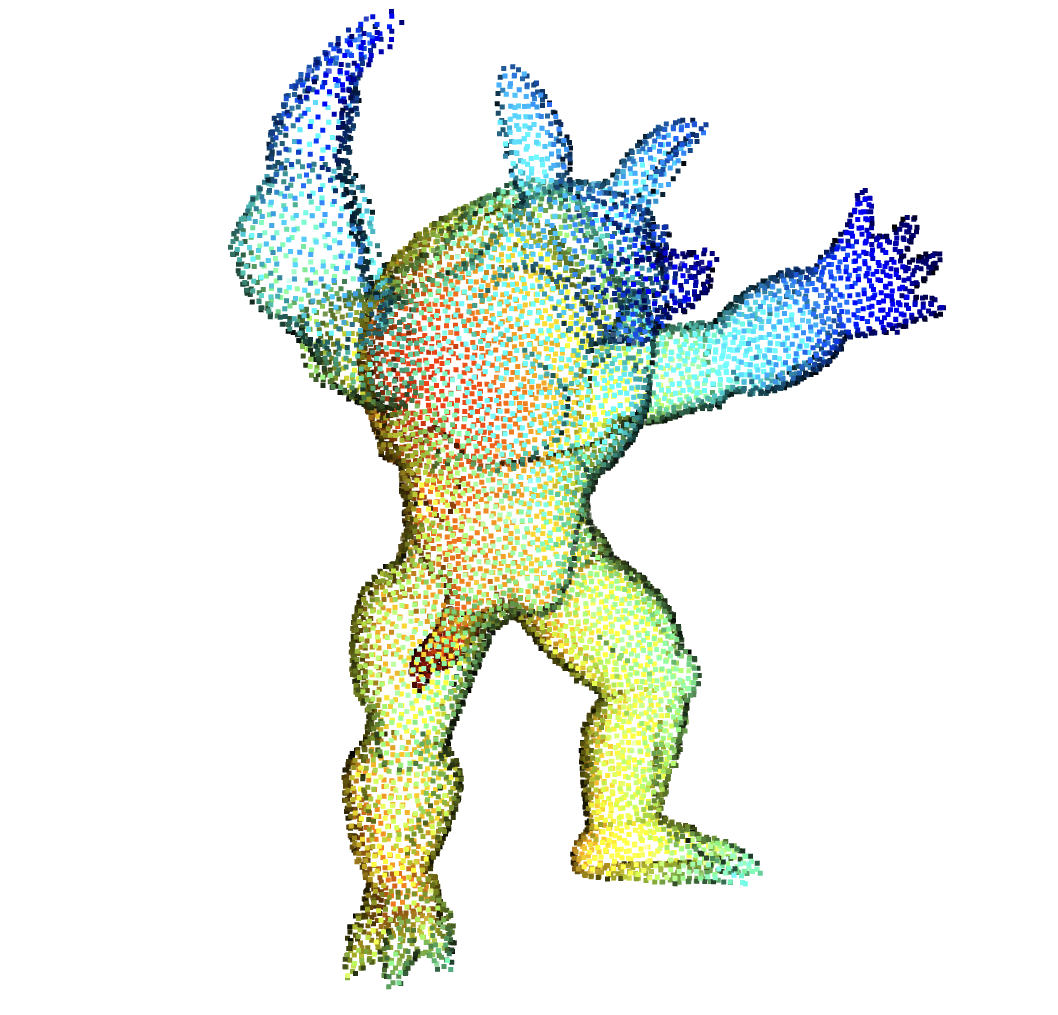} ~~~~~ &
    ~~~~~~~~~ \includegraphics[width=0.25\textwidth]{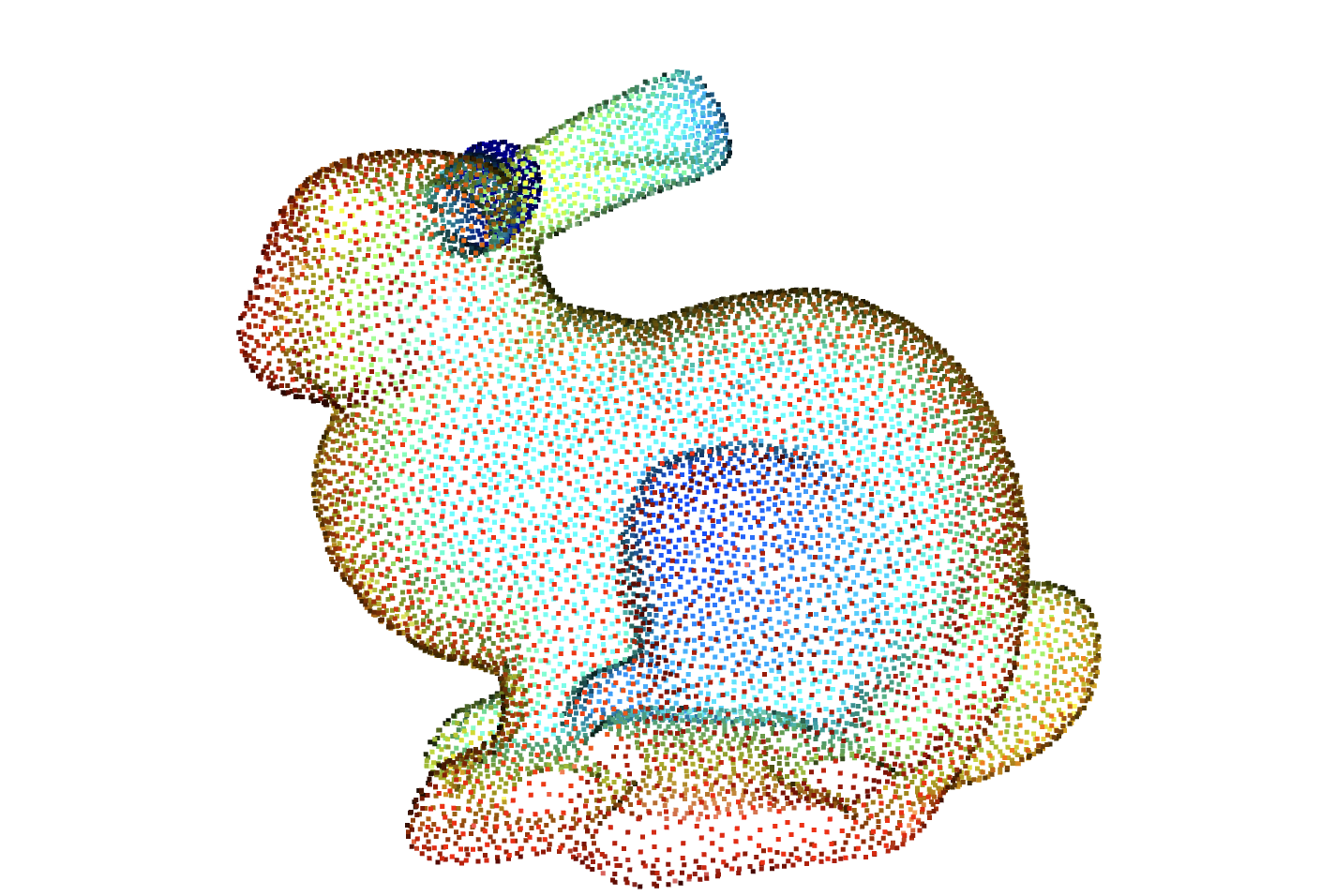}
    \end{tabular}\\
    \includegraphics[width=0.75\textwidth]{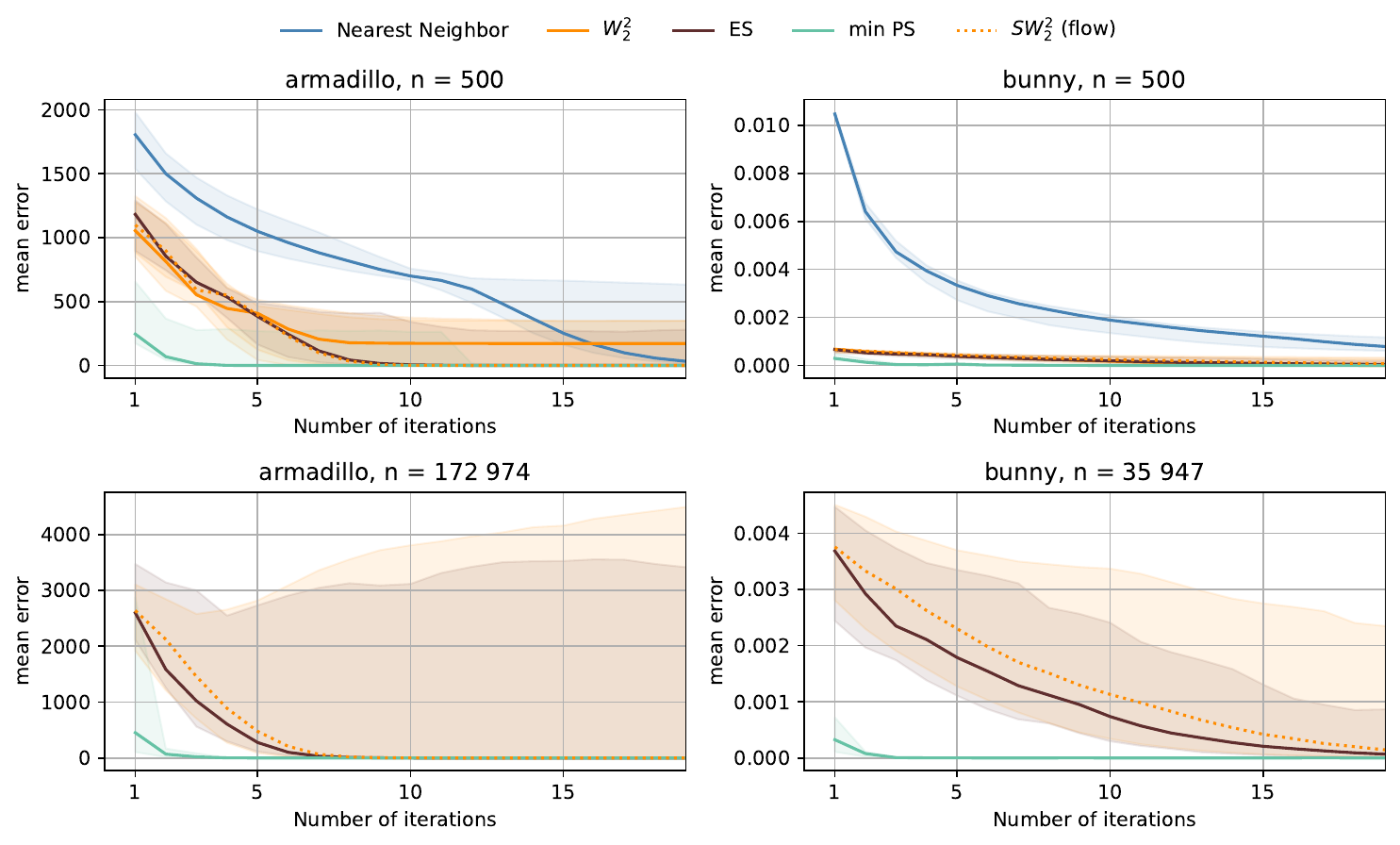} 

    \caption{Evolution of the loss along the iterations of the ICP algorithm.
    The loss is computed as the mean square distance between each target sample
    and the registered source. The first column corresponds to the results for
    the armadillo shape, while the second column corresponds to the bunny
    shape.}
    \label{fig:results_ICP}
\end{figure}

%% file: sections/ackn.tex
\subsection*{Acknowledgements}

We would like to thank Nathaël Gozlan and Agnès Desolneux for their insights on
technical aspects of \cref{sec:preliminary_nu_based_wass}. We thank two
anonymous reviewers for their comments and suggestions that helped us improve
this work, and specifically for suggesting \cref{rmk:centred_ES}.

This research was funded in part by the Agence nationale de la
recherche (ANR), Grant ANR-23-CE40-0017 and by the
France 2030 program, with the reference ANR-23-PEIA-0004.

%% file: sections/appendix.tex
\section{Appendix}

\subsection{Ambiguity in SWGG from
\texorpdfstring{\cite{mahey23fast}}{mahey23fast}}\label{sec:pathological_cases_swgg}

Let $\mu_1 = \tfrac{1}{n}\sum_i \delta_{x_i},\; \mu_2 = \tfrac{1}{n}\sum_i
\delta_{y_i},$ and $\theta \in \SS^{d-1}$. Consider $\sigma_\theta$ a
permutation which sorts $(\theta^\top x_i)_{i=1}^n$ and $\tau_\theta$ sorting
$(\theta^\top y_i)_{i=1}^n$. The Sliced Wasserstein Generalised Geodesic
distance \cite[Equation 8]{mahey23fast} is defined as
\begin{equation}\label{eqn:swgg_discrete2}
    \SWGG_2^2(\mu_1, \mu_2, \theta) := \cfrac{1}{n}\Sum{i=1}{n}\|x_{\sigma_\theta(i)} - y_{\tau_\theta(i)}\|_2^2.
\end{equation}
We illustrate the coupling induced by $\SWGG_2^2(\mu_1, \mu_2, \theta)$ in
\cref{fig:ex_swgg}:
\begin{figure}[H]
    \centering
    \includegraphics[width=0.5\textwidth]{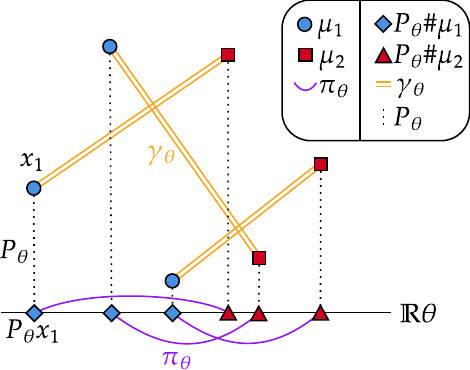}
    \caption{Coupling $\gamma_\theta \in \Pi(\mu_1, \mu_2)$ induced by
    $\SWGG_2^2(\mu_1, \mu_2, \theta)$ for $d=2$, $n=3$, $\theta = (1, 0)$. The
    support of the measure $\mu_1$ is represented by blue circles, and the
    support of $\mu_2$ with red squares. The projected measures
    $P_\theta\#\mu_1$ and $P_\theta\#\mu_2$ are represented by the blue diamonds
    and triangles respectively. The optimal coupling between $P_\theta\#\mu_1$
    and $P_\theta\#\mu_2$ is drawn with purple curves, and the associated
    coupling $\gamma_\theta$ between $\mu_1$ and $\mu_2$ is represented by the
    orange double lines. In this example, the projections of the points of the
    support of $\mu_1$ are distinct (as for $\mu_2$), thus the coupling
    $\pi_\theta$ determines uniquely the coupling $\gamma_\theta$, there is no
    ambiguity.}
    \label{fig:ex_swgg}
\end{figure}
Unfortunately, the RHS quantity in \cref{eqn:swgg_discrete2} depends on the
choice of the permutations, rendering the quantity ill-defined, as showcased in
\cref{ex:swgg_ambiguity}. 

\begin{example}[Ambiguity in $\SWGG$]\label{ex:swgg_ambiguity} Consider $d=2$,
    $n=2$, the points $x_1=(0, 1 ),\; x_2=(0, 0),\; y_1=(0, 0),\; y_2=(0, 1)$,
    the line $\theta = (1, 0)$ and the measures $\mu_1 =
    \tfrac{1}{2}(\delta_{x_1} + \delta_{x_2})$, $\mu_2 =
    \tfrac{1}{2}(\delta_{y_1} + \delta_{y_2})$. We have $\mu_1=\mu_2$, and
    $\theta^\top u = 0$ for all points $u\in \{x_1, x_2, y_1, y_2\}$, hence any
    choice of permutations $(\sigma_\theta, \tau_\theta)$ sorts the respective
    points $(\theta^\top x_i)$ and $(\theta^\top y_i)$. Choosing
    $(\sigma_\theta, \tau_\theta) = (I, I)$, we obtain 
    $$\SWGG_2^2(\mu_1, \mu_2, \theta) =
    \tfrac{1}{2}(\|x_1-y_1\|_2^2+\|x_2-y_2\|_2^2) = 1,$$ which in particular in
    non-zero, which shows that $\SWGG_2(\cdot, \cdot, \theta)$ is not a
    distance. Another possible choice $(\sigma_\theta, \tau_\theta) = ((2, 1),
    (2, 1))$ yields a value of 0.
\end{example}

One could consider the following ``fix'' to the permutation choice issue:
\begin{equation}\label{eqn:swgg_fix}
    \SWGGfix_2^2(\mu_1, \mu_2, \theta) := 
    \underset{(\sigma_\theta, \tau_\theta) \in \mathfrak{S}_\theta(X, Y)}{\min}
    \ \cfrac{1}{n}\Sum{i=1}{n}\|x_{\sigma_\theta(i)} - y_{\tau_\theta(i)}\|_2^2,
\end{equation}
where $\mathfrak{S}_\theta(X, Y)$ is the set of pairs of permutations
$(\sigma_\theta, \tau_\theta)$ that sort $(\theta^\top x_i)_{i=1}^n$ and
$(\theta^\top y_i)_{i=1}^n$ respectively. We illustrate this idea in
\cref{fig:ex_swgg_ambiguity}.
\begin{figure}[H]
    \centering
    \includegraphics[width=0.7\textwidth]{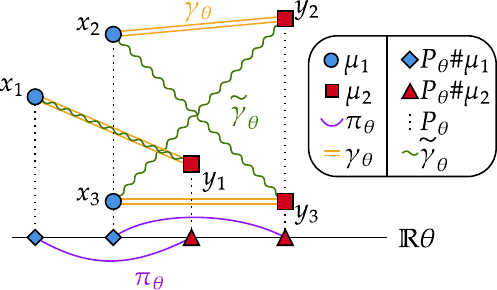}
    \caption{In this example, the projections sometimes coincide, and the
    optimal coupling $\pi_\theta$ between $P_\theta\#\mu_1$ and
    $P_\theta\#\mu_2$ does not determine the coupling between $(x_2, x_3)$ and
    $(y_2, y_3)$. In terms of permutations, there are two possibilities:
    $\gamma_\theta := \tfrac{1}{3}(\delta_{x_1\otimes y_1} + \delta_{x_2\otimes
    y_2} + \delta_{x_3\otimes y_3})$ displayed with orange double lines, and
    $\widetilde{\gamma}_\theta := \tfrac{1}{3}(\delta_{x_1\otimes y_1} +
    \delta_{x_2\otimes y_3} + \delta_{x_3\otimes y_2})$ represented by green
    squiggly lines. Here, the cost of $\gamma_\theta$ is lower, so we would
    choose it.}
    \label{fig:ex_swgg_ambiguity}
\end{figure}

\subsection{Midpoints are Geodesic
Middles}\label{sec:midpoints_are_geodesic_middles}

In the following, we remind a well-known simple result about geodesic spaces,
which we apply to show that Wasserstein means are middles of Wasserstein
geodesics (see \cref{prop:wass_mean_geodesic}). We consider a geodesic space
$(\X, d)$, which is to say that $d$ is a distance on $\X$ such that for any
$(x_1, x_2)\in \X^2$ there exists a curve $\gamma: [0, 1] \longrightarrow \X$
with $\gamma(0)=x_1$ and $\gamma(1)=x_2$ such that $d(\gamma(t), \gamma(s)) =
|t-s|d(x_1, d_2)$. Such a curve is called a geodesic between $x_1$ and $x_2$.
\begin{lemma}\label{lemma:midpoints} Let $(\X, d)$ be a geodesic space, let
    $x_1, x_2 \in \X$ and consider the set $\M(x_1, x_2)$ of Midpoints:
    \begin{equation}\label{eqn:midpoints}
        \M(x_1, x_2) = \underset{y\in \X}{\argmin}\ d(x_1, y)^2 + d(y, x_2)^2.
    \end{equation}
    This set is in fact exactly the set of middles of geodesics:
    \begin{equation}\label{eqn:midpoints_are_middles_of_geodesics}
        \M(x_1, x_2) = \left\lbrace \gamma(\tfrac{1}{2})\ |\ \gamma\ 
        \text{is\ a \ geodesic\ between\ } x_1\ \text{and}\ x_2\right\rbrace.
    \end{equation}
\end{lemma}
\begin{proof}
    Denote by $\M'(x_1, x_2)$ the RHS of
    \cref{eqn:midpoints_are_middles_of_geodesics}, first we show $\M'(x_1, x_2)
    \subset \M(x_1, x_2)$ and compute the optimal value of \cref{eqn:midpoints}.
    Let $\gamma$ a constant-speed geodesic between $x_1$ and $x_2$, we have
    $$d(x_1, \gamma(\tfrac{1}{2}))^2 + d(\gamma(\tfrac{1}{2}), x_2)^2 =
    d(\gamma(0), \gamma(\tfrac{1}{2}))^2 + d(\gamma(\tfrac{1}{2}), \gamma(1))^2
    = d(x_1, x_2)^2/2.$$ Now take any $y\in \X$, we have (by convexity of $t
    \longmapsto t^2$, then by the triangle inequality for $d$)
    \begin{align}
        d(x_1, y)^2 + d(y, x_2)^2 &= 2(d(x_1, y)^2/2 + d(y, x_2)^2/2) \\
        &\geq 2(d(x_1, y)/2 + d(y, x_2)/2)^2 
        \label{eqn:inequality_midpoints_cvx}\\ 
        &\geq d(x_1, x_2)^2/2. \label{eqn:inequality_midpoints_triangle}
    \end{align}
    This shows that any such $\gamma(\tfrac{1}{2})$ is solution of the
    optimisation problem which defines $\M(x_1, x_2)$, and thus $\M'(x_1, x_2)
    \subset \M(x_1, x_2)$. The value of the minimisation problem from
    \cref{eqn:midpoints} is $d(x_1, x_2)^2/2$.

    Let $y^* \in \M(x_1, x_2)$, we now show that $d(x_1, y^*)=d(y^*, x_2)=d(x_1,
    x_2)/2$. Since $y^*$ is optimal and that the optimal value is $d(x_1,
    x_2)^2$, the inequalities \cref{eqn:inequality_midpoints_cvx} and
    \cref{eqn:inequality_midpoints_triangle} are equalities for $y:=y^*$. First,
    \cref{eqn:inequality_midpoints_cvx} yields $d(x_1, y^*)=d(y^*, x_2)$, then
    \cref{eqn:inequality_midpoints_triangle} yields $d(x_1, y^*) =d(x_1,
    x_2)/2$. 

    We now show that $\M(x_1, x_2) \subset \M'(x_1, x_2)$: let $y^* \in \M(x_1,
    x_2)$, consider $\gamma_1$ a geodesic from $x_1$ to $y^*$, and $\gamma_2$ a
    geodesic from $y^*$ to $x_2$. We introduce the curve
    $$\gamma : \app{[0, 1]}{\X}{t}{\left\lbrace\begin{array}{cc} \gamma_1(2t) &
        \text{if}\ t\in [0, \tfrac{1}{2}]; \\
        \gamma_2(2t-1) & \text{if}\ t\in [\tfrac{1}{2}, 1]. \\
    \end{array} \right.} $$ Our objective is to show that $\gamma$ is a geodesic
    from $x_1$ to $x_2$ (since $\gamma(\tfrac{1}{2})=y^*$, this will show that
    $y^*\in \M'(x_1, x_2)$). By construction $\gamma(0)=x_1$, $\gamma(1)=x_2$.
    Let $(t, s)\in [0, 1]^2$ with $t \leq s$, we want to prove $d(\gamma(t),
    \gamma(s))=|s-t|d(x_1, x_2)$.

    Firstly, we consider the case $(t,s)\in [0, \tfrac{1}{2}]^2$. In this case,
    $$d(\gamma(t), \gamma(s)) = d(\gamma_1(2t), \gamma_1(2s))=(2s-2t)d(x_1, y^*)
    = (s-t)d(x_1, x_2),$$ where we used $d(x_1, y^*) = d(x_1, x_2)/2$, which we
    proved earlier for any optimal $y^*$. The case $(t,s)\in [\tfrac{1}{2},
    1]^2$ can be treated similarly.

    Secondly, we assume $t\in [0, \tfrac{1}{2}]$ and $s \in [\tfrac{1}{2}, 1]$.
    We first prove $d(\gamma(t), \gamma(s)) \leq (s-t)d(x_1, x_2)$ using the
    triangle inequality and $d(x_i, y^*) = d(x_1, x_2)/2$ for $i\in\{1,2\}$:
    \begin{align*}
        d(\gamma(t), \gamma(s)) &\leq d(\gamma(t), y^*) + d(y^*, \gamma(s)) \\
        &= d(\gamma_1(2t), \gamma_1(1)) + d(\gamma_2(0), \gamma_2(2s-1)) \\
        &= (1-2t)d(x_1, y^*) + (2s-1)d(y^*, x_2) \\
        &= (s-t)d(x_1, x_2).
    \end{align*}
    For the converse inequality $d(\gamma(t), \gamma(s)) \geq (s-t)d(x_1, x_2)$,
    we apply the triangle inequality:
    $$d(x_1, x_2) \leq d(x_1, \gamma(t)) + d(\gamma(t), \gamma(s)) +
    d(\gamma(s), x_2), $$ which yields:
    \begin{align*}
        d(\gamma(t), \gamma(s)) &\geq 
        d(x_1, x_2) - d(\gamma_1(0), \gamma_1(2t)) - d(\gamma_2(2s-1), x_2) \\
        &= (1-t-(1-s))d(x_1, x_2) = (s-t)d(x_1, x_2).
    \end{align*}
    The case $s\in [0, \tfrac{1}{2}]$ and $t \in [\tfrac{1}{2}, 1]$ is done
    symmetrically and thus $d(\gamma(t), \gamma(s)) = |s-t|d(x_1,x_2)$, which
    shows that $y^*\in \M'(x_1, x_2)$. We conclude that $\M'(x_1, x_2) = \M(x_1,
    x_2)$.
\end{proof}

\subsection{Reminders on Disintegration of Measures}\label{sec:disintegration}

In \cref{def:disintegration_P}, we recall the definition of disintegration of
measures with respect to a map (taken from \cite[Theorem
5.3.1]{ambrosio2005gradient}). By slight abuse of notation, we will write
$P^{-1}(y) := P^{-1}(\{y\})$ for a map $P: \X\longrightarrow \Y$ that need not
be injective and $y\in \Y$.

\begin{definition}\label{def:disintegration_P} Consider a Borel map $P:
    \X\longrightarrow \Y$ between Polish spaces $\X, \Y$ and $\mu \in
    \mathcal{P}(\X)$. There exists a $P\#\mu$-almost-everywhere unique Borel
    family $(\mu^y)_{y\in \Y}\subset \mathcal{P}(\X)$ of measures verifying
    $\mu^y(\X\setminus P^{-1}(y)) = 0$, and verifying the following identity
    against test functions $\phi \in \mathcal{C}_b^0(\X)$:
    \begin{equation}\label{eqn:disintegration_P_test_function}
        \int_{\X}\phi(x)\dd\mu(x) = 
        \int_\Y\left(\int_{P^{-1}(y)}\phi(x)\dd\mu^y(x)\right)\dd(P\#\mu)(y).
    \end{equation}
    We will write \cref{eqn:disintegration_P_test_function} symbolically as:
    \begin{equation}\label{eqn:disintegration_P_symbolic}
        \mu(\dd x)= (P\#\mu)\left(P(\dd x)\right)\ \mu^{P(x)}(\dd x).
    \end{equation}
\end{definition}

For example, in the case $\X = \R^d\times\R^d$ and $P(y, x) = y$, the
disintegration corresponds to the disintegration with respect to the first
marginal $\nu$ of a coupling $\gamma \in \Pi(\nu, \mu)$. In this case, each
measure $\gamma^y$ is a measure of $\mathcal{P}(\R^{2d})$ concentrated on the
slice $\{y\} \times \R^d$, which is routinely identified as a measure on $\R^d$
in literature. This disintegration is written symbolically as $\gamma(\dd y, \dd
x) = \nu(\dd y)\gamma^y(\dd x)$.

\subsection{Proof of the Disintegration Formula for \texorpdfstring{$\nu$}{nu}-based Wasserstein}\label{sec:proof_W_nu_disintegration}

In this section, we provide a proof to \cref{thm:W_nu_disintegration}, and use
the notation from the statement. Let $\rho \in \Gamma(\nu, \mu_1, \mu_2)$ (see
\cref{eqn:def_Gamma3}), we have
\begin{align}\label{eqn:W_nu_disintegration_proof_rho_lb}
    \int_{\R^{3d}}\|x_1-x_2\|_2^2\dd\rho(y, x_1, x_2) &= 
    \int_{\R^d}\left(\int_{\R^{2d}}\|x_1-x_2\|_2^2\dd\rho^y(x_1, x_2)\right)
    \dd\nu(y)\nonumber \\
    &\geq \int_{\R^d}\W_2^2(P_1\#\rho^y, P_2\#\rho^y)\dd\nu(y),
\end{align}
where we wrote the disintegration $\rho(\dd y, \dd x_1, \dd x_2) = \nu(\dd
y)\rho^y(\dd x_1, \dd x_2)$. Note that by \cite[Lemma
12.4.7]{ambrosio2005gradient}, the map $y \longmapsto \W_2^2(P_1\#\rho^y,
P_2\#\rho^y)$ is Borel.

Now since $\rho \in \Gamma(\nu, \mu_1, \mu_2)$, we can write $P_{1, 2}\#\rho
=: \gamma_1 \in \Pi^*(\nu, \mu_1)$ and $P_{1, 3}\#\rho =: \gamma_2 \in
\Pi^*(\nu, \mu_2)$. It follows that for $\nu$-almost every $y\in \R^d$, we
have for $i\in \{1, 2\}$ that $P_i\#\rho^y = \gamma_i^y$, where we
disintegrated $\gamma_i(\dd y, \dd x) = \nu(\dd y) \gamma_i^y(\dd x)$ (for
example by \cite[Lemma 5.3.2]{ambrosio2005gradient}). Taking the infimum on
$\rho$ on both sides yields
\begin{equation}\label{eqn:W_nu_disintegration_proof_lb}
    \W_\nu^2(\mu_1, \mu_2) \geq \underset{\gamma_i \in \Pi^*(\nu, \mu_i),
    \: i\in \{1,2\}}{\inf}\int_{\R^d}\W_2^2(\gamma_1^y,
    \gamma_2^y)\dd\nu(y).
\end{equation}
Fixing $\gamma_i \in \Pi^*(\nu, \mu_i)$ for $i\in \{1, 2\}$, we now construct a
3-plan $\rho \in \Gamma(\nu, \mu_1, \mu_2)$ which attains the lower bound in
\cref{eqn:W_nu_disintegration_proof_rho_lb}. Consider the disintegrations
$\gamma_i(\dd y, \dd x) = \nu(\dd y) \gamma_i^y(\dd x)$ for $i\in \{1, 2\}$. The
two families $(\gamma_i^y)_{y\in \R^d}$ are Borel in $\mathcal{P}_2(\R^d)$,
hence by \cite[Lemma 12.4.7]{ambrosio2005gradient}, there exists a Borel family
$(\rho^y)_{y\in \R^d}$ in $\mathcal{P}_2(\R^{2d})$ such that for all $y\in
\R^d$, $\rho^y \in \Pi^*(\gamma_1^y, \gamma_2^y)$. Setting $\rho(\dd y, \dd x_1,
\dd x_2) := \nu(\dd y)\rho^y(\dd x_1, \dd x_2)$ yields the desired 3-plan, since
for $\nu$-almost every $y\in \R^d$, $\rho^y$ is an optimal transport plan
between $\gamma_1^y$ and $\gamma_2^y$. We have shown that
\begin{equation}\label{eqn:W_nu_disintegration_proof_ub}
    \forall \gamma_i \in \Pi^*(\nu, \mu_i),\; i\in \{1, 2\},\; 
    \W_\nu^2(\mu_1, \mu_2) \leq 
    \int_{\R^d}\W_2^2(\gamma_1^y, \gamma_2^y)\dd\nu(y),
\end{equation}
which shows the equality in \cref{eqn:W_nu_disintegration}.

We finish by showing that the infimum in \cref{eqn:W_nu_disintegration} is
indeed attained. Note that having the weak convergence of plans
$(\gamma_n)\in \Pi(\nu, \mu_1)$ does not yield the $\nu$-almost-everywhere
convergence of the disintegrations $\gamma_n^y$ in general. Thankfully, we
can leverage the existence of a solution of the original formulation from
\cref{eqn:nu_based_Wass} by \cref{prop:W_nu_inf_attained}. Using the fact
that the two problems have the same value, we can take a solution $\rho$ of
\cref{eqn:nu_based_Wass} and construct a solution of
\cref{eqn:W_nu_disintegration} by disintegration.

%% file: impbot/implication_graph.tex
\definecolor{impBot_default}{HTML}{808080}
\definecolor{impBot_arrow}{HTML}{DC5800}
\definecolor{impBot_theorem}{HTML}{2196F3}
\definecolor{impBot_definition}{HTML}{FFEB3B}
\definecolor{impBot_assumption}{HTML}{FF3B3B}
\definecolor{impBot_proposition}{HTML}{2196F3}
\definecolor{impBot_prop}{HTML}{2196F3}
\definecolor{impBot_corollary}{HTML}{2196F3}
\definecolor{impBot_lemma}{HTML}{2196F3}

\tikzset{
  default/.style={
    rectangle,
    rounded corners,
    draw=impBot_default,
    fill=impBot_default!10!white,
    align=center,
    inner sep=2pt,
    outer sep=0pt,
    node font=\fontsize{8pt}{8pt}\selectfont,
  },
  arrow/.style={
    thick,
    draw=impBot_arrow,
    shorten >=2pt,
    shorten <=2pt,
    line width=1pt
  },
}

\tikzset{
  theorem/.style={
    default,
    draw=impBot_theorem,
    double=impBot_theorem!50!white,
    fill=impBot_theorem!10!white,
  },
  definition/.style={
    default,
    rounded corners=0pt,
    draw=impBot_definition,
    fill=impBot_definition!10!white,
  },
  assumption/.style={
    default,
    rounded corners=0pt,
    draw=impBot_assumption,
    fill=impBot_assumption!10!white,
  },
  proposition/.style={
    default,
    draw=impBot_proposition,
    fill=impBot_proposition!10!white,
  },
  prop/.style={
    default,
    draw=impBot_prop,
    fill=impBot_prop!10!white,
  },
  corollary/.style={
    default,
    draw=impBot_corollary,
    fill=impBot_corollary!10!white,
  },
  lemma/.style={
    default,
    densely dashed,
    draw=impBot_lemma,
    fill=impBot_lemma!10!white,
  },
}
\begin{center}\begin{tikzpicture}[xscale=2.5, yscale=0.6]
  \node[lemma] (lemmaCOLONtightness_Gamma3) at (0.0,0.0) {\Cref{lemma:tightness_Gamma3}};

  \node[theorem] (thmCOLONW_nu_disintegration) at (0.0,-4.0) {\Cref{thm:W_nu_disintegration}};

  \node[definition] (defCOLONS_theta) at (0.0,-5.0) {\Cref{def:S_theta}};

  \node[prop] (propCOLONsemi_1D_ot) at (0.0,-7.0) {\Cref{prop:semi_1D_ot}};

  \node[lemma] (lemmaCOLON1d_3plan_bimarginals) at (0.0,-14.0) {\Cref{lemma:1d_3plan_bimarginals}};

  \node[lemma] (lemmaCOLONCV_theta_monge_invariant_permutation) at (0.0,-17.0) {\Cref{lemma:CV_theta_monge_invariant_permutation}};

  \node[assumption] (assCOLONidentity_sorts_projections) at (0.0,-18.0) {\Cref{ass:identity_sorts_projections}};

  \node[definition] (defCOLONextreme_point) at (0.0,-21.0) {\Cref{def:extreme_point}};

  \node[definition] (defCOLONbipartite_P) at (0.0,-23.0) {\Cref{def:bipartite_P}};

  \node[lemma] (lemmaCOLONextract_cycle_P) at (0.0,-24.0) {\Cref{lemma:extract_cycle_P}};

  \node[lemma] (lemmaCOLONsplit_multicycle) at (0.0,-25.0) {\Cref{lemma:split_multicycle}};

  \node[lemma] (lemmaCOLONcycle_ollP) at (0.0,-26.0) {\Cref{lemma:cycle_ollP}};

  \node[definition] (defCOLONgeneral_position) at (0.0,-30.0) {\Cref{def:general_position}};

  \node[prop] (propCOLONelementary_props_lifted_plan) at (0.0,-32.0) {\Cref{prop:elementary_props_lifted_plan}};

  \node[definition] (defCOLONindependent_lifted_cost) at (0.0,-33.0) {\Cref{def:independent_lifted_cost}};

  \node[prop] (propCOLONES_theta_discrete_as_distance) at (0.0,-35.0) {\Cref{prop:ES_theta_discrete_as_distance}};

  \node[definition] (defCOLONexpected_sliced) at (0.0,-36.0) {\Cref{def:expected_sliced}};

  \node[lemma] (lemmaCOLONmidpoints) at (0.0,-39.0) {\Cref{lemma:midpoints}};

  \node[definition] (defCOLONdisintegration_P) at (0.0,-40.0) {\Cref{def:disintegration_P}};

  \node[prop] (propCOLONW_nu_inf_attained) at (1.0,-1.0) {\Cref{prop:W_nu_inf_attained}};

  \node[prop] (propCOLONwass_mean_geodesic) at (1.0,-3.0) {\Cref{prop:wass_mean_geodesic}};

  \node[corollary] (corCOLONS_theta_SWGG_as) at (1.0,-16.0) {\Cref{cor:S_theta_SWGG_as}};

  \node[lemma] (lemmaCOLONcharacterisation_Sigma_theta) at (1.0,-19.0) {\Cref{lemma:characterisation_Sigma_theta}};

  \node[prop] (propCOLONlifted_cost_almost_distance) at (1.0,-34.0) {\Cref{prop:lifted_cost_almost_distance}};

  \node[prop] (propCOLONW_nu_lsc) at (2.0,-2.0) {\Cref{prop:W_nu_lsc}};

  \node[lemma] (lemmaCOLONS_theta_well_def) at (2.0,-6.0) {\Cref{lemma:S_theta_well_def}};

  \node[prop] (propCOLONS_theta_semi_metric) at (2.0,-8.0) {\Cref{prop:S_theta_semi_metric}};

  \node[lemma] (lemmaCOLONplan_middle) at (2.0,-10.0) {\Cref{lemma:plan_middle}};

  \node[prop] (propCOLONCWtheta_discrete_v1) at (2.0,-20.0) {\Cref{prop:CWtheta_discrete_v1}};

  \node[lemma] (lemmaCOLONcondition_Psigmatau_in_Ptheta) at (2.0,-22.0) {\Cref{lemma:condition_Psigmatau_in_Ptheta}};

  \node[corollary] (corCOLONexpected_sliced_almost_distance) at (2.0,-37.0) {\Cref{cor:expected_sliced_almost_distance}};

  \node[prop] (propCOLONPStheta_lsc) at (3.0,-9.0) {\Cref{prop:PStheta_lsc}};

  \node[lemma] (lemmaCOLONplan_projected_middle) at (3.0,-11.0) {\Cref{lemma:plan_projected_middle}};

  \node[theorem] (thmCOLONextr_U_cap_PthetaXY) at (3.0,-27.0) {\Cref{thm:extr_U_cap_PthetaXY}};

  \node[corollary] (corCOLONexpected_sliced_distance_countably_discrete) at (3.0,-38.0) {\Cref{cor:expected_sliced_distance_countably_discrete}};

  \node[prop] (propCOLONS_theta_smaller_lift) at (4.0,-12.0) {\Cref{prop:S_theta_smaller_lift}};

  \node[theorem] (thmCOLONCW_theta_point_clouds) at (4.0,-28.0) {\Cref{thm:CW_theta_point_clouds}};

  \node[prop] (propCOLONminS_attained) at (4.0,-29.0) {\Cref{prop:minS_attained}};

  \node[theorem] (thmCOLONS_theta_equals_W_theta) at (5.0,-13.0) {\Cref{thm:S_theta_equals_W_theta}};

  \node[prop] (propCOLONS_theta_metric_atomless_projections) at (6.0,-15.0) {\Cref{prop:S_theta_metric_atomless_projections}};

  \node[prop] (propCOLONcondition_minS_equals_W2) at (6.0,-31.0) {\Cref{prop:condition_minS_equals_W2}};

  \begin{pgfonlayer}{background}
  \begin{scope}[transparency group, opacity=0.5773502691896258]
    \draw[arrow, lemma] (lemmaCOLONtightness_Gamma3.east) -- (propCOLONminS_attained.west);
    \draw[arrow, lemma] (lemmaCOLONtightness_Gamma3.east) -- (propCOLONW_nu_inf_attained.west);
    \draw[arrow, lemma] (lemmaCOLONtightness_Gamma3.east) -- (propCOLONW_nu_lsc.west);
  \end{scope}
  \begin{scope}[transparency group, opacity=0.5773502691896258]
    \draw[arrow, prop] (propCOLONW_nu_inf_attained.east) -- (propCOLONminS_attained.west);
    \draw[arrow, prop] (propCOLONW_nu_inf_attained.east) -- (propCOLONS_theta_semi_metric.west);
    \draw[arrow, prop] (propCOLONW_nu_inf_attained.east) -- (propCOLONW_nu_lsc.west);
  \end{scope}
  \begin{scope}[transparency group, opacity=1.0]
    \draw[arrow, prop] (propCOLONW_nu_lsc.east) -- (propCOLONPStheta_lsc.west);
  \end{scope}
  \begin{scope}[transparency group, opacity=0.7071067811865475]
    \draw[arrow, prop] (propCOLONwass_mean_geodesic.east) -- (lemmaCOLONplan_middle.west);
    \draw[arrow, prop] (propCOLONwass_mean_geodesic.east) -- (lemmaCOLONS_theta_well_def.west);
  \end{scope}
  \begin{scope}[transparency group, opacity=1.0]
    \draw[arrow, theorem] (thmCOLONW_nu_disintegration.east) -- (corCOLONS_theta_SWGG_as.west);
  \end{scope}
  \begin{scope}[transparency group, opacity=1.0]
    \draw[arrow, definition] (defCOLONS_theta.east) -- (corCOLONS_theta_SWGG_as.west);
  \end{scope}
  \begin{scope}[transparency group, opacity=0.5773502691896258]
    \draw[arrow, prop] (propCOLONsemi_1D_ot.east) -- (lemmaCOLONplan_projected_middle.west);
    \draw[arrow, prop] (propCOLONsemi_1D_ot.east) -- (corCOLONS_theta_SWGG_as.west);
    \draw[arrow, prop] (propCOLONsemi_1D_ot.east) -- (thmCOLONS_theta_equals_W_theta.west);
  \end{scope}
  \begin{scope}[transparency group, opacity=1.0]
    \draw[arrow, prop] (propCOLONS_theta_semi_metric.east) -- (propCOLONS_theta_metric_atomless_projections.west);
  \end{scope}
  \begin{scope}[transparency group, opacity=1.0]
    \draw[arrow, prop] (propCOLONPStheta_lsc.east) -- (propCOLONminS_attained.west);
  \end{scope}
  \begin{scope}[transparency group, opacity=1.0]
    \draw[arrow, lemma] (lemmaCOLONplan_middle.east) -- (lemmaCOLONplan_projected_middle.west);
  \end{scope}
  \begin{scope}[transparency group, opacity=1.0]
    \draw[arrow, lemma] (lemmaCOLONplan_projected_middle.east) -- (propCOLONS_theta_smaller_lift.west);
  \end{scope}
  \begin{scope}[transparency group, opacity=1.0]
    \draw[arrow, prop] (propCOLONS_theta_smaller_lift.east) -- (thmCOLONS_theta_equals_W_theta.west);
  \end{scope}
  \begin{scope}[transparency group, opacity=0.7071067811865475]
    \draw[arrow, theorem] (thmCOLONS_theta_equals_W_theta.east) -- (propCOLONcondition_minS_equals_W2.west);
    \draw[arrow, theorem] (thmCOLONS_theta_equals_W_theta.east) -- (propCOLONS_theta_metric_atomless_projections.west);
  \end{scope}
  \begin{scope}[transparency group, opacity=1.0]
    \draw[arrow, lemma] (lemmaCOLON1d_3plan_bimarginals.east) -- (propCOLONS_theta_metric_atomless_projections.west);
  \end{scope}
  \begin{scope}[transparency group, opacity=1.0]
    \draw[arrow, lemma] (lemmaCOLONCV_theta_monge_invariant_permutation.east) -- (thmCOLONCW_theta_point_clouds.west);
  \end{scope}
  \begin{scope}[transparency group, opacity=0.4472135954999579]
    \draw[arrow, assumption] (assCOLONidentity_sorts_projections.east) -- (thmCOLONextr_U_cap_PthetaXY.west);
    \draw[arrow, assumption] (assCOLONidentity_sorts_projections.east) -- (lemmaCOLONcondition_Psigmatau_in_Ptheta.west);
    \draw[arrow, assumption] (assCOLONidentity_sorts_projections.east) -- (propCOLONCWtheta_discrete_v1.west);
    \draw[arrow, assumption] (assCOLONidentity_sorts_projections.east) -- (thmCOLONCW_theta_point_clouds.west);
    \draw[arrow, assumption] (assCOLONidentity_sorts_projections.east) -- (lemmaCOLONcharacterisation_Sigma_theta.west);
  \end{scope}
  \begin{scope}[transparency group, opacity=0.7071067811865475]
    \draw[arrow, lemma] (lemmaCOLONcharacterisation_Sigma_theta.east) -- (propCOLONCWtheta_discrete_v1.west);
    \draw[arrow, lemma] (lemmaCOLONcharacterisation_Sigma_theta.east) -- (lemmaCOLONcondition_Psigmatau_in_Ptheta.west);
  \end{scope}
  \begin{scope}[transparency group, opacity=1.0]
    \draw[arrow, prop] (propCOLONCWtheta_discrete_v1.east) -- (thmCOLONCW_theta_point_clouds.west);
  \end{scope}
  \begin{scope}[transparency group, opacity=1.0]
    \draw[arrow, lemma] (lemmaCOLONcondition_Psigmatau_in_Ptheta.east) -- (thmCOLONextr_U_cap_PthetaXY.west);
  \end{scope}
  \begin{scope}[transparency group, opacity=1.0]
    \draw[arrow, lemma] (lemmaCOLONextract_cycle_P.east) -- (thmCOLONextr_U_cap_PthetaXY.west);
  \end{scope}
  \begin{scope}[transparency group, opacity=1.0]
    \draw[arrow, lemma] (lemmaCOLONsplit_multicycle.east) -- (thmCOLONextr_U_cap_PthetaXY.west);
  \end{scope}
  \begin{scope}[transparency group, opacity=1.0]
    \draw[arrow, lemma] (lemmaCOLONcycle_ollP.east) -- (thmCOLONextr_U_cap_PthetaXY.west);
  \end{scope}
  \begin{scope}[transparency group, opacity=1.0]
    \draw[arrow, theorem] (thmCOLONextr_U_cap_PthetaXY.east) -- (thmCOLONCW_theta_point_clouds.west);
  \end{scope}
  \begin{scope}[transparency group, opacity=1.0]
    \draw[arrow, theorem] (thmCOLONCW_theta_point_clouds.east) -- (propCOLONcondition_minS_equals_W2.west);
  \end{scope}
  \begin{scope}[transparency group, opacity=1.0]
    \draw[arrow, prop] (propCOLONelementary_props_lifted_plan.east) -- (propCOLONlifted_cost_almost_distance.west);
  \end{scope}
  \begin{scope}[transparency group, opacity=1.0]
    \draw[arrow, prop] (propCOLONlifted_cost_almost_distance.east) -- (corCOLONexpected_sliced_almost_distance.west);
  \end{scope}
  \begin{scope}[transparency group, opacity=1.0]
    \draw[arrow, prop] (propCOLONES_theta_discrete_as_distance.east) -- (corCOLONexpected_sliced_distance_countably_discrete.west);
  \end{scope}
  \begin{scope}[transparency group, opacity=1.0]
    \draw[arrow, corollary] (corCOLONexpected_sliced_almost_distance.east) -- (corCOLONexpected_sliced_distance_countably_discrete.west);
  \end{scope}
  \begin{scope}[transparency group, opacity=0.7071067811865475]
    \draw[arrow, lemma] (lemmaCOLONmidpoints.east) -- (propCOLONwass_mean_geodesic.west);
    \draw[arrow, lemma] (lemmaCOLONmidpoints.east) -- (propCOLONPStheta_lsc.west);
  \end{scope}
\end{pgfonlayer}
\end{tikzpicture}\end{center}